\documentclass{amsart}
\usepackage{mathrsfs}
\usepackage{amsfonts}
\usepackage{graphicx}
\usepackage{amsmath}
\usepackage{amssymb}
\usepackage{bm}
\usepackage{color}
\usepackage{mathdots}

\usepackage[scr=boondoxo]{mathalfa}

\usepackage[linktocpage=true]{hyperref}
\hypersetup{colorlinks,linkcolor=blue,urlcolor=cyan,citecolor=blue,
pdftitle={Twisted Yangians of split types BI, CI, DI and Drinfeld type current relations},
pdfauthor={Vidas Regelskis}}

\newtheorem{theorem}{Theorem}[section]
\newtheorem{thm}[theorem]{Theorem}

\newtheorem{crl}[theorem]{Corollary}

\newtheorem{lemma}[theorem]{Lemma}
\newtheorem{prop}[theorem]{Proposition}

\numberwithin{equation}{section}
	
\theoremstyle{remark}
\newtheorem{remark}[theorem]{Remark}

\newcommand{\nc}{\newcommand}

\nc{\ot}{\otimes}
\nc{\op}{\oplus}
\nc{\ol}{\overline}

\nc{\mc}{\mathcal}
\nc{\ms}{\mathsf}
\nc{\mf}{\mathfrak}
\nc{\mb}{\mathbf}
\nc{\bb}{\mathbb}
\nc{\mr}{\mathscr}

\nc{\al}{\alpha}
\nc{\bet}{\beta}
\nc{\eps}{\epsilon}
\nc{\del}{\delta}
\nc{\ga}{\gamma}
\nc{\Ga}{\Gamma}
\nc{\ka}{\kappa}
\nc{\vka}{\varkappa}
\nc{\la}{\lambda}
\nc{\om}{\omega}
\nc{\si}{\sigma}
\nc{\Ups}{\upsilon}
\nc{\vphi}{\varphi}

\nc{\id}{\operatorname{id}}
\nc{\gr}{\operatorname{gr}}
\nc{\msf}[1]{\mathsf{#1}}
\nc{\diag}{\operatorname{diag}}

\nc{\ud}{\underline}
\nc{\tl}{\tilde}
\nc{\wt}{\widetilde}
\nc{\wh}{\widehat}

\nc{\C}{\mathbb{C}}
\nc{\N}{\mathbb{N}}
\nc{\Z}{\mathbb{Z}}

\nc\mfgl{\mathfrak{gl}}
\nc\mfg{\mathfrak{g}}
\nc\mfsl{\mathfrak{sl}}
\nc\mfsp{\mathfrak{sp}}
\nc\mfo{\mathfrak{o}}

\nc{\lan}{\langle}
\nc{\ran}{\rangle}

\nc{\tw}{{\sf tw}}

\nc{\equ}[1]{\begin{equation}#1\end{equation}}
\nc{\eqa}[1]{\begin{equation}\begin{alignedat}{50}#1\end{alignedat}\end{equation}}
\nc{\eqn}[1]{\begin{equation*}\begin{alignedat}{50}#1\end{alignedat}\end{equation*}}
\nc{\eqg}[1]{\begin{equation}\begin{gathered}#1\end{gathered}\end{equation}}

\nc{\ali}[1]{\begin{alignat}{50}#1\end{alignat}}
\nc{\als}[1]{\begin{subequations}\begin{alignat}{50}#1\end{alignat}\end{subequations}}
\nc{\aln}[1]{\begin{alignat*}{50}#1\end{alignat*}}

\nc{\gat}[1]{\begin{gather}#1\end{gather}}
\nc{\gas}[1]{\begin{subequations}\begin{gather}#1\end{gather}\end{subequations}}
\nc{\gan}[1]{\begin{gather*}#1\end{gather*}}

\nc\el{\nonumber\\}
\nc\nn{\nonumber}

\nc{\tx}[1]{\qu\text{#1}\qu}

\nc{\qu}{\quad}
\nc{\qq}{\qquad}

\nc{\iso}{\stackrel{\sim}{\longrightarrow}}
\nc{\into}{\hookrightarrow}
\nc{\onto}{\twoheadrightarrow}

\DeclareMathOperator{\tr}{tr}
\DeclareMathOperator{\End}{End}
\DeclareMathOperator{\Sym}{Sym}

\renewcommand{\,}{\kern 0.1em} 

\begin{document}

\title[Twisted Yangians of split types BI, CI, DI]{Twisted Yangians of split types BI, CI, DI \\ and Drinfeld type current relations}

\author{Vidas Regelskis}
\address{Department of Physics, Astronomy and Mathematics, University of Hertfordshire, 
Hatfield AL10 9AB, UK and
Institute of Theoretical Physics and Astronomy, Vilnius University, Saul\.etekio av.~3, Vilnius 10257, Lithuania.}
\email{vidas.regelskis@gmail.com}


\begin{abstract}
We study twisted Yangians associated with the split symmetric pairs of types BI, CI and DI. We introduce a new presentation of these algebras, which we call the transposed presentation, governed by a twisted reflection equation that interacts naturally with the Gaussian decomposition of the generating matrix. Working entirely within the $R$-matrix presentation, we derive Drinfeld-type current presentations of the special twisted Yangian $SY^{\tw}(\mfg_N)$ and of the extended twisted Yangian $X^{\tw}(\mfg_N)$, in which the Serre relations are stated in a closed current form. Extracting coefficients recovers the Drinfeld presentation of Lu, Wang and Zhang, in the normalisation of Lu. As a consequence, we establish an explicit isomorphism between the $R$-matrix and Drinfeld presentations conjectured by Lu, Wang and Zhang. As a byproduct, we obtain a tensor product decomposition of $X^{\tw}(\mfg_N)$ into the twisted Yangian in the Drinfeld presentation and a polynomial ring in countably many central variables. We also obtain Poincar\'e--Birkhoff--Witt bases in the Drinfeld generators and the coideal coproduct on the low Drinfeld modes at all but the last node.
\end{abstract}

\maketitle
\setcounter{tocdepth}{1}
\tableofcontents
\setcounter{tocdepth}{2}

\setlength{\medmuskip}{2.0mu plus 1.0mu minus 1.0mu} 


\section{Introduction} \label{sec:into}

The Yangian associated with a simple Lie algebra admits three essentially different presentations: the $R$-matrix (or RTT) presentation, built from a solution of the quantum Yang--Baxter equation, Drinfeld's original $J$-presentation, generated by the underlying Lie algebra together with a set of degree-one elements, and Drinfeld's new (current) presentation, given in terms of generating series of the Chevalley-type generators. In type A the $R$-matrix presentation goes back to the St.~Petersburg school, see \cite{MNO96, Mo07}, and an explicit isomorphism with the current presentation is obtained through a Gaussian decomposition of the generating matrix \cite{Mo07}. In types B, C and D the $R$-matrix presentation was developed in \cite{AMR06}, and an explicit isomorphism was constructed by Jing, Liu and Molev \cite{JLM18}, who read off the Drinfeld-type relations from the matrix identities given by the same method. 
The equivalence of these three presentations for the orthogonal and symplectic Yangians was established by Guay, Wendlandt and the author \cite{GRW19a}.
The present paper develops a counterpart of this circle of ideas for twisted Yangians.

Twisted Yangians are coideal subalgebras of Yangians attached to symmetric pairs. Building on Sklyanin's reflection equation \cite{Sk88}, they were introduced in the $R$-matrix presentation by Olshanski \cite{Ol92} for the symmetric pairs of types AI and AII, by Molev and Ragoucy \cite{MR02} for type AIII, and by Guay and the author \cite{GR16} for the remaining classical types B, C and D, the low-rank cases having been treated in \cite{GRW16}. Their representation theory was developed in \cite{MNO96, Mo07, GRW17, GRW19b}, and for general symmetric pairs a presentation in Drinfeld's original $J$-form is available through \cite{Ma02, BR17}. From a more recent viewpoint, twisted Yangians arise as rational degenerations of affine $\imath$quantum groups, the coideal subalgebras of affine quantum groups attached to quantum symmetric pairs; this places them within the broader programme of Drinfeld-type presentations for such coideal subalgebras, whose affine $\imath$-counterparts were obtained by Lu and Wang \cite{LW21} in split ADE type and by Zhang \cite{Zh22} in split BCFG type.

A Drinfeld-type presentation of twisted Yangians themselves has emerged only recently. For type AI it was obtained by Lu, Wang and Zhang \cite{LWZ25a} through a Gaussian decomposition of the $R$-matrix presentation of Olshanski's twisted Yangians. In parallel, the same authors \cite{LWZ25b} produced, for all split symmetric pairs, a family of current algebras by degenerating the Drinfeld presentations of the corresponding affine $\imath$quantum groups, and conjectured these to be isomorphic to the $R$-matrix twisted Yangians; the quasi-split types were subsequently treated by Lu and Zhang \cite{LZ25}. Lu \cite{Lu26} then gave a minimalistic presentation and, outside type $G_2$, identified these current algebras with the twisted Yangians in Drinfeld's $J$-presentation, thereby realising them as coideal subalgebras of Yangians. These developments have already found applications, for instance to shifted twisted Yangians and finite W-algebras of classical type \cite{LPTTW25}.

Two points motivate the present work. First, outside type AI and quasi-split AIII the Drinfeld-type relations have been obtained by degeneration from the $q$-loop setting rather than directly inside the $R$-matrix presentation of the twisted Yangian, and an explicit isomorphism between the two presentations, conjectured in \cite{LWZ25b}, remained to be constructed in types B, C and D, cf.~\cite[\S1.2]{Lu26}. 
Second, the Drinfeld-type relations had not been cast in a uniform \emph{current} form. In the minimalistic presentation of \cite{Lu26} they appear in coefficient form, while in the generating-function presentation of \cite[\S4.3]{LWZ25b} the relation between the series $h_i(u)$ and $b_j(v)$ still carries explicit low-mode correction terms, and the two Serre relations --- the most intricate part of the presentation --- are written in quite different forms.

We address both points. Working entirely within the $R$-matrix presentation, we apply a Gaussian decomposition of the generating matrix, in the spirit of \cite{JLM18, LWZ25a}, and derive the Drinfeld-type relations --- including the Serre relations in closed current form --- directly for the split symmetric pairs \eqref{split-pairs} of types BI, CI and DI. We also reformulate current relations obtained in \cite[Remark~4.12]{LWZ25a}. Our relations may be compared directly with the generating-function presentation of \cite[\S4.3]{LWZ25b}: our \eqref{hb} corresponds to their (4.13), and our Serre relations \eqref{S1} and \eqref{S2} to their (4.16) and (4.17). Our relation \eqref{hb} needs no low-mode correction terms, and both Serre relations \eqref{S1} and \eqref{S2} are organised around a single building block $x_{ij}(v,u)$, defined in \eqref{xij}, together with the symmetrising bracket $\{f(u)\}^u = f(u)+f(-u)$ of \eqref{u-symm}, which recurs throughout \eqref{hb}, \eqref{S1} and \eqref{S2}.

A key technical device is a new presentation of the twisted Yangians of types B, C and D, which we call the \emph{transposed presentation}. It is governed by a \emph{twisted} reflection equation \eqref{RE}, in contrast with the (non-twisted) reflection equation used in \cite{GR16} (an equivalent twisted form, written with the transposition of \cite{GR16} with respect to the bilinear form, is noted in \cite[Rem.~4.1]{GR16}); the two presentations are related by the matrix substitution $S(u) \mapsto A\,S(u)\,J\,A^t$. The twisted reflection equation interacts more transparently with the Gaussian decomposition and is what makes the direct computation of the Drinfeld-type relations feasible.

Our main results are Theorem \ref{T:iso} and Corollaries \ref{C:LWZ} and \ref{C:iso-ext}. Theorem \ref{T:iso} establishes a Drinfeld-type current presentation of the special twisted Yangian $SY^{\tw}(\mfg_N)$ for the split symmetric pairs \eqref{split-pairs}; extracting coefficients from its defining identities recovers the Drinfeld presentation of \cite[Def.~2.7]{Lu26}, and hence, upon specialising $\hbar$ and rescaling the generators, that of \cite[Def.~4.1]{LWZ25b}, while restricting to the lowest modes yields the minimalistic presentation of \cite[Thm.~3.1]{Lu26}. Corollary \ref{C:LWZ} identifies $SY^{\tw}(\mfg_N)$ explicitly with the twisted Yangian $\mc{Y}^{\imath}(\mfg_N)$ in the Drinfeld presentation of \cite{LWZ25b}, thereby confirming, in the split types BI, CI and DI, the isomorphism between the $R$-matrix and Drinfeld presentations conjectured therein. Corollary~\ref{C:iso-ext} gives the analogous current presentation of the extended twisted Yangian $X^{\tw}(\mfg_N)$; a byproduct of its proof is the tensor product decomposition $X^{\tw}(\mfg_N) \cong \mc{Y}^{\imath}(\mfg_N) \ot \C[\ms{t}_0, \ms{t}_1, \ldots]$ with a polynomial ring in countably many central variables, refining the decomposition \eqref{X=SYcz} obtained in \cite{GR16}. The relations themselves are established by largely structural arguments, via a reduction to low-rank subalgebras and quasi-determinantal embeddings, while injectivity rests on the Drinfeld presentation of $\mc{Y}^{\imath}(\mfg_N)$ \cite{LWZ25b}, in the form of \cite[Def.~2.7]{Lu26}, together with its filtration and Poincar\'e--Birkhoff--Witt theorem \cite{LWZ25b}, through a comparison of associated graded algebras. The isomorphism $\mc{Y}^{\imath}(\mfg_N)\cong SY^{\tw}(\mfg_N)$ of Corollary \ref{C:LWZ}, the tensor product decomposition and the bases below use the Serre relations only in their finite form; the current Serre relations \eqref{S1}--\eqref{S2}, which rest on Proposition \ref{P:Serre}, enter only the presentations of Theorem \ref{T:iso} and Corollary \ref{C:iso-ext}. Two further consequences are recorded in \S\ref{sec:ext} and \S\ref{sec:coideal}: Poincar\'e--Birkhoff--Witt bases of $SY^{\tw}(\mfg_N)$ and $X^{\tw}(\mfg_N)$ in the Drinfeld generators (Corollary \ref{C:PBW}), and the coideal coproduct on the low Drinfeld modes at the nodes $i<n$, including an explicit form of the cross-term entering $\Delta(h_{i,1})$ (Proposition \ref{P:coideal}). We expect these arguments to serve as a blueprint for obtaining Drinfeld current presentations of twisted Yangians attached to other symmetric pairs, as well as of twisted super-Yangians; a Gaussian-decomposition treatment of the latter in quasi-split type A has recently been carried out by Lu \cite{Lu27}.

The paper is organised as follows. In \S\ref{sec:tY} we recall the twisted Yangians of types B, C and D, introduce their transposed presentation, and construct a family of quasi-determinantal embeddings. \S\ref{sec:gauss} sets up the Gaussian decomposition of the generating matrix and the associated reduced generating matrices. \S\ref{sec:low} treats the low-rank cases, which provide the base of the inductive arguments, and \S\ref{sec:gen} contains the general case, culminating in \S\ref{sec:Dr}, where the Drinfeld-type relations are assembled and the main results are proved; \S\ref{sec:ext} treats the extended twisted Yangian and the linear bases, and \S\ref{sec:coideal} the coideal structure. Three appendices collect the fundamental representation of $X^{\tw}(\mfg_N)$ (Appendix \ref{app:Rep}), and the lengthy auxiliary relations and evaluated commutators (Appendices \ref{app:Aux} and \ref{app:Comms}) used in the proof of Proposition \ref{P:Serre}.


\section{Twisted Yangians} \label{sec:tY}


\subsection{Lie algebras}

Choose $N \ge 2$ and write $N=2n$ or $N=2n+1$ according to whether $N$ is even or odd. Let $E_{ij}\in\End(\C^N)$ and $e_i\in\C^N$ denote the standard matrix units and the standard basis vectors of $\C^N$ so that $E_{ij} e_k = \del_{jk} e_i$; here $1\le i,j,k \le N$. We will denote linear functionals dual to $e_i$ by $e^*_i$. 

Matrix units $E_{ij}$ satisfy defining relations of the $\mfgl_N$ Lie algebra:
\equ{
[E_{ij}, E_{kl}] = \del_{jk} E_{il} - \del_{il} E_{kj}. \label{[E,E]}
}
The orthogonal Lie algebra $\mfo_N$ and the symplectic Lie algebra $\mfsp_N$ can be regarded as subalgebras of $\mfgl_N$ as follows. For any $1 \le i,j\le N$ set $\theta_{ij}:=\theta_i\theta_j$ with $\theta_i:=1$ in the orthogonal case and $\theta_i:=\del_{i>N/2}-\del_{i\le N/2}$ in the symplectic case.
Introduce elements $F_{ij} := E_{ij} - \theta_{ij} E_{\ol{\jmath}\,\ol{\imath}}$ with $\ol{\imath} := N-i+1$ and $\ol{\jmath} := N-j+1$. These elements satisfy relations 
\gat{ 
\label{[F,F]}
[F_{ij},F_{kl}] = \del_{jk} F_{il} - \del_{il} F_{kj} + \theta_{ij} (\del_{j \bar l} F_{k \ol{\imath}} - \del_{i \bar k} F_{\ol{\jmath}\, l}) ,
\\
\label{F+F=0}
F_{ij} + \theta_{ij} F_{\ol{\jmath}\,\ol{\imath}}=0,
}
which are the defining relations of $\mfo_N$ and $\mfsp_N$; we will denote both algebras by $\mfg_N$ and restrict to $N\ge 3$ in the orthogonal case. We set $\theta := \theta_{1}$ so that $\theta = 1$ in the orthogonal case and $\theta=-1$ in the symplectic case.


\subsection{Symmetric pairs}

We consider symmetric pairs $(\mfg_N, \mfg_N^\rho)$, where $\rho$ is an involution of $\mfg_N$ defined by $\rho(X) = -GX^t G^{-1}$ for a particular matrix $G$ and $\mfg_N^\rho$ denotes the $\rho$-fixed subalgebra of $\mfg_N$. Our choice of matrix $G$ is as follows.
\begin{itemize}\itemsep0.2em

\item BI:\hspace{0.8em} $(\mfg_N, \mfg_N^\rho) = (\mfo_{2n+1}, \mfo_{p} \op \mfo_{q})$ such that $p< q$. Then $G := \sum_{i=1}^{p} (E_{ii} + E_{\bar \imath \bar \imath}) + \sum_{i=p+1}^{N-p} E_{i \bar\imath}$.

\item DI:\hspace{0.78em} $(\mfg_N, \mfg_N^\rho) = (\mfo_{2n}, \mfo_{p} \op \mfo_{q})$ such that $p\le q$. We choose $G$ to be the same as above. \vspace{0.2em}

\item DIII: $(\mfg_N, \mfg_N^\rho) = (\mfo_{2n}, \mfgl_{n})$. Then $G := \sum_{i=1}^{n} (E_{i\bar\imath} - E_{\bar \imath i})$. \vspace{0.2em}

\item CI:\hspace{0.8em} $(\mfg_N, \mfg_N^\rho) = (\mfsp_{2n}, \mfgl_n)$. Then $G := \sum_{i=1}^{n} (E_{ii} - E_{\bar\imath\bar\imath})$.

\item CII:\hspace{0.49em} $(\mfg_N, \mfg_N^\rho) = (\mfsp_{2n}, \mfsp_{p} \op \mfsp_{q})$ such that $p\le q$. Then $G := \sum_{i=1}^{\frac{p}{2}} (E_{i \bar\imath} - E_{\bar \imath i}) - \sum_{i=\frac{p}{2}+1}^n (E_{i \bar\imath} - E_{\bar \imath i})$.

\end{itemize}
When $G$ is a diagonal matrix, $\rho$ is the Chevalley involution. The corresponding pairs
\equ{
\text{BI}:\;(\mf{o}_{2n+1}, \mf{o}_{n} \op \mf{o}_{n+1}) , \qq
\text{DI}:\;(\mf{o}_{2n}, \mf{o}_n \op \mf{o}_n) , \qq
\text{CI}:\;(\mf{sp}_{2n}, \mf{gl}_n) 
\label{split-pairs}
}
are split symmetric pairs and will be the primary focus of this paper. We will use the symbol $\om$ to denote the Chevalley involution.

\smallskip

Set $(\pm) := -1$ when $\mfg_N^\rho= \mfgl_n$ and $(\pm):=1$ otherwise. Then $G F_{ji} G^{-1} =(\pm)\, \theta \sum_{a,b=1}^N g_{ib}\, F_{ab}\, g_{aj}$, where $g_{ab}$ are matrix entries of $G$, and $G^2 = (\pm)\,\theta\,I$.
We define elements
\ali{
F^{\rho}_{ij} := \sum_{1\le a\le N} (F_{ia} - G F_{ai} G^{-1}) g_{aj} = \sum_{1\le a\le N} (F_{ia} g_{aj} - g_{ia} F_{ja} )
}
satisfying commutation relations
\ali{
[F^\rho_{ij}, F^\rho_{kl}] &= g_{kj} F^\rho_{il} - g_{il} F^\rho_{kj} - g_{ik} F^\rho_{jl} + g_{lj} F^\rho_{ki} \nn \\ 
& \qu + \sum_{1\le a\le N} \Big( \theta_a \theta_{\bar{k}}\, (\del_{\bar\jmath k}\, g_{ia} - \del_{\bar\imath k}\, g_{aj} ) F^\rho_{\bar{a} l} + \theta_a \theta_{\bar{l}} \, (\del_{\bar\jmath l}\, g_{ia} - \del_{\bar\imath l}\, g_{aj} ) F^\rho_{k\bar{a}} \Big) \label{grho-rels}
}
and symmetry relations
\equ{
F^{\rho}_{ij} + (\pm)\,\theta\,F^\rho_{ji} = 0 , \qq
\sum_{1\le a\le N} \theta_{aj} (F^{\rho}_{ia} g_{\bar a\bar\jmath} - g_{ia}F^\rho_{\bar a\bar\jmath}) = 0 .
 \label{grho-symm}
}
Relations \eqref{grho-rels} and \eqref{grho-symm} are defining relations of the $\rho$-fixed subalgebra $\mfg^\rho_N$.


\subsection{Matrix operators}

Introduce matrix operators
\equ{
\label{IPQ}
I := \sum_{1\le i,j\le N} E_{ii} \ot E_{jj} , \qq
P := \sum_{1\le i,j\le N} E_{ij} \ot E_{ji} , \qq
Q := \sum_{1\le i,j\le N} \theta_{ij} E_{\bar{\jmath}\,\bar{\imath}} \ot E_{ji} , \qu
}
and a matrix-valued function 
\equ{
R(u) := I - u^{-1} P - (\ka-u)^{-1}Q , 
\label{R(u)}
}
where $\ka = N/2 - \theta$ and $u$ is a formal parameter. It is a solution of the quantum Yang-Baxter equation in $\C^{N}\ot\C^{N}\ot \C^{N}$:
\equ{
R_{12}(u-v)\,R_{13}(u-z)\,R_{23}(v-z) = R_{23}(v-z)\,R_{13}(u-z)\,R_{12}(u-v). \label{YBE}
}
Here, and throughout the manuscript, the subscript notation indicates the tensor spaces the matrix operators act on.

Let $t$ denote matrix transposition, $(E_{ij})^{t} := E_{ji}$, and let $R'(u)$ denote the partially-transposed $R(u)$, that is $R'(u) = R^{t_1} (u)= R^{t_2} (u)$ (the second equality follows from \eqref{IPQ}) and let $J := \sum_{i} \theta_i E_{i\bar{\imath}}$. Then
\equ{
R'(\ka-u) = J_1 R(u) J^{-1}_1 = J_2 R(u) J^{-1}_2 . \label{R=JRtJ}
}
Let $A \in \End(\C^N)$ be an orthogonal matrix such that $AJA^t = \pm J$. Then
\equ{
A_1 A_2 R(u) A^t_1 A^t_2 = R(u) , \label{AARAA}
}
and the same holds for $R'(u)$ in place of $R(u)$.

\smallskip

\enlargethispage{0.5em}

Define a matrix-valued function 
\ali{
G(u) := 
\begin{cases} 
\dfrac{(p-q)J + 4 u\,G}{p-q+4u} & \text{for types BI, DI, CII,} \\
G & \text{otherwise}.
\end{cases}
}
It is a solution of the twisted reflection equation in $\C^N \ot \C^N$:
\equ{
R(u-v)\,G_1(u)\,R'(\ka-u-v)\,G_2(v) = G_2(v)\,R'(\ka-u-v)\,G_1(u)\,R(u-v)
}
and the symmetry relation in $\C^N$:
\equ{
G^t(u) = (\pm)\,\theta\, G(\ka-u) + \frac{G(u)-G(\ka-u)}{2u-\ka} + \frac{\tr(G(u)J)\,(G(\ka-u)-J)}{2u-2\ka}
}
and satisfies the unitarity relation
\equ{
G(u)\,J\,G(-u)\,J = I
}
where $I \in \End(\C^N)$ is the identity matrix. Moreover, $\tr (G(u) J) = \theta\frac{(p-q)(p+q-4u)}{p-q+4u}$ for types BI, DI, CII and $\tr(G(u)J) = 0$ otherwise.


\subsection{Extended twisted Yangian}

The extended twisted Yangian $X^\tw(\mfg_N,G)$ is an associative unital $\C$-algebra with countably many generators $s_{ij}^{(r)}$ with $1\le i, j \le N$ and $r\in \N$ satisfying a family of relations that can be presented as follows. Introduce formal series
\equ{
s_{ij}(u) = g_{ij} + \sum_{r\ge 1} s^{(r)}_{ij} u^{-r} \in X^\tw(\mfg_N,G) [[u^{-1}]]
\label{sij(u)}
}
and combine them into the generating matrix $S(u) = \sum_{i,j=1}^N E_{ij} \ot s_{ij}(u)$. Defining relations of $X^{\tw}(\mfg_N,G)$ in its {\it transposed presentation} are then given by the twisted reflection equation
\equ{
R(u-v)\,S_1(u)\,R'(\ka-u-v)\,S_2(v) = S_2(v)\,R'(\ka-u-v)\,S_1(u)\,R(u-v) . \label{RE}
}
In particular,
\ali{ \label{[s,s]}
[\,s_{ij}(u),s_{kl}(v)] &= \frac{1}{u-v}\Big(s_{kj}(u)\,s_{il}(v)-s_{kj}(v)\,s_{il}(u)\Big) \el
{}&\qu + \frac{1}{u+v} \sum_{1\le a\le N} \Big(\delta_{k\bar{\jmath}}\,\theta_{aj}\, s_{ia}(u)\,s_{\bar{a}l}(v)-
\delta_{l\bar{\imath}}\,\theta_{ai}\, s_{k\bar{a}}(v)\,s_{aj}(u)\Big) \el
{}&\qu - \frac{1}{u^2-v^2} \sum_{1\le a\le N} \delta_{\bar{\imath}j} \Big(\theta_{aj}\,s_{ka}(u)\,s_{\bar{a}l}(v) - \theta_{ai}\, s_{k\bar{a}}(v)\,s_{al}(u)\Big) \el	
{}&\qu - \frac{1}{u-v-\ka} \sum_{1\le a\le N} \Big( \delta_{k\bar{\imath}}\,\theta_{ai}\, s_{aj}(u)\, s_{\bar{a}l}(v) - \delta_{l\bar{\jmath}} \,\theta_{aj}\, s_{k\bar{a}}(v)\, s_{ia}(u) \Big) \el
{}&\qu - \frac{1}{u+v-\ka}\, \Big( s_{ik}(u)\, s_{jl}(v) - s_{ki}(v)\, s_{lj}(u) \Big) \nn \el
{}&\qu + \frac{1}{(u+v) (u-v-\ka)} \sum_{1\le a\le N} \Big( \delta_{k\bar{\imath}}\,\theta_{ai}\,s_{ja}(u)\, s_{\bar{a}l}(v) - \delta_{l\bar{\jmath}}\,\theta_{aj}\,s_{k\bar{a}}(v)\, s_{ai}(u) \Big) \el
{}&\qu + \frac{1}{(u-v) (u+v-\ka)}\, \Big( s_{ki}(u)\, s_{jl}(v) - s_{ki}(v)\, s_{jl}(u) \Big) \el
{}&\qu - \frac{1}{(u-v-\ka) (u+v-\ka)} \sum_{1\le a\le N} \Big( \delta_{k\bar{\imath}}\,\theta_{ak} s_{\bar{a} a}(u)\, s_{jl}(v) - \delta_{l\bar{\jmath}}\,\theta_{aj} \,s_{ki}(v)\, s_{\bar{a} a}(u) \Big) .
}
Relations \eqref{R=JRtJ} and \eqref{AARAA} imply that the mapping $S(u) \mapsto A S(u)J A^t$ for a suitable choice of matrix $A$ defines an isomorphism of $X^\tw(\mfg_N,G)$ with the extended reflection algebra of \cite[\S5]{GR16}, which is defined by the (non-twisted) reflection equation alone. The algebra called the extended twisted Yangian in \cite{GR16} corresponds to the quotient of $X^\tw(\mfg_N,G)$ by the symmetry relation \eqref{sym} below.


\subsection{Special twisted Yangian}

The special twisted Yangian $SY^\tw(\mfg_N,G)$ is obtained by taking the quotient of $X^\tw(\mfg_N,G)$ by the two-sided ideal generated by the symmetry relation (see \cite[\S4]{GR16})
\ali{ \label{sym}
S^t(u) = (\pm)\,\theta S(\ka-u) + \frac{S(u)-S(\ka-u)}{2u-\ka} + \frac{\tr(G(u)J)\,S(\ka-u) - \tr(S(u)J)J }{2u-2\ka} 
}
and the unitarity relation
\equ{
S(u)\,JS(-u)J = I. \label{SS=wI}
}
In terms of generating series, symmetry relation \eqref{sym} reads as
\ali{ \label{symij}
s_{ji}(u) = (\pm)\,\theta s_{ij}(\ka-u) + \frac{s_{ij}(u)-s_{ij}(\ka-u)}{2u-\ka} + \frac{\sum_{1\le a\le N} ( \theta_a\, g_{\bar{a}a}(u)\,s_{ij}(\ka-u) - \theta_{ia}\,\del_{i\bar{\jmath}}\, s_{\bar{a}a}(u)) }{2u-2\ka} 
}
and the unitarity relation \eqref{SS=wI} reads as
\equ{
\sum_{1\le a\le N} \theta_{aj}\, s_{i a}(u)\, s_{\bar a \bar\jmath}(-u) = \del_{ij} . \label{unitij}
}

The special twisted Yangian $SY^\tw(\mfg_N,G)$ can be viewed as a subalgebra of the Yangian $Y(\mfg_N)$ of \cite{AMR06} via the embedding (cf.~\cite[Thm.~3.1]{GR16})
\equ{
S(u) \mapsto T(u - \ka/2)\,G(u)\,T^t(\ka/2-u) . \label{S->TGT}
}
Here $T(u) = \sum_{i,j=1}^N E_{ij} \ot t_{ij}(u)$ is the generating matrix of $Y(\mfg_N)$, that is, the generating matrix of the extended Yangian $X(\mfg_N) \supset Y(\mfg_N)$ normalised by a central series ensuring that the image of $S(u)$ satisfies the unitarity relation \eqref{SS=wI}. Note that \eqref{S->TGT} involves the ordinary matrix transpose, with no sign factors $\theta_{ij}$, in contrast with the twisted transposition of \cite{GR16}. The matrix $G(u)$ reduces to the constant matrix $G$ for the split pairs of types CI and DI, while for split BI it retains its rational form.


\subsection{The associated graded algebra}

Introduce an ascending filtration on $SY^\tw(\mfg_N,G)$ by assigning $\deg s^{(r)}_{ij} = r-1$ for all $r\ge 1$. Denote by $\bar s^{(r)}_{ij}$ the image of $s^{(r)}_{ij}$ in the $(r\!-\!1)$-th component of the associated graded algebra $\gr SY^\tw(\mfg_N,G)$. Then
\ali{
[\bar{s}^{(r)}_{ij}, \bar{s}^{(q)}_{kl}] &= g_{kj} \bar{s}^{(r+q-1)}_{il} - g_{il} \bar{s}^{(r+q-1)}_{kj} + (-1)^r \Big( g_{ik} \bar{s}^{(r+q-1)}_{jl} - g_{lj} \bar{s}^{(r+q-1)}_{ki} \Big) \nn\\ 
& \qu + \sum_{1\le a\le N} \Big( \theta_{l\bar{a}}\big( \del_{\bar\jmath\, l}\, g_{ia} + (-1)^r \del_{\bar\imath l} \, g_{aj} \big) \bar{s}^{(r+q-1)}_{k\bar{a}} - \theta_{k\bar{a}} \big( \del_{\bar\imath k} \, g_{aj} + (-1)^r \del_{\bar\jmath k}\, g_{ia} \big) \bar{s}^{(r+q-1)}_{\bar{a}l} \Big) .
}
and
\equ{
\bar{s}_{ji}^{(r)} - (-1)^r (\pm)\,\theta\,\bar{s}_{ij}^{(r)} = 2\,\theta\, \ga^{(r-1)}_{ij}
\qu\text{where}\qu 
\ga^{(r)}_{ij} := \begin{cases} 
\del_{r0} \frac{q-p}{4} \big( g_{ij} - \theta_{i} \del_{i\bar\jmath}\big) & \text{for types BI, DI, CII}, \\
0 & \text{otherwise,} 
\end{cases} 
}
and
\equ{
\sum_{1\le a\le N} \theta_{aj} \Big(\bar{s}_{ia}^{(r)}g_{\bar a \bar\jmath} + (-1)^{r} g_{ia} \bar{s}_{\bar a\bar\jmath}^{(r)} \Big) = 0 .
}
Here we have used $g_{ij} = (\pm)\,\theta{}g_{ji}$ and $\sum_{a=1}^N \theta_a g_{\bar a a} = \tr(GJ) = \theta\,(q-p)$ for types BI, DI, CII and $\tr(GJ) = 0$ otherwise.

Let $U(\mfg_N[x])$ denote the universal enveloping algebra of the current Lie algebra $\mfg_N[x]$ generated by elements $F_{ij}\, x^r$ with $1\le i,j \le N$ and $r\ge 0$, and let $U(\mfg_N[x]^\rho)$ denote the subalgebra of $U(\mfg_N[x])$ generated by elements
\equ{
F^{(\rho,r)}_{ij} := \sum_{1\le a\le N} (F_{ia} g_{aj} - (-1)^r\,g_{ia} F_{ja})\,x^r . \label{Frr}
}
Then the mapping
\equ{
\bar{s}^{(r)}_{ij} \mapsto F^{(\rho,r-1)}_{ij} + \ga^{(r-1)}_{ij}
\label{SY-graded-iso}
}
defines an isomorphism $\gr SY^\tw(\mfg_N,G) \cong U(\mfg_N[x]^\rho)$, cf.~the proof of \cite[Thm.~4.1]{GR16}.

Let the associated graded algebra $\gr X^\tw(\mfg_N,G)$ of the extended twisted Yangian $X^\tw(\mfg_N,G)$ be defined analogously to that of $SY^\tw(\mfg_N,G)$. Applying \cite[Thm.~5.3]{GR16} and then \cite[Thm.~4.2]{GR16}, which remove the central series associated with the symmetry and with the unitarity relation, respectively, we may write $S(u) = \ms{y}(u)\,\wt{S}(u)$, where $\ms{y}(u) = 1 + \sum_{r\ge0} \ms{y}_r u^{-r-1}$ is the product of the two normalising central series and the coefficients of $\wt S(u)$ generate a subalgebra of $X^\tw(\mfg_N,G)$ isomorphic to $SY^\tw(\mfg_N,G)$ via $\wt S(u)\mapsto S(u)$. The passage from the central generators of \cite{GR16} to the elements $\ms{y}_r$ is triangular, so the latter are algebraically independent, and $X^\tw(\mfg_N,G)$ is the tensor product of the above subalgebra and the polynomial algebra in $\ms{y}_0,\ms{y}_1,\ldots$; this decomposition is compatible with the filtration by \cite[Cor.~5.3]{GR16}. Let $\zeta_r$ denote the image of $2\ms{y}_r$ in the $r$-th component of $\gr X^\tw(\mfg_N,G)$. Since $s^{(r)}_{ij} = \wt{s}^{(r)}_{ij} + g_{ij}\,\ms{y}_{r-1}$ modulo terms of degree less than $r-1$, \eqref{SY-graded-iso} implies that $\gr X^\tw(\mfg_N,G) \cong U(\mfg_N[x]^\rho) \ot \C[\zeta_0, \zeta_1, \ldots]$, where $\C[\zeta_0, \zeta_1, \ldots]$ is the algebra of polynomials in infinitely many variables. The isomorphism of algebras is given by the mapping
\equ{
\bar{s}^{(r)}_{ij} \mapsto F^{(\rho,r-1)}_{ij} + \ga^{(r-1)}_{ij} + \tfrac12 \, g_{ij}\, \zeta_{r-1}. \label{X-graded-iso}
}


\subsection{Linear basis} \label{sec:basis}

We adapt Theorem 3.2 and Corollary 5.3 in \cite{GR16}. Given any total ordering on the set of elements $s_{ij}^{(r)}$, a linear basis of $SY^\tw(\mfg_N,G)$ is provided by ordered monomials in the following elements:

\begin{itemize}

\item BI, DI: $s_{ij}^{(2r-1)}$ with $i< \bar\jmath$, $j\le p$ or $\bar\imath,j>p$ and $s_{ij}^{(2r)}$ with $i< \bar\jmath$, $j\le p$. \vspace{0.2em}

\item DIII: \hspace{0.6em} $s_{ij}^{(2r-1)}$ with $j\le n < i$ and $s_{ij}^{(2r)}$ with $i\le n$ or $j> n$. \vspace{0.2em}

\item CI: \hspace{1.4em} $s_{ij}^{(2r-1)}$, $s_{ij}^{(2r)}$ with $i\le \bar\jmath$. \vspace{0.2em}

\item CII:\hspace{1.32em} $s_{ij}^{(2r-1)}$ with $\bar\imath, j\!\le\! \tfrac{p}{2}$ or $\bar\imath, j \!>\! \tfrac{p}{2}$ or $i\!\le\! \tfrac{p}{2}$ or $\bar\jmath \!\le\! \tfrac{p}{2}$ and 
$s_{ij}^{(2r)}$ with $j\!\le\! \tfrac{p}{2} \!<\! i \!< \!\ol{\,\tfrac{p}{2}\,}\!$ or $\tfrac{p}{2} \!<\! j \!< \!\ol{\,\tfrac{p}{2}\,}\! \le i$. 

\end{itemize}
Here $r\ge1$. Moreover, $i>j$ for $s^{(2r-1)}_{ij}$ and $i\ge j$ for $s_{ij}^{(2r)}$ in all cases except CII and DIII, for which $i\ge j$ for $s^{(2r-1)}_{ij}$ and $i> j$ for $s_{ij}^{(2r)}$.

\enlargethispage{0.5em}

In the remaining parts of this paper we will focus on special and extended twisted Yangians for split symmetric pairs \eqref{split-pairs}, which we denote by $SY^\tw(\mfg_N)$ and $X^\tw(\mfg_N)$, respectively. A linear basis of $SY^\tw(\mfg_N)$ is given by ordered monomials in the elements 
\equ{
\label{SY-basis}
s^{(r)}_{ij} \;\text{ with } \; j < i < \bar\jmath + \tfrac{1-\theta}{2} \;\text{ and }\; s^{(2r)}_{ii} \;\text{ with }\; 1\le i \le n \;\text{ and }\; r\ge 1.
}
For $X^\tw(\mfg_N)$ a linear basis is given by ordered monomials in the elements 
\equ{
\label{X-basis}
s^{(r)}_{ij} \;\text{ with } \; j < i < \bar\jmath + \tfrac{1-\theta}{2} \;\text{ and }\; s^{(r)}_{11} \;\text{ and }\; s^{(2r)}_{ii} \;\text{ with }\; 1< i \le n+1 \;\text{ and }\; r\ge 1.
}


\subsection{Embeddings}

In this subsection we obtain a family of embeddings, for $1\le m < n-1$ when $\mfg_N=\mfo_{2n}$ and $1\le m < n$ otherwise,
\[
\psi^{(N)}_m : X^\tw(\mfg_{N-2m}) \to X^\tw(\mfg_N) .
\]
We will follow the approaches in \cite{JLM18} and \cite{LZ25} with suitable modifications.

Let $X=(x_{ab})$ be a square matrix over a unital ring, and let $X^{ij}$ be the submatrix obtained by deleting the $i$-th row and the $j$-th column. If $X^{ij}$ is invertible, the $(i,j)$-th {\it quasi-determinant} of $X$ is defined by the formula below and denoted graphically by the boxed notation (see~\cite[Def.~1.10.1]{Mo07}):
\equ{
|X|_{ij} := x_{ij} - r^j_i\,(X^{ij})^{-1}\,c^i_j = \begin{vmatrix}
x_{11} & \cdots & x_{1j} & \cdots & x_{1N} \\
\cdots & \cdots & \cdots & \cdots & \cdots \\
x_{i1} & \cdots & \boxed{x_{ij}} & \cdots & x_{iN} \\
\cdots & \cdots & \cdots & \cdots & \cdots\\
x_{N1} & \cdots & x_{Nj} & \cdots & x_{NN} \\
\end{vmatrix} ,
}
where $r^j_i$ and $c^i_j$ are the $i$-th row and the $j$-th column of $X$ with $x_{ij}$ deleted; cf.~\cite[Prop.~1.10.4]{Mo07} for the formula $|X|_{ij} = ((X^{-1})_{ji})^{-1}$.

For split symmetric pairs \eqref{split-pairs} the matrix $G$ is diagonal, $G = \sum_{a=1}^N \theta_{\bar a} E_{aa}$, allowing the diagonal series $s_{ii}(u)$ to be inverted. This is a crucial property that we will exploit. For $1< i,j<N$ set 
\ali{
\si_{ij}(u) &:= \begin{vmatrix}
s_{11}(u) & s_{1j}(u) \\
s_{i1}(u) & \boxed{s_{ij}(u)}
\end{vmatrix} = s_{ij}(u) - s_{i1}(u)\,s_{11}(u)^{-1} s_{1j}(u)
\label{si(u)}
\intertext{and}
\mr{S}(u) & \;= \sum_{1\le a_i,b_i\le N} E_{a_1 b_1} \ot E_{a_2 b_2} \ot \mr{S}^{a_1 a_2}_{b_1 b_2}(u) \nn \\
& := R(1)\,S_1(u+1)\,R'(\ka-2u-1)\,S_2(u) = S_2(u)\,R'(\ka-2u-1)\,S_1(u+1)\,R(1) .
\label{SS(u)}
}

\begin{lemma} \label{L:SS(u)}
For $1<i,j<N$ we have
\begin{flalign*}
\qq (1)\qu & [s_{11}(u) , \si_{ij}(v)] = 0 , & \\ 
\qq (2)\qu & \si_{ij}(u) = s^{-1}_{11}(u+1)\, \mr{S}^{1i}_{1j}(u).& 
\end{flalign*}
In (2) we assume $\mfg_N\ne\mfo_4$: for $\mfg_N=\mfo_4$ we have $\ka=1$, and $R(u)$ has a pole at $u=1$.
\end{lemma}

\begin{proof} (1)
Set $i=j=l=1$ and rename $k\to i$ in \eqref{[s,s]}. Then substitute $u\to u+1$ and $v\to u$. For $1< i < N$ we obtain
\ali{
s_{11}(u+1)\, s_{i1}(u) = \bigg( \frac{2u-\ka+2}{2u-\ka+1}\, s_{i1}(u+1) - \frac{1}{2u-\ka+1}\, s_{1i}(u+1)\bigg) s_{11}(u) . 
\label{s11-si1}
}
This allows us to rewrite \eqref{si(u)} as
\ali{
\si_{ij}(u) &= s_{ij}(u) - s_{11}^{-1}(u+1) s_{11}(u+1) s_{i1}(u)\,s^{-1}_{11}(u) s_{1j}(u) \nn\\ & = s_{ij}(u) - s^{-1}_{11}(u+1) \bigg( \frac{2u-\ka+2}{2u-\ka+1}\, s_{i1}(u+1) - \frac{1}{2u-\ka+1}\, s_{1i}(u+1) \bigg)\, s_{1j}(u) .
\label{si(u)-1}
}
Next, set $i=j=k=1$ and rename $l\to i$ in \eqref{[s,s]}. Then solve the resulting expression for $s_{i1}(u)$ and substitute $v\to \ka-u$. This gives
\equ{
s_{i1}(u) = \frac{s_{1i}(u) + (2u-\ka-1)\,s^{-1}_{11}(\ka-u)\, s_{11}(u)\, s_{1i}(\ka-u)}{2u-\ka} .
\label{si1(u)}
}
Substituting \eqref{si1(u)}, with $u$ replaced by $u+1$, into \eqref{s11-si1} and using $[s_{11}(u),s_{11}(v)]=0$, we obtain
\equ{
s_{i1}(u) = s^{-1}_{11}(\ka-u-1)\, s_{1i}(\ka-u-1)\, s_{11}(u) 
\label{si1-sim}
}
for $1<i<N$. Hence \eqref{si(u)} becomes
\equ{
\si_{ij}(u) = s_{ij}(u)-s^{-1}_{11}(\ka-u-1)\, s_{1i}(\ka-u-1)\, s_{1j}(u) . 
\label{si(u)-sim}
}
To get the wanted result we need to compute commutators of the form $[s_{11}(u), s_{1i}(v)]$ and $[s_{11}(u), s_{ij}(v)]$.
Set $i=j=1$ and rename $k\to i$, $l\to j$ in \eqref{[s,s]}. This gives
\ali{
(u+v-\ka) \, [s_{11}(u), s_{ij}(v)] &= s_{i1}(v)\, s_{j1}(u) - s_{1i}(u)\, s_{1j}(v) \nn\\ & \qu + \frac{u+v-\ka+1}{u-v}\,(s_{i1}(u)\, s_{1j}(v) - s_{i1}(v)\, s_{1j}(u)) .
\label{s11-sij}
}
Set $i=j=k=1$ and rename $l\to i$ in \eqref{[s,s]}. Then
\ali{
(u+v-\ka)\,[s_{11}(u),s_{1i}(v)] &= s_{11}(v)\, s_{i1}(u) + \frac{2v-\ka+1}{u-v}\,s_{11}(u)\,s_{1i}(v) - \frac{u+v-\ka+1}{u-v} \, s_{11}(v) \,s_{1i}(u) . \label{s11-s1i}
}
Now let $w:=\ka-v-1$. Replacing $v$ by $w$ in \eqref{s11-s1i} and writing $s_{11}(u)\,s_{1i}(w)=s_{1i}(w)\,s_{11}(u)+[s_{11}(u),s_{1i}(w)]$, we obtain
\ali{
(u-v)(u+v-\ka)\,[s_{11}(u),s_{1i}(w)] &= (u+v-\ka+1)\,s_{11}(w)\,s_{i1}(u) - (u-v)\,s_{11}(w)\,s_{1i}(u) \nn\\ & \qu - (2v-\ka+1)\,s_{1i}(w)\,s_{11}(u) . \label{s11-s1i-w}
}
Multiply \eqref{s11-s1i-w} by $s^{-1}_{11}(w)$ on the left and by $s_{1j}(v)$ on the right. Add to it \eqref{s11-s1i}, with $i$ replaced by $j$, multiplied by $u-v$ and, on the left, by $s^{-1}_{11}(w)\,s_{1i}(w)$. Since $s_{11}(u)$ commutes with $s^{-1}_{11}(w)$, and $s^{-1}_{11}(w)\,s_{1i}(w)\,s_{11}(v)=s_{i1}(v)$ by \eqref{si1-sim}, the two terms with the factor $2v-\ka+1$ cancel, and we obtain
\ali{
(u-v)(u+v-\ka)\,[s_{11}(u),\,s^{-1}_{11}(w)\,s_{1i}(w)\,s_{1j}(v)] &= (u-v)\,\big(s_{i1}(v)\,s_{j1}(u) - s_{1i}(u)\,s_{1j}(v)\big) \nn\\ & \qu + (u+v-\ka+1)\,\big(s_{i1}(u)\,s_{1j}(v) - s_{i1}(v)\,s_{1j}(u)\big) . \nn
}
By \eqref{s11-sij}, the right hand side equals $(u-v)(u+v-\ka)\,[s_{11}(u),s_{ij}(v)]$. Hence $[s_{11}(u),\si_{ij}(v)]=0$ by \eqref{si(u)-sim}.

\medskip

\noindent (2) This follows by evaluating $(e^*_1 \ot e^*_i) \,( \mr{S}(u)\, e_1 \ot e_j )$ and comparing the result with \eqref{si(u)-1}. 
\end{proof}

The following proposition underpins the rest of the paper.

\begin{prop} \label{P:emb1}
The mapping 
\equ{
\psi^{(N)}_1 \;:\; X^{\tw}(\mfg_{N-2}) \to X^{\tw}(\mfg_{N}) , \qu s_{ij}(u) \mapsto \si_{i+1,j+1}(u)
\label{psi_1}
}
defines an injective homomorphism of extended twisted Yangians for split symmetric pairs.
\end{prop}

\begin{proof}
We begin by proving that the mapping $\psi^{(N)}_1$ is a homomorphism of algebras. The proof relies on verification of matrix identities and is similar in spirit to that of \cite[Lem.~3.5--3.6]{JLM18} and \cite[Prop.~6.3]{LZ25}. Thus, we will omit long mechanical computations and indicate the key identities only. 

\smallskip

Set $a:=u-v$ and $\tl{a}:=u+v$. We have the following relation in $\End(\C^N)^{\ot 4}\ot X^{\tw}(\mfg_N)((u^{-1},v^{-1}))$:
\eqa{
& R_{23}(a-1)\,R_{13}(a)\,R_{24}(a)\,R_{14}(a+1) \,\mr{S}_{12}(u) \\ 
& \qu \times R'_{14}(\ka-1-\tl{a})\,R'_{13}(\ka-2-\tl{a})\,R'_{24}(\ka-\tl{a})\,R'_{23}(\ka-1-\tl{a}) \,\mr{S}_{34}(v) 
\\ 
& \qq\qq = \mr{S}_{34}(v) \, R'_{23}(\ka-1-\tl{a})\,R'_{24}(\ka-\tl{a})\,R'_{13}(\ka-2-\tl{a})\,R'_{14}(\ka-1-\tl{a}) \\ & \hspace{4.4cm} \times \mr{S}_{12}(u) \, R_{14}(a+1)\,R_{24}(a)\,R_{13}(a)\,R_{23}(a-1) \, \label{RE2}
}
which follows by a repeated application of the Yang-Baxter relation \eqref{YBE} and reflection equation \eqref{RE}. 
Set (cf.~\eqref{R(u)})
\equ{
\wt{R}(u) := \wt{I} - u^{-1} \wt{P} - (\ka - 1 - u)^{-1} \wt{Q} 
}
where operators $\wt{I}$, $\wt{P}$ and $\wt{Q}$ are defined via \eqref{IPQ} except the sums are over $1<i,j<N$. In the space spanned by vectors $e_{i} \ot e_j$ with $1<i,j<N$, the operator $\wt{R}(u)$ and its partially-transposed counterpart $\wt{R}'(u)$ coincide with $R(u)$ and $R'(u)$ of $X^{\tw}(\mfg_{N-2})$.
We will show that \eqref{RE2} specialises to
\equ{
\wt{R}_{24}(a)\,\mr{S}_{12}(u) \, \wt{R}'_{24}(\ka-1-\tl{a}) \,\mr{S}_{34}(v) = \mr{S}_{34}(v) \, \wt{R}'_{24}(\ka-1-\tl{a}) \,\mr{S}_{12}(u) \, \wt{R}_{24}(a) \, \label{RE2a}
}
in the sense that the matrix coefficients of both sides between vectors $e_1 \ot e_{i} \ot e_1 \ot e_{j}$ with $1< i,j<N$ coincide.

Using \eqref{SS(u)} and the Yang-Baxter equation (to move $R_{34}(1)$ rightwards) we rewrite the right hand side of \eqref{RE2} as
\ali{
& S_4(v)\,R'_{34}(\ka-2v-1)\,S_3(v+1)\,R'_{24}(\ka-\tl{a})\,R'_{23}(\ka-1-\tl{a})\,R'_{14}(\ka-1-\tl{a})\,R'_{13}(\ka-2-\tl{a}) \nn\\ & \qq\qq \times S_2(u)\,R'_{12}(\ka-2u-1)\,S_1(u+1)\, R_{34}(1)\,R_{12}(1) R_{14}(a+1)\,R_{24}(a)\,R_{13}(a)\,R_{23}(a-1) .
\label{rhs1}
}
It was shown in the proof of \cite[Lem.~3.6]{JLM18}%
\footnote{In the symplectic case the term containing $\varphi(u)$ in \cite[Lem.~3.5]{JLM18} vanishes, since $R_{12}(1)\,P_{13}\,Q_{23}=0$; thus one may take $\varphi(u)=0$.} 
that
\ali{
& R_{34}(1) \, R_{12}(1) \, R_{14}(a+1)\,R_{24}(a)\,R_{13}(a)\,R_{23}(a-1) \, e_1 \ot e_{i} \ot e_1 \ot e_{j} \nn\\
& \hspace{4.5cm} = \frac{a-2}{a-1} \, R_{34}(1)\,R_{12}(1)\,\wt{R}_{24}(a)\, e_1 \ot e_{i} \ot e_1 \ot e_{j} 
\label{RRRR1}
}
for $1< i,j< N$. 
By similar arguments, we find
\ali{
& (e_1^* \ot e_i^* \ot e_p^* \ot e_r^*) \Big( R'_{24}(\ka-\tl{a})\,R'_{23}(\ka-1-\tl{a})\,R'_{14}(\ka-1-\tl{a}) \nn\\ 
& \hspace{4.15cm} \times R'_{13}(\ka-2-\tl{a})\, R_{34}(1) \, R_{12}(1) \, e_q \ot e_s \ot e_1 \ot e_j \Big) \nn\\
& \qq = \frac{\tl{a}-\ka+3}{\tl{a}-\ka+2} \, (e_1^* \ot e_i^* \ot e_p^* \ot e_r^*) \, \Big( R_{34}(1) \,\wt{R}'_{24}(\ka-1-\tl{a})\, R_{12}(1) \, e_q \ot e_s \ot e_1 \ot e_j \Big)
\label{RRRR2}
}
for $1< i,j< N$ and $1\le p,r,q,s \le N$. Hence, applying the second row of \eqref{rhs1} to the vector $e_1 \ot e_i \ot e_1 \ot e_j$ and invoking \eqref{RRRR1} and \eqref{SS(u)} gives
\ali{
& \frac{a-2}{a-1} R_{34}(1)\, R_{12}(1) S_1(u+1)\,R'_{12}(\ka-2u-1)\,S_2(u) \, \wt{R}_{24}(a) \, e_1 \ot e_{i} \ot e_1 \ot e_{j} .
}
Next, applying the first row of \eqref{rhs1} to this result and invoking \eqref{RRRR2} and \eqref{SS(u)} gives
\ali{
& \frac{\tl{a}-\ka+3}{\tl{a}-\ka+2}\, \frac{a-2}{a-1} \, (e_1^* \ot e_k^* \ot e_p^* \ot e_r^*) \Big( S_4(v)\,R'_{34}(\ka-2v-1)\,S_3(v+1) \, R_{34}(1) \nn \\ & \hspace{1.5cm} \times \wt{R}'_{24}(\ka-1-\tl{a})\,  R_{12}(1) S_1(u+1)\,R'_{12}(\ka-2u-1)\,S_2(u) \, \wt{R}_{24}(a) \, e_1 \ot e_i \ot e_1 \ot e_j \Big)\nn\\[0.2em]
& \qq = \frac{\tl{a}-\ka+3}{\tl{a}-\ka+2}\, \frac{a-2}{a-1} \, (e_1^* \ot e_k^* \ot e_p^* \ot e_r^*) \Big( \mr{S}_{34}(v) \, \wt{R}'_{24}(\ka-1-\tl{a})\, \mr{S}_{12}(u) \, \wt{R}_{24}(a) \, e_1 \ot e_i \ot e_1 \ot e_j \Big) \label{rhs2}
}
for $1< i,j,k< N$ and $1\le p,r\le N$. Restricting to $p=1$ and $1< r < N$ yields the right hand side of the wanted identity, \eqref{RE2a}, up to an overall scalar factor. The left hand side of \eqref{RE2a} is obtained analogously. Namely, we rewrite the left hand side of \eqref{RE2} as (cf.~\eqref{rhs1})
\ali{
& R_{23}(a-1)\,R_{13}(a)\,R_{24}(a)\,R_{14}(a+1) \, R_{12}(1)\, R_{34}(1) \, S_1(u+1)\,R'_{12}(\ka-2u-1)\,S_2(u) \nn\\ 
& \qu \times R'_{13}(\ka-2-\tl{a})\,R'_{14}(\ka-1-\tl{a})\, R'_{23}(\ka-1-\tl{a}) \, R'_{24}(\ka-\tl{a})\, S_3(v+1)\,R'_{34}(\ka-2v-1) \,S_4(v) .
\label{lhs1}
}
Relations \eqref{RRRR1} and \eqref{RRRR2} imply that
\ali{
& (e^*_1 \ot e^*_{i} \ot e^*_1 \ot e^*_{j}) \Big( R_{23}(a-1) \, R_{13}(a)\, R_{24}(a)\, R_{14}(a+1)\, R_{12}(1) \, R_{34}(1) \, e_{q} \ot e_{s} \ot e_{p} \ot e_{r} \Big) \nn\\
& \hspace{4.5cm} = \frac{a-2}{a-1} \,(e^*_1 \ot e^*_{i} \ot e^*_1 \ot e^*_{j}) \Big( \wt{R}_{24}(a) \,R_{12}(1)\, R_{34}(1) \, e_{q} \ot e_{s} \ot e_{p} \ot e_{r} \Big)
\label{RRRR3}
}
and
\ali{
& 
(e^*_q \ot e^*_s \ot e^*_1 \ot e^*_j) \Big( R_{12}(1) \, R_{34}(1)\, R'_{13}(\ka-2-\tl{a}) \nn\\ 
& \hspace{3cm} \times \,R'_{14}(\ka-1-\tl{a})\,R'_{23}(\ka-1-\tl{a}) \,R'_{24}(\ka-\tl{a})\, 
e_1 \ot e_i \ot e_p \ot e_r \Big) \nn\\
& \qq = \frac{\tl{a}-\ka+3}{\tl{a}-\ka+2} \,(e^*_q \ot e^*_s \ot e^*_1 \ot e^*_j ) \Big( R_{12}(1) \, \wt{R}'_{24}(\ka-1-\tl{a}) \, R_{34}(1) \, e_1 \ot e_i \ot e_p \ot e_r \Big)
\label{RRRR4}
}
for $1<i,j<N$ and $1\le p,r,q,s\le N$.
Sandwiching \eqref{lhs1} with $e^*_1 \ot e^*_k \ot e^*_1 \ot e^*_l$ and $e_1 \ot e_i \ot e_1 \ot e_j$, and invoking \eqref{RRRR3}, \eqref{RRRR4} and \eqref{SS(u)} yields
\ali{
\frac{\tl{a}-\ka+3}{\tl{a}-\ka+2}\, \frac{a-2}{a-1} \, (e_1^* \ot e_k^* \ot e_1^* \ot e_l^* ) \Big(\wt{R}_{24}(a) \, \mr{S}_{12}(u) \, \wt{R}'_{24}(\ka-1-\tl{a})\, \mr{S}_{34}(v) \, e_1 \ot e_i \ot e_1 \ot e_j \Big) 
}
which is the left hand side of the wanted identity, \eqref{RE2a}, up to the same overall scalar factor as in \eqref{rhs2}.
This shows that the series $\mr{S}^{1,i+1}_{1,j+1}(u)$ with $1\le i,j\le N-2$ satisfy the defining relations of $X^{\tw}(\mfg_{N-2})$. Lemma~\ref{L:SS(u)} and relation $[s_{11}(u), s_{11}(v)]=0$ imply that the same is true for $\si_{i+1,j+1}(u)$, as required.

It remains to prove that $\psi^{(N)}_1$ is injective. To this end, it is sufficient to show that $\psi^{(N)}_1$ induces an injective homomorphism of the associated graded algebras, $\gr X^\tw(\mfg_{N-2}) \to \gr X^\tw(\mfg_N)$. 
Since $g_{i+1,1}=g_{1,j+1}=0$, every coefficient of $u^{-r}$ in $s_{i+1,1}(u)\,s^{-1}_{11}(u)\,s_{1,j+1}(u)$ has degree at most $r-2$. Hence, by \eqref{si(u)}, the image of $\si_{i+1,j+1}(u)$ in $\gr X^\tw(\mfg_{N})[[u^{-1}]]$ is
\equ{
\ol{\si}_{i+1,j+1}(u) = g_{i+1,j+1} + \sum_{r\ge 1} \bar{s}_{i+1,j+1}^{(r)} \,u^{-r} .
}
The matrices $G$, $J$ and the elements $\ga^{(r)}_{ij}$ for $\mfg_{N-2}$ are those for $\mfg_N$ restricted to the indices $2,\ldots,N-1$, relabelled as $1,\ldots,N-2$. Therefore, under the isomorphisms \eqref{X-graded-iso} for $\mfg_{N-2}$ and $\mfg_N$, the induced homomorphism is the tensor product of the homomorphism $U(\mfg_{N-2}[x]^\rho) \to U(\mfg_N[x]^\rho)$ induced by the natural inclusion with the identity map of $\C[\zeta_0,\zeta_1,\ldots]$. It is injective by the Poincar\'e--Birkhoff--Witt theorem, and the claim follows.
\end{proof}

\begin{remark}
The statement of Proposition \ref{P:emb1} holds true for all extended twisted Yangians $X^\tw(\mfg_N,G)$ with $g_{11}=1$. When $g_{11}=0$, the series $s_{11}(u)$ is non-invertible, and the construction above does not apply. Nor can the series $\mr{S}^{1,i+1}_{1,j+1}(u)$, which do not involve $s^{-1}_{11}(u)$, be used instead: by \eqref{SS(u)} we have
\equ{
\mr{S}^{1,i+1}_{1,j+1}(u) = g_{11}\,g_{i+1,j+1} - g_{i+1,1}\,g_{1,j+1} + O(u^{-1}) \qu\text{for}\qu 1\le i,j\le N-2 ,
}
and for $g_{11}=0$ (types DIII and CII) also $g_{i+1,1}=g_{1,j+1}=0$ for $1\le i,j\le N-2$, so these series have vanishing constant terms and cannot serve as generating series of an extended twisted Yangian of $\mfg_{N-2}$.
\end{remark}

\begin{crl} \label{C:emb}
The mapping 
\equ{
\psi^{(N)}_m \;:\; X^{\tw}(\mfg_{N-2m}) \to X^{\tw}(\mfg_{N}) , \qu s_{ij}(u) \mapsto \begin{vmatrix}
s_{11}(u) & \cdots & s_{1m}(u) & s_{1,m+j}(u) \\
\cdots & \cdots & \cdots & \cdots \\
s_{m1}(u) & \cdots & s_{mm}(u) & s_{m,m+j}(u) \\
s_{m+i,1}(u) & \cdots & s_{m+i,m}(u) & \boxed{s_{m+i,m+j}(u)} 
\end{vmatrix} 
\label{psi_m}
}
defines an injective homomorphism of extended twisted Yangians for split symmetric pairs.
\end{crl}

\begin{proof} 
This follows by induction on $m$. The base case is given by Proposition \ref{P:emb1}. The general step relies on Sylvester's theorem for quasi-determinants; see \cite[Proof of Thm.~3.7]{JLM18}. 
\end{proof}

The following statement is a counterpart of \cite[Prop.~3.8]{JLM18} and \cite[Prop.~6.5]{LZ25} and follows by the same argument from Sylvester's theorem for quasi-determinants.

\begin{prop} \label{P:psipsi}
We have the equality of mappings
\equ{
\psi^{(N)}_{l} \circ \psi^{(N-2l)}_{m} = \psi^{(N)}_{l+m} 
}
for all $l,m\ge1$ such that $\psi^{(N)}_{l+m}$ is defined, that is, $l+m<n-1$ when $\mfg_N=\mfo_{2n}$ and $l+m<n$ otherwise. \qed
\end{prop}

The following lemma is the counterpart of \cite[Cor.~3.10]{JLM18}. The argument of \cite{JLM18}, based on quantum minors, does not apply here, since the series $s_{ab}(u)$ with $1\le a,b\le m$ satisfy a twisted reflection equation rather than an $RTT$ relation. Instead, we reduce the claim to Lemma \ref{L:SS(u)}\,(1) by induction on $m$.

\begin{lemma} \label{L:emb-comm}
Let $1\le m<n$, or $m=n$ when $\mfg_N=\mfo_{2n+1}$. For $m<c,d<\bar m$ set
\equ{
\si^{[m]}_{cd}(u) := \begin{vmatrix}
s_{11}(u) & \cdots & s_{1m}(u) & s_{1d}(u) \\
\cdots & \cdots & \cdots & \cdots \\
s_{m1}(u) & \cdots & s_{mm}(u) & s_{md}(u) \\
s_{c1}(u) & \cdots & s_{cm}(u) & \boxed{s_{cd}(u)} 
\end{vmatrix} ,
\label{si[m]}
}
so that $\si^{[1]}_{cd}(u)=\si_{cd}(u)$, and $\si^{[m]}_{m+i,m+j}(u)=\psi^{(N)}_m(s_{ij}(u))$ by \eqref{psi_m} whenever $\psi^{(N)}_m$ is defined. Then the coefficients of the series $s_{ab}(u)$ with $1\le a,b\le m$ commute with those of the series $\si^{[m]}_{cd}(v)$ with $m<c,d<\bar m$. In particular, for $m<n-1$ when $\mfg_N=\mfo_{2n}$ or $m<n$ otherwise, we have
\equ{
[s_{ab}(u), \psi^{(N)}_m (s_{ij}(v))] = 0 
}
for all $1\le a,b \le m$ and $1\le i,j\le N-2m$. Here $s_{ab}(u)$ are series with coefficients in $X^\tw(\mfg_N)$ and $s_{ij}(v)$ are series with coefficients in $X^\tw(\mfg_{N-2m})$.
\end{lemma}

\begin{proof}
We argue by induction on $m$, for all $N$ at once. Fix $m<c,d<\bar m$ and put $I=\{1,\ldots,m\}$, so that, in the obvious block notation,
\[
\si^{[m]}_{cd}(v) = s_{cd}(v) - S_{cI}(v)\,S_{II}(v)^{-1}\,S_{Id}(v) .
\]
For $m=1$ the claim is Lemma \ref{L:SS(u)}\,(1). Let $m\ge2$. By Sylvester's theorem for quasi-determinants, $\si^{[m]}_{cd}(v)$ is the image under $\psi^{(N)}_1$ of the series $\si^{[m-1]}_{c-1,d-1}(v)$ of $X^\tw(\mfg_{N-2})$. Hence $s_{11}(u)$ commutes with $\si^{[m]}_{cd}(v)$ by Lemma \ref{L:SS(u)}\,(1), and so does $\si_{ab}(u)=\psi^{(N)}_1(s_{a-1,b-1}(u))$ for $2\le a,b\le m$, by the induction hypothesis applied in $X^\tw(\mfg_{N-2})$. 
Since $s_{ab}(u)=\si_{ab}(u)+s_{a1}(u)\,s^{-1}_{11}(u)\,s_{1b}(u)$ by \eqref{si(u)}, it remains to show that, for $2\le a\le m$, the coefficients of $s_{1a}(u)$ and $s_{a1}(u)$ commute with $\si^{[m]}_{cd}(v)$. By \eqref{si1(u)}, the coefficients of $s_{a1}(u)$ are polynomials in those of $s_{11}(u)$ and $s_{1a}(u)$, so it suffices to consider $s_{1a}(u)$.

Taking the coefficient of $u^{-1}$ in \eqref{[s,s]} with $(i,j)=(a,1)$ and using $g_{11}=g_{aa}=1$ and $\theta_{a1}=1$, we obtain
\[
[\,s^{(1)}_{a1},\,S(v)\,]=Y\,S(v)-S(v)\,Y ,\qq
Y:=E_{1a}-E_{a1}+E_{\bar 1\bar a}-E_{\bar a\bar 1} .
\]
Since $\bar 1>\bar a\ge\bar m>c,d$, the rows $I\cup\{c\}$ and the columns $I\cup\{d\}$ meet the support of $Y$ only in $\{1,a\}$. Hence, with $Y':=E_{1a}-E_{a1}$, the derivation $\partial:=[\,s^{(1)}_{a1},\,\cdot\,]$ satisfies $\partial s_{cd}(v)=0$ and, suppressing the argument~$v$,
\[
\partial\big(S_{cI}\,S_{II}^{-1}\,S_{Id}\big) = - S_{cI}\,Y'\,S_{II}^{-1}\,S_{Id} + S_{cI}\,[\,Y',S_{II}^{-1}]\,S_{Id} + S_{cI}\,S_{II}^{-1}\,Y'\,S_{Id} = 0 .
\]
Thus $\partial \si^{[m]}_{cd}(v)=0$. Finally, the coefficients of $u^{0}$ and $u^{-1}$ in \eqref{s11-s1i} with $i=a$, followed by those of $v^0$ and $v^{-r}$ with $r\ge1$, give
\[
s^{(1)}_{1a}=-s^{(1)}_{a1} , \qq 2s^{(r+1)}_{1a}=[\,s^{(2)}_{11},s^{(r)}_{1a}]+(\ka-1)\,s^{(r)}_{1a}-2s^{(r)}_{11}\,s^{(1)}_{a1} \qu (r\ge1) .
\]
By induction on $r$, all coefficients of $s_{1a}(u)$ commute with $\si^{[m]}_{cd}(v)$.
\end{proof}


\section{Gaussian series} \label{sec:gauss}


\subsection{Gaussian decomposition}

By \cite[Thm.~4.9.7]{GGRW05}, the generating matrix $S(u)$ has a unique decomposition
\equ{
S(u) = F(u)\,D(u)\,E(u) \label{S=FDE}
}
where
\gat{\setlength\arraycolsep{4pt}
D(u) = \begin{bmatrix}
d_1(u) && \cdots & 0 \\
& d_2(u) && \vdots \\
\vdots & & \ddots \\
0 & \cdots & & d_N(u)
\end{bmatrix} , 
\\[0.5em] 
\setlength\arraycolsep{4pt}
F(u) = \begin{bmatrix}
1 && \cdots & 0\, \\
f_{21}(u) & \ddots && \vdots\, \\
\vdots & & \ddots \\
f_{N1}(u) & f_{N2}(u) & \cdots & 1\,
\end{bmatrix} , \qq
E(u) = \begin{bmatrix}
\,1 & e_{12}(u) & \dots & e_{1N}(u) \\
 & \ddots && e_{2N}(u) \\
\,\vdots & & \ddots & \vdots \\
\,0 & \cdots & & 1
\end{bmatrix},
}
and the matrix entries above are defined in terms of quasi-determinants:
\gat{\setlength\arraycolsep{2pt}
d_i(u) = \begin{vmatrix} 
s_{11}(u) & \cdots & s_{1,i-1}(u) & s_{1i}(u) \\
\vdots & \ddots & \vdots & \vdots \\ 
s_{i-1,1}(u) & \cdots & s_{i-1,i-1}(u) & s_{i-1,i}(u) \\
s_{i1}(u) & \cdots & s_{i,i-1}(u) & \boxed{s_{ii}(u)} \\
\end{vmatrix},
\\[0.5em] 
\setlength\arraycolsep{2pt}
f_{ji}(u) = \begin{vmatrix} 
s_{11}(u) & \cdots & s_{1,i-1}(u) & s_{1i}(u) \\
\vdots & \ddots & \vdots & \vdots \\ 
s_{i-1,1}(u) & \cdots & s_{i-1,i-1}(u) & s_{i-1,i}(u) \\
s_{j1}(u) & \cdots & s_{j,i-1}(u) & \boxed{s_{ji}(u)} \\
\end{vmatrix} d^{-1}_i(u) , 
\qu
e_{ij}(u) = d^{-1}_i(u) \begin{vmatrix} 
s_{11}(u) & \cdots & s_{1,i-1}(u) & s_{1j}(u) \\
\vdots & \ddots & \vdots & \vdots \\ 
s_{i-1,1}(u) & \cdots & s_{i-1,i-1}(u) & s_{i-1,j}(u) \\
s_{i1}(u) & \cdots & s_{i,i-1}(u) & \boxed{s_{ij}(u)} \\
\end{vmatrix} .
}
We will refer to the series $f_{ji}(u)$, $d_i(u)$, $e_{ij}(u)$ as the {\it Gaussian series}, and to their coefficients $f^{(r)}_{ji}$, $d_{i,r}$, $e^{(r)}_{ij}$ defined by
\equ{
f_{ji}(u) = \sum_{r\ge 0} f^{(r)}_{ji} u^{-r-1} ,\qu 
d_{i}(u) = g_{ii} + \sum_{r\ge 0} d_{i,r} u^{-r-1} , \qu
e_{ij}(u) = \sum_{r\ge 0} e^{(r)}_{ij} u^{-r-1} 
}
as the {\it Gaussian generators}.

We will need a {\it reduced generating matrix} defined as follows. For $0\le m < n$ define $(N-2m)\times (N-2m)$-dimensional matrices
\gat{
\setlength\arraycolsep{2pt}
\renewcommand{\arraystretch}{0.7}
D^{[m]}(u) := \begin{bmatrix}
d_{m+1}(u) && \cdots & 0 \\
& d_{m+2}(u) && \vdots \\
\vdots & & \ddots \\
0 & \cdots & & d_{\ol{m+1}}(u)
\end{bmatrix} , \label{D[m]}
\\[0.5em] 
\setlength\arraycolsep{2pt}
\renewcommand{\arraystretch}{0.7}
F^{[m]}(u) := \begin{bmatrix}
1 && \cdots & 0\, \\
f_{m+2,m+1}(u) & \ddots && \vdots\, \\
\vdots & & \ddots \\
f_{\ol{m+1},m+1}(u) & \cdots & f_{\ol{m+1},\ol{m+2}}(u) & 1\,
\end{bmatrix} , \qu\;
E^{[m]}(u) := \begin{bmatrix}
\,1 & e_{m+1,m+2}(u) & \dots & e_{m+1,\ol{m+1}}(u) \\
 & \ddots && \vdots \\
\,\vdots & & \ddots & e_{\ol{m+2},\ol{m+1}}(u) \\
\,0 & \cdots & & 1
\end{bmatrix} . \label{F[m]E[m]}
}
Here recall that $\bar \imath = N-i+1$.
The reduced generating matrix is then
\equ{
S^{[m]}(u) := F^{[m]}(u)\,D^{[m]}(u)\,E^{[m]}(u) . \label{S[m]}
}
Let $s^{[m]}_{ij}(u)$ with $m< i,j< \ol{m}$ denote matrix entries of $S^{[m]}(u)$. We then have the following analogues of \cite[Prop.~4.1]{JLM18} and \cite[Cor.~4.2]{JLM18}.

\begin{prop} \label{P:S[m]}
Let $0\le m<n-1$ when $\mfg_N=\mfo_{2n}$ and $0\le m<n$ otherwise, and set $\psi^{(N)}_0:=\mathrm{id}$. The series $s^{[m]}_{m+i,m+j}(u)$ of $X^{\tw}(\mfg_{N})$ coincide with the images of the generating series $s_{ij}(u)$ of $X^{\tw}(\mfg_{N-2m})$ under the embedding \eqref{psi_m}, that is,
\equ{
s^{[m]}_{m+i,m+j}(u) = \psi^{(N)}_m(s_{ij}(u)) \qu\text{for} \qu 1 \le i,j\le N-2m .
}
\end{prop}

\begin{proof}
The case $m=0$ is trivial, so let $m\ge1$. Put $I=\{1,\ldots,m\}$, $K=\{m+1,\ldots,N-m\}$, and use the block notation with respect to $I$ and $K$. Since $F$ is lower and $E$ is upper triangular, \eqref{S=FDE} gives
\[
S_{II}=F_{II}D_{II}E_{II}, \qu 
S_{KI}=F_{KI}D_{II}E_{II}, \qu 
S_{IK}=F_{II}D_{II}E_{IK}, \qu 
S_{KK}=F_{KI}D_{II}E_{IK}+F_{KK}D_{KK}E_{KK}.
\] 
Hence $S_{KK} - S_{KI}\,S_{II}^{-1}\,S_{IK} = F_{KK}\,D_{KK}\,E_{KK} = S^{[m]}$, and the $(m+i,m+j)$-th entry of the left hand side is the quasi-determinant in \eqref{psi_m}.
\end{proof}

\begin{crl} \label{C:RE[m]}
For $m$ as in Proposition \ref{P:S[m]}, the subalgebra $X^{\tw[m]}(\mfg_N) \subset X^{\tw}(\mfg_N)$ generated by the coefficients of the series $s^{[m]}_{ij}(u)$ with $m<i,j<\ol{m}$ is isomorphic to $X^{\tw}(\mfg_{N-2m})$. In particular,
\[
R^{[m]}(u-v)\,S_1^{[m]}(u)\,R^{[m]\prime}(\ka^{[m]}-u-v)\,S_2^{[m]}(v) = S_2^{[m]}(v)\,R^{[m]\prime}(\ka^{[m]}-u-v)\,S_1^{[m]}(u)\,R^{[m]}(u-v) 
\]
where $\ka^{[m]}:=\frac{N-2m}{2} - \theta$ and $R^{[m]}(u)$ is the $R$-matrix of $X^{\tw}(\mfg_{N-2m})$.
\end{crl}

\begin{proof}
By Proposition \ref{P:S[m]}, the injective homomorphism $\psi^{(N)}_m$ of Corollary \ref{C:emb} maps the generating matrix of $X^{\tw}(\mfg_{N-2m})$ to $S^{[m]}(u)$, with the indices shifted by $m$. Hence its image is $X^{\tw[m]}(\mfg_N)$.
\end{proof}


\subsection{Inverse Gaussian decomposition} 

Inverting \eqref{S=FDE} gives
\equ{
S^{-1}(u) = E^{-1}(u)\,D^{-1}(u)\,F^{-1}(u) \label{Si=FDEi}
}
where $E^{-1}(u)$ is an upper triangular matrix and $F^{-1}(u)$ is a lower triangular matrix with entries given by $(E^{-1}(u))_{kk} = (F^{-1}(u))_{kk} = 1$ for $1\le k \le N$ and
\ali{
(F^{-1}(u))_{ji} = - f_{ji}(u) + \sum_{i<a<j} f_{ja}(u)\,f_{ai}(u) & - \sum_{i<a<b<j} f_{jb}(u)\,f_{ba}(u)\,f_{ai}(u) \nn \\[0.2em] 
& + \ldots + (-1)^{j-i} f_{j,j-1}(u) \cdots f_{i+1,i}(u) , \label{Fi} \\
(E^{-1}(u))_{ij} = - e_{ij}(u) + \sum_{i<a<j} e_{ia}(u)\,e_{aj}(u) & - \sum_{i<a<b<j} e_{ia}(u)\,e_{ab}(u)\,e_{bj}(u) \nn \\ & + \ldots + (-1)^{j-i} e_{i,i+1}(u) \cdots e_{j-1,j}(u) \label{Ei} 
}
for $1\le i < j \le N$.
In particular, the lower right $n \times n$ submatrix of $S^{-1}(u)$ only involves $f_{ji}(u)$, $e_{ij}(u)$ with $N-n < i<j\le N$ and $d_k(u)$ with $N-n<k\le N$.
The inverse $(S^{[m]}(u))^{-1}$ of the reduced generating matrix $S^{[m]}(u)$ is given by analogous formulas.


\section{Low rank cases} \label{sec:low}

Isomorphisms for extended twisted Yangians of low rank were investigated in \cite{GRW16}; we adapt those results to the presentation used in this paper, beginning by recalling the rank one case from \cite{LWZ25a}.


\subsection{\texorpdfstring{Extended twisted Yangian $X^+(\mfgl_2)$}{Extended twisted Yangian X\textasciicircum{}+(gl\_2)}}

Consider the extended twisted Yangian $X^+(\mfgl_2)$ (see \cite[\S2.13]{Mo07}) and write its generating matrix as
\equ{
S^\circ(u) =
\begin{bmatrix} 
1 & 0 \\ f^\circ_{21}(u) & 1 
\end{bmatrix}
\begin{bmatrix} 
d^\circ_1(u) & 0 \\ 0 & d^\circ_2(u) 
\end{bmatrix}
\begin{bmatrix} 
1 & e^\circ_{12}(u) \\ 0 & 1 
\end{bmatrix} .
}
Here (and below) the circle $^\circ$ indicates belonging to $X^+(\mfgl_2)$. Note that
\equ{
f^\circ_{21}(u), e^\circ_{12}(u) \in u^{-1} X^+(\mfgl_2) [[u^{-1}]], \qu 
d^\circ_{1}(u), d^\circ_{2}(u) \in 1 + u^{-1} X^+(\mfgl_2) [[u^{-1}]] .
}
We recall that the $R$-matrix of $X^+(\mfgl_2)$ is
\equ{
R^\circ(u) = I - u^{-1} P
}
with $I$ and $P$ given by \eqref{IPQ} with $N=2$.
For any function or operator $f(u)$ set
\equ{
\{ f(u) \}^u := f(u) + f(-u) . \label{u-symm}
}
Defining relations of $X^+(\mfgl_2)$ \cite[eq.~(2.6)]{Mo07} then imply (for a derivation, see \cite[\S4]{LWZ25a}):
\ali{
e^{\circ}_{12}(u) = f^{\circ}_{21}(-u-1&) , \qu d^{\circ}_1(u)\,d^{\circ}_2(-1-u) = d^{\circ}_2(u)\,d^{\circ}_1(-1-u), \qu [d^{\circ}_i(u),\,d^{\circ}_j(v)] = 0 \qu\text{for}\qu i,j = 1,2, \label{gl2:1} 
\\
d^{\circ}_1(u)\,f^{\circ}_{21}(v) &= \bigg(\frac{(u-v-1)(u+v+1)}{(u-v)(u+v)}\,f^{\circ}_{21}(v) + \bigg\{ \frac{u+\tfrac12}{u\,(u-v)}\,f^{\circ}_{21}(u) \bigg\}^{\!u}\, \bigg) \, d^{\circ}_1(u) , \label{gl2:2} 
\\
d^{\circ}_2(u-1)\,f^{\circ}_{21}(v) &= \bigg( \frac{(u-v)(u+v)}{(u-v-1)(u+v+1)}\,f^{\circ}_{21}(v) -\bigg\{ \frac{u-\tfrac12}{u\,(u-v-1)}\,f^{\circ}_{21}(u-1) \bigg\}^{\!u}\, \bigg) \, d^{\circ}_2(u-1) , \label{gl2:3} 
\\
[f^{\circ}_{21}(u), f^{\circ}_{21}(v)] &= -\frac{1}{u-v}\,(f^{\circ}_{21}(u)-f^{\circ}_{21}(v))^2 - \frac{1}{u+v+1}\,((d^{\circ}_1(u))^{-1} d^{\circ}_2(u) - (d^{\circ}_1(v))^{-1} d^{\circ}_2(v)) . \label{gl2:4}
}

The coefficients of the following series are central in $X^+(\mfgl_2)$:
\equ{
\ms{z}^{\circ}(u) := d^{\circ}_1(-u)\,d^{\circ}_2(u-1) , \qu 
\ms{c}^{\circ}(u) := d^{\circ}_1(u)\,(d^{\circ}_1(-u))^{-1}. \label{gl2:5}
}
The series $\ms{c}^{\circ}(u)$ is the {\it symmetry series} of $X^+(\mfgl_2)$, and $\ms{c}^{\circ}(u)\,\ms{z}^{\circ}(u) = d^{\circ}_1(u)\,d^{\circ}_2(u-1)$ is its {\it Sklyanin determinant}, see \cite[\S2.5 and \S2.13]{Mo07}.
The coefficients of $\ms{z}^{\circ}(u)$ at the even powers of $u^{-1}$ and the coefficients of $\ms{c}^{\circ}(u)$ at the odd powers of $u^{-1}$ are algebraically independent and generate the whole centre of $X^+(\mfgl_2)$.

Introduce the series
\equ{
b(u) := f^{\circ}_{21}(u-\tfrac12), \qu 
h(u) := (d^{\circ}_1(u-\tfrac12))^{-1} d^{\circ}_2(u-\tfrac12) 
}
and elements $b_r$, $h_r$ with $r\ge 0$ via
\[
b(u) = \sum_{r\ge 0} b_r u^{-r-1}, \quad
h(u) = 1+ \sum_{r\ge 0} h_r u^{-r-1}.
\]
Relations \eqref{gl2:1}--\eqref{gl2:4} imply \cite[Lem.~4.2]{LWZ25a}\footnote{Relation \eqref{gl2:7} is a refinement of relation (4.19) in \cite{LWZ25a}.}
\gat{
[h(u),h(v)] = 0, \qu h(-u) = h(u) , \qu
[b(u),b(v)] = \frac{c}{2} \frac{(b(u) - b(v))^2}{v-u} + \frac{h(v)-h(u)}{v+u} , \label{gl2:6}
\\
h(u)\,b(v) = \bigg( \frac{(u-v+\tfrac{c}{2})(u+v-\tfrac{c}{2})}{(u-v-\tfrac{c}{2})(u+v+\tfrac{c}{2})} \, b(v) - \bigg\{ \frac{c\,(u-\tfrac{c}{2})}{u\,(u-v-\tfrac{c}{2})}\, b(u-\tfrac{c}{2}) \bigg\}^{\!u}\, \bigg) \, h(u) , \label{gl2:7}
}
with $c=2$. In particular, $h_{2r}=0$ for $r\ge 0$.
By \cite[Thm.~5.3]{LWZ25a}, the subalgebra generated by elements $b_r$ and $h_{2r+1}$, $r\ge 0$, is isomorphic to the twisted Yangian $Y^+(\mfsl_2)$, and \eqref{gl2:6}--\eqref{gl2:7} are its defining relations.


\subsection{\texorpdfstring{Extended twisted Yangian $X^\tw(\mfsp_2)$}{Extended twisted Yangian X\textasciicircum{}tw(sp\_2)}}

The $R$-matrix of $X^\tw(\mfsp_2)$ and that of $X^+(\mfgl_2)$ satisfy
\equ{
\label{RR-sp2}
R(u) = \frac{u-1}{u-2}\,R^{\circ}(\tfrac12u) , \qu 
R'(\ka-u) = \frac{u-1}{u}\,R^{\circ\prime}(-\tfrac12(u-2)) 
}
where $\ka = 2$. The $G$-matrix is $G=E_{11}-E_{22}$. Set $A:=E_{11}+\sqrt{-1}E_{22}$. This leads to the following analogue of the fourth statement of \cite[Prop.~4.1]{GRW16}.

\begin{prop} \label{P:sp2->gl2}
The mapping $X^{\tw}(\mfsp_2) \to X^+(\mfgl_2)$ given by
\equ{
S(u) \mapsto A S^\circ(\tfrac 12\,(u-1))\, A
\label{sp2->gl2}
}
defines an isomorphism of algebras. In terms of Gaussian series it reads as
\ali{
\label{sp2-G1}
d_1(u) & \mapsto d^\circ_1(\tfrac 12\,(u-1)) , \qu &e_{12}(u) &\mapsto \sqrt{-1}\,e^\circ_{12}(\tfrac 12\,(u-1)) , \\
\label{sp2-G2}
d_{2}(u) & \mapsto -d^\circ_{2}(\tfrac 12\,(u-1)) , \qu & f_{21}(u) & \mapsto \sqrt{-1}\,f^\circ_{21}(\tfrac 12\,(u-1)) .
}
\end{prop}

\begin{proof}
Applying the mapping \eqref{sp2->gl2} and matrix identities \eqref{RR-sp2} to \eqref{RE} yields defining relations of $X^+(\mfgl_2)$. The mapping \eqref{sp2->gl2} is clearly invertible and so it is an isomorphism of algebras. Relations \eqref{sp2-G1} and \eqref{sp2-G2} follow by applying Gaussian decomposition to both sides of \eqref{sp2->gl2}.
\end{proof}

Combining Proposition \ref{P:sp2->gl2} with relations \eqref{gl2:1}--\eqref{gl2:5} we obtain the following.

\begin{crl} \label{C:sp2}
The following relations hold in $X^\tw(\mfsp_2)$:
\gat{
e_{12}(u) = f_{21}(-u) , \qu d_1(u)\, d_2(-u) = d_2(u)\, d_1(-u) , \qu [ d_i(u), d_j(v) ] = 0 \qu\text{for}\qu i,j=1,2, \label{sp2:1}
\\[0.2em]
d_1(u+1)\,f_{21}(v) = \bigg[ \frac{(u-v-1)(u+v+1)}{(u-v+1)(u+v-1)}\,f_{21}(v) + \bigg\{ \frac{2\,(u+1)}{u\,(u-v+1)}\,f_{21}(u+1) \bigg\}^{\!u} \,\bigg]\, d_1(u+1) , \label{sp2:2}
\\[0.2em]
d_2(u-1)\,f_{21}(v) = \bigg[ \frac{(u-v+1)(u+v-1)}{(u-v-1)(u+v+1)}\,f_{21}(v) - \bigg\{ \frac{2\,(u-1)}{u\,(u-v-1)}\,f_{21}(u-1) \bigg\}^{\!u} \,\bigg]\, d_2(u-1) , \label{sp2:3}
\\[0.2em]
[f_{21}(u), f_{21}(v)] = \frac{2}{v-u} (f_{21}(u)-f_{21}(v))^2 + \frac{2}{v+u} ( d_1^{-1}(v)\,d_2(v) - d_1^{-1}(u)\,d_2(u) ) . \label{sp2:4}
}
The coefficients of the following series are central in $X^\tw(\mfsp_2)$:
\equ{
\ms{z}(u) := -d_1(u)\,d_2(-u), \qu 
\ms{c}(u) := d^{-1}_1(2-u)\, d_1(u) .
\label{sp2:5}
}
\end{crl}

Introduce the series
\equ{
b(u) := \tfrac{1}{\sqrt{-2}} f_{21}(u) , \qu 
h(u) := -d_1^{-1}(u)\, d_2(u) 
}
and note that the mapping \eqref{sp2->gl2} maps these series as
\equ{
b(u) \mapsto \tfrac{1}{\sqrt{2}}\, b^\circ(\tfrac12u), \qu 
h(u) \mapsto h^\circ(\tfrac12u) .
}
This immediately yields the following result.

\begin{thm} \label{T:iso-sp2}
The special twisted Yangian $SY^{\tw}(\mfsp_2)$ is isomorphic to the algebra generated by coefficients of the series $b(u)$ and $h(u)$ subject to relations \eqref{gl2:6}--\eqref{gl2:7} with $c=4$.
\end{thm}


\subsection{\texorpdfstring{Extended twisted Yangian $X^\tw(\mfo_3)$}{Extended twisted Yangian X\textasciicircum{}tw(o\_3)}}

Let $V \subset \C^2 \ot \C^2$ be the three-dimensional subspace spanned by vectors
\[
v_1 := e_1 \ot e_1 , \qu 
v_2 := \tfrac{1}{\sqrt{2}} ( e_1 \ot e_2 + e_2 \ot e_1 ) , \qu
v_3 := - e_2 \ot e_2 ,
\]
and let the projection $|_{V\ot V}$ from $\C^2 \ot \C^2 \ot \C^2 \ot \C^2$ to its subspace $V \ot V$ be given by $\tfrac14\, R^\circ_{12}(-1)\, R^\circ_{34}(-1)$. Let $R(u)$ be the $R$-matrix of $X^\tw(\mfo_3)$. Then 
\ali{
\label{RR-o3-1}
R(u) &\equiv \frac{u+\frac12}{u-\frac12} \, R^\circ_{14}(2(u-\tfrac12))\,R^\circ_{13}(2u)\,R^\circ_{24}(2u)\,R^\circ_{23}(2(u+\tfrac12)) \bigg|_{V \ot V} ,
\\
\label{RR-o3-2}
R'(\ka-u) &\equiv \frac{u-1}{u} \, R^{\circ\prime}_{23}(2(1-u))\, R^{\circ\prime}_{24}(2(\tfrac12-u))\, R^{\circ\prime}_{13}(2(\tfrac12-u))\, R^{\circ\prime}_{14}(-2u) \bigg|_{V \ot V} ,
}
where $\ka=\tfrac12$ and $\equiv$ denotes equality of matrix operators upon identification of basis vectors $e_i \ot e_j \in \C^3 \ot \C^3$ with $v_i \ot v_j \in V \ot V$. In particular, the kernel of the projection $|_{V\ot V}$ is stable under the action of matrix operators in the rhs\ of \eqref{RR-o3-1} and \eqref{RR-o3-2}, see \cite[\S4.2]{AMR06}. The $G$-matrix is the identity matrix. 
We now state the following analogue of \cite[Prop.~4.5]{GRW16}.

\begin{prop} \label{P:o3->gl2}
The mapping $X^{\tw}(\mfo_3) \to X^+(\mfgl_2)$ given by
\equ{
\label{o3->gl2}
S(u) \mapsto \tfrac12 R^\circ(-1)\, S_1^\circ(2u-1)\,R^{\circ\prime}(-4u+1)\,S_2^\circ(2u) \,\Big|_{V} ,
}
where basis vectors $e_i \in \C^3$ and $v_i \in V$ are identified, defines an isomorphism of algebras.
In terms of Gaussian series it reads as
\gan{
d_1(u) \mapsto \tfrac{4u}{4u-1}\, d^\circ_1(2u-1)\,d^\circ_1(2u) , \qu 
d_2(u) \mapsto d^\circ_1(2u-1)\,d^\circ_2(2u) , \qu 
d_3(u) \mapsto \tfrac{4u+1}{4u}\, d^\circ_2(2u-1)\,d^\circ_2(2u),
\\
e_{12}(u) \mapsto \sqrt{2}\,e^\circ_{12}(2u) , \qu
e_{13}(u) \mapsto - e^\circ_{12}(2u)\,e^\circ_{12}(2u) - \tfrac{1}{4u}\, (d^\circ_1(2u))^{-1}\,d^\circ_2(2u) ,\qu 
e_{23}(u) \mapsto -\sqrt{2}\,e^\circ_{12}(2u-1) , 
\\
f_{21}(u) \mapsto \sqrt{2}\,f^\circ_{21}(2u) , \qu
f_{31}(u) \mapsto - f^\circ_{21}(2u)\,f^\circ_{21}(2u) - \tfrac{1}{4u}\, (d^\circ_1(2u))^{-1}\,d^\circ_2(2u) ,\qu 
f_{32}(u) \mapsto -\sqrt{2}\,f^\circ_{21}(2u-1) .
}
\end{prop}

\begin{proof}
The explicit form of the mapping \eqref{o3->gl2} is obtained by applying Gaussian decomposition \eqref{S=FDE} to the generating matrices on both sides of \eqref{o3->gl2}, taking matrix entries, and applying relations \eqref{gl2:1}--\eqref{gl2:4}. For instance, taking the (1,1)-th matrix entry on both sides of the mapping immediately gives
\equ{
\label{o3:G1}
d_1(u) \mapsto \frac{4u}{4u-1}\, d^\circ_1(2u-1)\,d^\circ_1(2u) .
}
Next, taking the (1,2)-th matrix entry on both sides of the mapping gives
\equ{
d_1(u)\,e_{12}(u) \mapsto \frac{d^\circ_1(2u-1)}{4u-1} \bigg( 2\sqrt{2} \, u \, d^\circ_1(2u) \, e^\circ_{12}(2u) + \frac{ 4u-1}{\sqrt{2}}\,e^\circ_{12}(2u-1) \, d^\circ_1(2u) + \frac{1}{\sqrt{2}}\, f^\circ_{21}(2u) \, d^\circ_1(2u) \bigg) .
\label{o3:G2}
}
From \eqref{gl2:1}--\eqref{gl2:2} we know that
\[
f^\circ_{21}(2u)\,d^\circ_1(2u) = 4u\, d^\circ_1(2u)\,e^\circ_{12}(2u) - (4u-1)\,e^\circ_{12}(2u-1)\,d^\circ_1(2u) ,
\]
allowing us to rewrite the rhs~of \eqref{o3:G2} as $\frac{4\sqrt{2}\,u}{4u-1} \, d^\circ_1(2u-1)\,d^\circ_1(2u)\,e^\circ_{12}(2u)$.
It remains to combine \eqref{o3:G1} with \eqref{o3:G2}, yielding the wanted expression, $e_{12}(u) \mapsto \sqrt{2}\, e^\circ_{12}(2u)$.
Images of the remaining Gaussian series of $X^\tw(\mfo_3)$ are obtained in a similar way.

We now turn to proving that the mapping \eqref{o3->gl2} defines an isomorphism of algebras. Applying \eqref{o3->gl2} and matrix identities \eqref{RR-o3-1}--\eqref{RR-o3-2} to the defining relation \eqref{RE} yields a matrix relation, which holds true by virtue of the quantum Yang-Baxter equation, properties of projection operators and defining relations of $X^+(\mfgl_2)$. (An analogous computation is presented in the proof of \cite[Prop.~4.7]{GRW16}.) This proves that the mapping \eqref{o3->gl2} defines a homomorphism of algebras.
Next, consider the associated graded algebras $\gr X^\tw(\mfo_3)$ and $\gr X^+(\mfgl_2)$. 
A linear basis of $\gr X^\tw(\mfo_3)$ (resp.~$\gr X^+(\mfgl_2)$) is given by ordered monomials in the elements $\bar{s}^{(r)}_{21}$, $\bar{s}^{(r)}_{11}$, $\bar{s}^{(2r)}_{22}$ (resp.~$\bar{s}^{\circ(r)}_{21}$, $\bar{s}^{\circ(r)}_{11}$, $\bar{s}^{\circ(2r)}_{22}$) with $r\ge 1$.
The homomorphism \eqref{o3->gl2} descends to a bijective linear map $\gr X^\tw(\mfo_3) \to \gr X^+(\mfgl_2)$ given by 
\gan{
\bar{s}_{21}^{(r)} \mapsto 2^{1/2-r}\,\bar{s}^{\circ(r)}_{21} , \qu
\bar{s}_{11}^{(r)} \mapsto \tfrac14\del_{r1} + 2^{1-r}\,\bar{s}^{\circ(r)}_{11} , \qu 
\bar{s}^{(2r)}_{22} \mapsto 2^{-2r} (\bar{s}^{\circ(2r)}_{22} + \bar{s}^{\circ(2r)}_{11})
}
and so it defines an isomorphism of graded algebras. Consequently, the mapping \eqref{o3->gl2} is an isomorphism of algebras.
\end{proof}

Combining Proposition \ref{P:o3->gl2} with relations \eqref{gl2:1}--\eqref{gl2:5} we obtain the following result.

\begin{crl} \label{C:so3}
The following identities hold in $X^\tw(\mfo_3)$:
\gat{
e_{12}(u-\tfrac12) = -e_{23}(u) = - f_{32}(-u+\tfrac12) = f_{21}(-u) , \label{so3:1}
\\
e_{13}(u) = -\tfrac12 f_{21}^2(-u-\tfrac12) - \frac{1}{4u-1}\, d^{-1}_1(u)\,d_2(u) , \qu  
f_{31}(u) = -\tfrac12 f_{21}^2(u) - \frac{1}{4u-1}\, d^{-1}_1(u)\,d_2(u) , \label{so3:2}
\\[0.2em] 
\frac{2u-1}{4u-1} \, d_1(u-\tfrac14)\, d_2(-u-\tfrac14) = \frac{2u+1}{4u+1}\, d_1(-u-\tfrac14)\, d_2(u-\tfrac14) , \label{so3:3} \\[0.5em]
 d_1(u)\,d_3(-u) = d_1(-u)\,d_3(u) = d_2(u)\,d_2(-u) , \qu [ d_i(u) , d_j(v) ] = 0 \qu\text{for}\qu 1\le i,j\le3, \label{so3:4}
}
\ali{
d_1(u+\tfrac14)\,f_{21}(v) &= \bigg[ \frac{(u-v-\frac34)(u+v+\frac34)}{(u-v+\frac14)(u+v-\frac14)}\,f_{21}(v) + \,\bigg\{ \frac{u+\frac12}{u\,(u-v+\frac14)}\,f_{21}(u+\tfrac14) \bigg\}^{\!u} \,\bigg]\, d_1(u+\tfrac14) , \label{so3:5}
\\[0.3em]
d_2(u)\,f_{21}(v) &= \bigg[ \frac{(u-v+\frac12)(u-v-1)(u+v)(u+v+\frac12)}{(u-v-\frac12)(u-v)(u+v+1)(u+v-\frac12)}\,f_{21}(v) \nn \\
& \qq - \frac{u\,(u+\frac14)\,f_{21}(u)}{(u-\frac14)(u+\frac12)(u-v)} + \frac{u\,(u-\frac14)\,f_{21}(u-\tfrac12)}{(u+\frac14)(u-\frac12)(u-v-\frac12)} \nn \\
& \qq - \frac{u\,(u-\frac34)\,f_{21}(-u+\tfrac12)}{3\,(u-\frac14)(u-\frac12)(u+v-\frac12)} + \frac{u\,(u+\frac34)f_{21}(-u-1) }{3\,(u+\frac14)(u+\frac12)(u+v+1)} \bigg]\, d_2(u) , \!\!\! \label{so3:6}
\\[0.3em]
[f_{21}(u), f_{21}(v)] &= \frac{\frac12(f_{21}(u)-f_{21}(v))^2}{v-u} + \frac{1}{u+v+\frac12} \bigg[ \frac{v}{v-\frac14}\,d_1^{-1}(v)\,d_2(v) - \frac{u}{u-\frac14}\,d_1^{-1}(u)\,d_2(u)\bigg] . \label{so3:7}
}
The coefficients of the following series are central in $X^\tw(\mfo_3)$:
\equ{
\ms{z}(u) := d_1(u)\,d_3(-u) , \qu
\ms{c}(u) := \frac{2u-1}{2u}\,d_1(u)\,d_1^{-1}(\tfrac12-u) . \label{so3:8}
}
\end{crl}

Introduce the series
\equ{
b(u) := f_{21}(u-\tfrac14) , \qu 
h(u) := \tfrac{4u-1}{4u-2} \, d_1^{-1}(u-\tfrac14)\, d_2(u-\tfrac14) 
}
and note that the mapping \eqref{o3->gl2} maps these series as
\equ{
b(u) \mapsto \sqrt{2}\,b^\circ(2u), \qu 
h(u) \mapsto h^\circ(2u) .
}
This immediately yields the following statement.

\begin{thm} \label{T:iso-so3}
The special twisted Yangian $SY^{\tw}(\mfo_3)$ is isomorphic to the algebra generated by coefficients of the series $b(u)$ and $h(u)$ subject to relations \eqref{gl2:6}--\eqref{gl2:7} with $c=1$.
\end{thm}


\subsection{\texorpdfstring{Extended twisted Yangian $X^\tw(\mfo_4)$}{Extended twisted Yangian X\textasciicircum{}tw(o\_4)}}

Let $V = \C^2 \ot \C^2$ and let
\equ{
v_1 := e_1 \ot e_1, \qu
v_2 := e_1 \ot e_2, \qu
v_3 := e_2 \ot e_1, \qu
v_4 := -e_2 \ot e_2 
}
be a basis of $V$. Let $R(u)$ be the $R$-matrix of $X^\tw(\mfo_4)$. Then (see \cite[Lem.~4.9]{AMR06})
\ali{
\label{RR-o4}
R(u) \equiv \frac{u}{u-1}\,R^\circ_{13}(u)\,R^\circ_{24}(u) , \qu
R'(\ka-u) \equiv \frac{u-1}{u}\,R^{\circ\prime}_{13}(1-u)\,R^{\circ\prime}_{24}(1-u) , 
}
where $\ka=1$ and $\equiv$ denotes equality of linear maps upon identification of basis vectors $e_i \ot e_j \in \C^4 \ot \C^4$ with $v_i \ot v_j \in V \ot V$. The $G$-matrix is the identity matrix.

\begin{prop} \label{P:o4->gl2}
The mapping $X^{\tw}(\mfo_4) \to X^+(\mfgl_2) \ot X^+(\mfgl_2)$ given by             
\equ{
S(u) \mapsto S^\circ(u-\tfrac12)\ot S^\circ(u-\tfrac12) , \label{o4->gl2}
}
subject to the identification of basis vectors described above, defines an embedding of algebras.
In terms of Gaussian series it reads as
\gan{
\begin{aligned}
d_1(u) \mapsto d^\circ_1(u-\tfrac12)\ot d^\circ_1(u-\tfrac12) , \qu d_3(u) \mapsto d^\circ_2(u-\tfrac12)\ot d^\circ_1(u-\tfrac12) , \\
d_2(u) \mapsto d^\circ_1(u-\tfrac12)\ot d^\circ_2(u-\tfrac12) , \qu d_4(u) \mapsto d^\circ_2(u-\tfrac12)\ot d^\circ_2(u-\tfrac12) , 
\end{aligned}\\
\begin{aligned}
e_{12}(u) & \mapsto \phantom{-} 1\ot e^\circ_{12}(u-\tfrac12) , \qu e_{13}(u) \mapsto \phantom{-} e^\circ_{12}(u-\tfrac12)\ot 1 , \qu e_{14}(u) \mapsto -e^\circ_{12}(u-\tfrac12)\ot e^\circ_{12}(u-\tfrac12) , 
\\
e_{34}(u) &\mapsto -1\ot e^\circ_{12}(u-\tfrac12) , \qu
e_{24}(u) \mapsto -e^\circ_{12}(u-\tfrac12)\ot 1 , \qu
e_{23}(u) \mapsto 0 , 
\\
f_{21}(u) & \mapsto \phantom{-} 1\ot f^\circ_{21}(u-\tfrac12) , \qu
f_{31}(u) \mapsto \phantom{-} f^\circ_{21}(u-\tfrac12)\ot 1 , \qu
f_{41}(u) \mapsto -f^\circ_{21}(u-\tfrac12)\ot f^\circ_{21}(u-\tfrac12) , 
\\
f_{43}(u) &\mapsto -1\ot f^\circ_{21}(u-\tfrac12) , \qu
f_{42}(u) \mapsto -f^\circ_{21}(u-\tfrac12)\ot 1 , \qu
f_{32}(u) \mapsto 0 .
\end{aligned}
}
\end{prop}

\begin{proof}
The proof is analogous to that of Proposition \ref{P:o3->gl2}. The explicit form of the mapping \eqref{o4->gl2} is obtained using the method outlined in the proof of Proposition \ref{P:o3->gl2}. Next, applying \eqref{o4->gl2} and matrix identities \eqref{RR-o4} to \eqref{RE} immediately yields relations of the tensor algebra $X^+(\mfgl_2) \ot X^+(\mfgl_2)$. This proves that \eqref{o4->gl2} defines a homomorphism of algebras. It remains to prove that it is injective. 
Consider the associated graded algebras $\gr X^\tw(\mfo_4)$ and $\gr (X^+(\mfgl_2) \ot X^+(\mfgl_2))$.
A linear basis of $\gr X^\tw(\mfo_4)$ is given by ordered monomials in the elements $\bar{s}_{21}^{(r)}$, $\bar{s}_{31}^{(r)}$, $\bar{s}_{11}^{(r)}$, $\bar{s}_{22}^{(2r)}$, $\bar{s}_{33}^{(2r)}$ and that of $\gr(X^+(\mfgl_2) \ot X^+(\mfgl_2))$ is given by ordered monomials in the elements $\bar{s}_{21}^{\circ(r)}\ot 1$, $1\ot \bar{s}_{21}^{\circ(r)}$, $\bar{s}_{11}^{\circ(r)} \ot 1$, $1\ot \bar{s}_{11}^{\circ(r)}$, $\bar{s}^{\circ(2r)}_{22} \ot 1$, $1\ot \bar{s}^{\circ(2r)}_{22}$ with $r\ge 1$. The homomorphism \eqref{o4->gl2} descends to a linear map $\gr X^\tw(\mfo_4) \to \gr(X^+(\mfgl_2) \ot X^+(\mfgl_2))$ given by 
\gan{
\bar{s}^{(r)}_{21} \mapsto 1\ot\bar{s}^{\circ(r)}_{21} , \qu \bar{s}^{(r)}_{31} \mapsto \bar{s}^{\circ(r)}_{21}\ot1 , \qu \bar{s}^{(r)}_{11} \mapsto \bar{s}^{\circ(r)}_{11} \ot 1 + 1 \ot \bar{s}^{\circ(r)}_{11} , \\
\bar{s}^{(2r)}_{22} \mapsto \bar{s}^{\circ(2r)}_{11} \ot 1 + 1 \ot \bar{s}^{\circ(2r)}_{22} , \qu
\bar{s}^{(2r)}_{33} \mapsto \bar{s}^{\circ(2r)}_{22} \ot 1 + 1 \ot \bar{s}^{\circ(2r)}_{11} .
}
This mapping is injective, hence so is \eqref{o4->gl2}.
\end{proof}

Combining Proposition \ref{P:o4->gl2} with relations \eqref{gl2:1}--\eqref{gl2:5} we obtain the following statement.

\begin{crl} \label{C:so4} \allowdisplaybreaks
The following relations hold in $X^\tw(\mfo_4)$: 
\gat{
e_{12}(u) = - e_{34}(u) = - f_{43}(-u) = f_{21}(-u) , \qu
e_{13}(u) = - e_{24}(u) = - f_{42}(-u) = f_{31}(-u) , \label{so4:1} \\[0.2em]
e_{14}(-u) = f_{41}(u) = - f_{31}(u)\,f_{21}(u) , \qu e_{23}(u) = f_{32}(u) = 0, \label{so4:2}
\\[0.2em] 
d_1(u)\, d_4(-u) = d_1(-u)\, d_4(u) = d_2(u)\,d_3(-u) = d_2(-u)\,d_3(u) , \label{so4:3} \\[0.2em]
 [ d_i(u), d_j(v) ] = 0 \qu\text{for}\qu 1\le i,j\le 4, \label{so4:4}
\\
d_1(u+\tfrac12)\,f_{k1}(v) = \bigg[ \frac{(u-v-\frac12)(u+v+\frac12)}{(u-v+\frac12)(u+v-\frac12)}\,f_{k1}(v) + \bigg\{ \frac{u+\frac12}{u\,(u-v+\frac12)}\,f_{k1}(u+\tfrac12) \bigg\}^{\!u} \,\bigg]\, d_1(u+\tfrac12) , \label{so4:5}
\\[0.2em]
d_2(u-\tfrac12)\,f_{21}(v) = \bigg[ \frac{(u-v+\frac12)(u+v-\frac12)}{(u-v-\frac12)(u+v+\frac12)}\,f_{21}(v) - \bigg\{ \frac{u-\frac12}{u\,(u-v-\frac12)}\,f_{21}(u-\tfrac12) \bigg\}^{\!u} \,\bigg]\, d_2(u-\tfrac12) , \label{so4:6}
\\[0.2em]
d_2(u+\tfrac12)\,f_{31}(v) = \bigg[ \frac{(u-v-\frac12)(u+v+\frac12)}{(u-v+\frac12)(u+v-\frac12)}\,f_{31}(v) + \bigg\{ \frac{u+\frac12}{u\,(u-v+\frac12)}\,f_{31}(u+\tfrac12) \bigg\}^{\!u} \,\bigg]\, d_2(u+\tfrac12) , \label{so4:7}
\\[0.2em]
[f_{k1}(u), f_{\ell 1}(v)] = \del_{k\ell} \bigg[ \frac{1}{v-u}\, (f_{k1}(u)-f_{k1}(v))^2 + \frac{1}{u+v} ( d_1^{-1}(v)\,d_k(v) - d_1^{-1}(u)\,d_k(u) ) \bigg] \label{so4:8}
}
for $k,\ell=2,3$.
The coefficients of the following series are central in $X^\tw(\mfo_4)$:
\equ{
\ms{z}(u) := d_1(u)\,d_4(-u), \qu 
\ms{c}(u) := d_1(u)\, d_1^{-1}(1-u) . \label{so4:9}
}
\end{crl}

Introduce the series
\ali{
b_1(u) := f_{21}(u) , \qu b_2(u) := f_{31}(u) , \qu 
h_1(u) := d_1^{-1}(u)\, d_2(u) , \qu h_2(u) := d_1^{-1}(u)\, d_3(u) 
}
and note that the mapping \eqref{o4->gl2} maps these series as
\equ{
b_1(u) \mapsto 1 \ot b^\circ(u) , \qu 
h_1(u) \mapsto 1 \ot h^\circ(u) , \qu b_2(u) \mapsto b^\circ(u) \ot 1 , \qu h_2(u) \mapsto h^\circ(u) \ot 1 .
}
This immediately yields the following statement.

\begin{thm} \label{T:iso-so4}
The special twisted Yangian $SY^{\tw}(\mfo_4)$ is isomorphic to the algebra generated by coefficients of the series $b_i(u)$ and $h_j(u)$ with $i,j=1,2$ satisfying relations analogous to \eqref{gl2:6}--\eqref{gl2:7} with $c=2$ when $i=j$ and trivial relations when $i\ne j$.
\end{thm}


\section{The general case} \label{sec:gen}

We first use a subalgebra that is a homomorphic image of the extended twisted Yangian $X^+(\mfgl_n)$ and results of \cite{LWZ25a} to obtain relations of type AI in $X^\tw(\mfg_{N})$. We then express the central series associated with the symmetry and unitarity relations through Gaussian series, and derive relations of types BI, CI and DI; the auxiliary relations and evaluated commutators needed for the Serre type relations are deferred to Appendices \ref{app:Aux} and \ref{app:Comms}. The main results, the Drinfeld-type current relations of $SY^\tw(\mfg_{N})$ and $X^\tw(\mfg_{N})$, are presented in Sections \ref{sec:Dr} and \ref{sec:ext}. For Serre type relations we use the short-hand notation
\equ{
\Sym_{u_1,\ldots,u_k} f(u_1,\ldots, u_k) := \sum_{\si \in \Sigma_k} f(u_{\si(1)} , \ldots, u_{\si(k)}) 
\label{Sym}
}
for any function or operator $f(u_1,\ldots, u_k)$.


\subsection{Relations of type AI} \label{sec:gen-AI}

Restricting relations \eqref{[s,s]} to $1\le i,j,k,l\le n$ gives
\ali{ \label{[s,s]:AI}
[\,s_{ij}(u),s_{kl}(v)] &= \frac{1}{u-v}\Big(s_{kj}(u)\,s_{il}(v)-s_{kj}(v)\,s_{il}(u)\Big)\nn \\
{} &\qu - \frac{1}{u+v-\ka}\, \Big( s_{ik}(u)\, s_{jl}(v) - s_{ki}(v)\, s_{lj}(u) \Big) \nn \\
{} &\qu + \frac{1}{(u-v) (u+v-\ka)} \Big( s_{ki}(u)\, s_{jl}(v)-s_{ki}(v)\, s_{jl}(u) \Big) .
}
In other words, the upper left $n\times n$ submatrix $S^\circ(u)$ of $S(u)$ satisfies the twisted reflection equation
\equ{
\label{RE-A}
R^\circ(u-v)\,S^{\circ}_1(u)\,R^{\circ\prime}(\ka-u-v)\,S^{\circ}_2(v) = S^{\circ}_2(v)\, R^{\circ\prime}(\ka-u-v)\,S^{\circ}_1(u)\,R^\circ(u-v)
}
where $R^\circ(u) := I - u^{-1} P \in \End(\C^n \ot \C^n)[[u^{-1}]]$ and $R^{\circ\prime}(u)$ is its partial transpose. Comparing \eqref{RE-A} with \cite[(2.6)]{Mo07}, we see that the mapping $s_{ij}(u) \mapsto s_{ij}(u+\ka/2)$ defines a homomorphism from the extended twisted Yangian $X^+(\mfgl_n)$, that is, the algebra $X_G(\mfo_n)$ of \cite[Def.~2.13.1]{Mo07} with $G=I$, onto the subalgebra of $X^{\tw}(\mfg_N)$ generated by the coefficients of the series $s_{ij}(u)$ with $1\le i,j\le n$. (We will denote this subalgebra by $X^+_\ka(\mfgl_n)$.) This homomorphism allows us to immediately write down relations for Gaussian series of $S^\circ(u)$ by adapting those found in \cite{LWZ25a} (also see Lemma \ref{L:LWZ} in Section \ref{sec:Dr}). They are obtained by exploiting relations of $X^+(\mfgl_2)$ and $X^+(\mfgl_3)$ and embeddings $X^+(\mfgl_m) \into X^+(\mfgl_n)$. The latter include a shift by $\frac{n-m}{2}$ of the spectral parameter (see \cite[Lem~2.14.1]{Mo07}) responsible for the presence of $i$'s and $j$'s in the rational expressions below. Relation \eqref{AI:ff3} is the image of \cite[Lem.~4.6]{LWZ25a} under the anti-automorphism of \cite[Lem.~3.5]{LWZ25a}, which fixes the diagonal Gaussian series and exchanges $e_{ij}(u)$ with $f_{ji}(u)$.

\begin{lemma} \label{L:AI-rels}
The following relations hold in $X^\tw(\mfg_N)$, for $1\le i,j < n$ and $1\le k,l\le n$, 
\ali{
[d_k(u),d_l(v) ] &= 0 , \qu \text{and}\qu [d_k(u), f_{i+1,i}(v)] = 0 \qu \text{if}\qu k\ne i,i+1 , \label{AI:dd}
\\[0.2em]
d_i(u)\,f_{i+1,i}(v) &= \bigg(\frac{(u-v-1)(u+v-\ka+i)\,f_{i+1,i}(v)}{(u-v)(u+v-\ka+i-1)} + \frac{(2u-\ka+i)\, f_{i+1,i}(u)}{(2u-\ka+i-1)(u-v)} \nn \\ & \hspace{4.1cm} - \frac{(2u-\ka+i-2)\, f_{i+1,i}(\ka-i+1-u) }{(2u-\ka+i-1)(u+v-\ka+i-1)} \bigg) d_i(u) , \label{AI:df1}
\\[0.2em]
d_{i+1}(u)\,f_{i+1,i}(v) &= \bigg(\frac{(u-v+1)(u+v-\ka+i)\, f_{i+1,i}(v)}{(u-v)(u+v-\ka+i+1)} - \frac{(2u-\ka+i)\, f_{i+1,i}(u)}{(2u-\ka+i+1)(u-v)} \nn \\ & \hspace{4.1cm} + \frac{(2u-\ka+i+2)\, f_{i+1,i}(\ka-i-1-u)}{(2u-\ka+i+1)(u+v-\ka+i+1)} \bigg) d_{i+1}(u) , \label{AI:df2}
\\[0.2em]
[f_{i+1,i}(u), f_{i+1,i}(v)] &= \frac{(f_{i+1,i}(v) - f_{i+1,i}(u))^2}{v-u} + \frac{d^{-1}_{i}(v)\,d_{i+1}(v) - d^{-1}_{i}(u)\,d_{i+1}(u)}{v+u-\ka+i} \,, \label{AI:ff2} 
\intertext{and}
f_{i+1,i}(u)\,f_{j+1,j}(v) &= \frac{u-v+\tfrac{i-j-1}{2}}{u-v+\tfrac{i-j+1}{2}}\,f_{j+1,j}(v)\,f_{i+1,i}(u) \nn \\ & \qu + \frac{1}{u-v+\tfrac{i-j+1}{2}}\,([f^{(0)}_{i+1,i},f_{j+1,j}(v)] - [f_{i+1,i}(u),f_{j+1,j}^{(0)}]) , \label{AI:ff1}
\\[0.5em]
\Sym_{v,w} [f_{j+1,j}(w)&,[f_{j+1,j}(v),f_{i+1,i}(u)]] \nn \\ & = \frac{1}{v+w-\ka+j}\, \Sym_{v,w} \bigg( \frac{(2v-\ka+j)\,f_{i+1,i}(u)}{(u-v+\frac{i-j-1}{2})(u+v-\ka+\frac{i+j-1}2)} \nn \\ & \hspace{2.8cm} - \frac{f_{i+1,i}(v-\frac{i-j-1}2)}{u-v+\frac{i-j-1}2} + \frac{f_{i+1,i}(\ka-v-\frac{i+j-1}2)}{u+v-\ka+\frac{i+j-1}2} \bigg)\,d^{-1}_j(v)\,d_{j+1}(v) , \label{AI:fff}
}
if $|i-j|=1$, and 
\equ{
[f_{j+1,j}(u), f_{i+1,i}(v)] = 0 \qu\text{if}\qu |i-j|>1 . \label{AI:ff0}
}
Moreover, for $1\le i < n$,
\gat{
e_{i,i+1} (u) = f_{i+1,i}(\ka-i-u) , \label{AI:f=e} \\
d_i(u)\,d_{i+1}(\ka-i-u) = d_{i+1}(u)\,d_{i}(\ka-i-u) .\label{AI:dd=dd}
}
and, for $1\le i < n-1$,
\equ{
[f_{i+1,i}(u), f_{i+2,i+1}(v)] = \frac{1}{u-v}\,\big( f_{i+2,i+1}(v)\,(f_{i+1,i}(v) - f_{i+1,i}(u) ) + f_{i+2,i}(u) - f_{i+2,i}(v) \big) . \label{AI:ff3}
}
\end{lemma}

\enlargethispage{1.5em}

\begin{remark} \label{R:dif} 
The counterparts of \eqref{AI:df1} and \eqref{AI:df2} for the inverses $d_i^{-1}(u)$ and $d_{i+1}^{-1}(u)$ are
\ali{
d_i^{-1}(u)\,f_{i+1,i}(v) = \bigg(\frac{(u-v)(u+v-\ka+i-1)\,f_{i+1,i}(v)}{(u-v-1)(u+v-\ka+i)} & - \frac{(2u-\ka+i-2)\, f_{i+1,i}(u-1)}{(2u-\ka+i-1)(u-v-1)} \nn \\ & + \frac{(2u-\ka+i)\, f_{i+1,i}(\ka-i-u) }{(2u-\ka+i-1)(u+v-\ka+i)} \bigg) d_i^{-1}(u) , \label{AI:dif1}
\\[0.2em]
d_{i+1}^{-1}(u)\,f_{i+1,i}(v) = \bigg(\frac{(u-v)(u+v-\ka+i+1)\, f_{i+1,i}(v)}{(u-v+1)(u+v-\ka+i)} & + \frac{(2u-\ka+i+2)\, f_{i+1,i}(u+1)}{(2u-\ka+i+1)(u-v+1)} \nn \\ & - \frac{(2u-\ka+i)\, f_{i+1,i}(\ka-i-u)}{(2u-\ka+i+1)(u+v-\ka+i)} \bigg) d_{i+1}^{-1}(u) . \label{AI:dif2}
}
\end{remark}


\subsection{Central series} \label{sec:central}

The extended twisted Yangian $X^\tw(\mfg_N)$ has two central series, denoted $\ms{c}(u)$ and $\ms{z}(u)$, that are associated with the symmetry and unitarity relations. For low rank cases these series were identified in Section \ref{sec:low}; here we express them in terms of the diagonal Gaussian series for any rank. 

Define a series $\ms{c}(u) = 1 + \sum_{r\ge 0}\ms{c}_r u^{-r-1} \in X^\tw(\mfg_N)[[u^{-1}]]$ by
\ali{
\label{c(u)}
\ms{c}(u) := p(u)\,d_1(u)\,d_1^{-1}(\ka-u) \qu\text{where}\qu
p(u) := \begin{cases}
\dfrac{(u-\frac14)(u-\ka)}{u\,(u-\ka+\frac14)} & \text{when } \mfg_N=\mfo_{2n+1}, \\
1 & \text{otherwise}.
\end{cases}
}
The series $p(u)$ is such that the $(1,1)$ entry of the symmetry relation \eqref{symij} is equivalent to the requirement $\ms{c}(u)=1$.

The coefficients of $\ms{c}(u)$ are central in $X^{\tw}(\mfg_N)$. For $\mfg_N=\mfsp_2$, $\mfo_3$ and $\mfo_4$ this is part of Corollaries \ref{C:sp2}, \ref{C:so3} and \ref{C:so4}. Otherwise, $n\ge2$ in types BI and CI and $n\ge3$ in type DI, and $X^{\tw}(\mfg_N)$ is generated by the coefficients of $d_1(u)$, $f_{21}(u)$ and the subalgebra $X^{\tw[1]}(\mfg_N)$, by Lemma \ref{L:gen}\,(2) below, whose proof uses only the isomorphism \eqref{X-graded-iso}. The coefficients of $\ms{c}(u)$ commute with those of $d_1(u)$ by \eqref{AI:dd} and with $X^{\tw[1]}(\mfg_N)$ by Lemma \ref{L:SS(u)}\,(1), and combining \eqref{AI:df1} at $u$ with \eqref{AI:dif1} at $\ka-u$, both with $i=1$, shows that conjugation by $d_1(u)\,d^{-1}_1(\ka-u)$ fixes $f_{21}(v)$.

Next, adapting \cite[Prop.~4.1]{GR16}, whose proof uses only the reflection equation, we find that the unitarity relation of $X^{\tw}(\mfg_N)$ is
\equ{
JS(-u)J = \ms{z}(u)\,S^{-1}(u) \label{S=zS}
}
where $\ms{z}(u)= 1 + \sum_{r\ge 0}\ms{z}_r u^{-r-1} \in X^\tw(\mfg_N)[[u^{-1}]]$ is a series with coefficients central in $X^{\tw}(\mfg_N)$. Applying \eqref{S=FDE}, \eqref{Si=FDEi} and taking the bottom-right entry on both sides of \eqref{S=zS} gives
\equ{
\theta\,d_1(-u) = \ms{z}(u)\,d^{-1}_{N}(u) . \label{d=zd}
}
Note that $\ms{c}(\ka-u) = \ms{c}^{-1}(u)$ and $\ms{z}(-u) = \ms{z}(u)$. To identify the coefficients of these series, write $S(u)=\ms{y}(u)\,\wt S(u)$ as before \eqref{X-graded-iso}, where $\ms y(u)=1+\sum_{r\ge0}\ms y_r u^{-r-1}$ is central and the coefficients of $\wt S(u)$ generate a subalgebra of $X^\tw(\mfg_N)$ isomorphic to $SY^\tw(\mfg_N)$. By the uniqueness of the Gaussian decomposition, $S(u)$ and $\wt S(u)$ have the same matrices $F(u)$ and $E(u)$, while $d_j(u)=\ms y(u)\,\wt d_j(u)$ for all $j$, where the $\wt d_j(u)$ are the diagonal Gaussian series of $\wt S(u)$. Since $\wt S(u)$ satisfies the symmetry and unitarity relations, \eqref{c(u)} and \eqref{S=zS} give
\[
\ms c(u) = \ms y(u)\,\ms y^{-1}(\ka-u) , \qq \ms z(u) = \ms y(u)\,\ms y(-u) .
\]
Comparing coefficients, $\ms c_{2r}-2\,\ms y_{2r}$ and $\ms z_{2r+1}-2\,\ms y_{2r+1}$ are polynomials in the elements $\ms y_s$ with $s<2r$ and $s<2r+1$, respectively. Hence the elements $\ms c_{2r}$ and $\ms z_{2r+1}$ with $r\ge0$ are algebraically independent and generate the same algebra as the elements $\ms y_r$. Since $X^\tw(\mfg_N)$ is the tensor product of the subalgebra generated by the coefficients of $\wt S(u)$ and the polynomial algebra in the elements $\ms y_r$, and $SY^\tw(\mfg_N)$ has trivial centre by \cite[Cor.~3.5]{GR16}, the even coefficients of $\ms{c}(u)$ and the odd coefficients of $\ms{z}(u)$ generate the centre of $X^\tw(\mfg_N)$. More precisely,
\equ{
X^\tw(\mfg_N) \cong SY^\tw(\mfg_N) \ot \C[\ms{c}_0, \ms{c}_2, \ldots] \ot \C[\ms{z}_1, \ms{z}_3, \ldots] .
\label{X=SYcz}
}

\smallskip

We next record relations between the off-diagonal Gaussian series that follow from \eqref{S=zS}, and then show that an analogue of \eqref{d=zd} holds for all diagonal Gaussian series, viz.~\eqref{z(u)}.

\begin{lemma} \label{L:fe}
The following relations hold in $X^{\tw}(\mfg_N)$:
\equ{
f_{N-i+1,N-i}(u) = - e_{i,i+1}(-u), \qu 
e_{N-i,N-i+1}(u) = - f_{i+1,i}(-u) \label{f-e:1}
}
for $1\le i < n$. Furthermore, if $\mfg_{N} = \mfsp_{2n}$, then
\equ{
f_{n+1,n}(u) = e_{n,n+1}(-u) \label{f-e:2}
}
and if $\mfg_{N} = \mfo_{2n+1}$, then
\equ{
-f_{n+2,n+1}(u+\tfrac12) = f_{n+1,n}(u) = - e_{n+1,n+2}(-u) = e_{n,n+1}(-u-\tfrac12) \label{f-e:3}
}
and if $\mfg_{N} = \mfo_{2n}$, then
\ali{
-f_{n+2,n+1}(u) = f_{n,n-1}(u) & = e_{n-1,n}(-u) = -e_{n+1,n+2}(-u) , \label{f-e:4} \\ 
-f_{n+2,n}(u) = f_{n+1,n-1}(u) & = e_{n-1,n+1}(-u) = - e_{n,n+2}(-u) , \label{f-e:5} \\
f_{n+1,n}(u) &= e_{n,n+1}(u) = 0 . \label{f-e:6}
}
\end{lemma}

\begin{proof}
Consider the algebra $X^{\tw[i-1]}(\mfg_N)$ with $1\le i < n$. Corollary \ref{C:RE[m]} implies that
\equ{
J^{[i-1]} S^{[i-1]}(-u) J^{[i-1]} = \ms{z}^{[i-1]}(u)\,(S^{[i-1]}(u))^{-1} , \label{S=zS[m]}
}
where $\ms{z}^{[i-1]}(u)$ denotes the image of the central series $\ms{z}(u)$ of $X^\tw(\mfg_{N-2i+2})$ under the isomorphism of Corollary \ref{C:RE[m]}. Taking the bottom-right entry on both sides gives
\equ{
\theta\,d_i(-u) = \ms{z}^{[i-1]}(u)\,d^{-1}_{N-i+1}(u) . \label{d=zd[m]}
}
Next, taking the entry to the left of the bottom-right entry on both sides of \eqref{S=zS[m]} gives 
\[
\theta\,d_i(-u)\,e_{i,i+1}(-u)\,= - \ms{z}^{[i-1]}(u)\,d^{-1}_{N-i+1}(u)\,f_{N-i+1,N-i}(u)
\]
which, upon combining with \eqref{d=zd[m]}, gives the first relation in \eqref{f-e:1}. The second relation in \eqref{f-e:1} is obtained by taking the entry above the bottom-right entry on both sides of \eqref{S=zS[m]} and combining with \eqref{d=zd[m]}. The remaining relations follow from \eqref{gl2:1} and Propositions \ref{P:sp2->gl2}, \ref{P:o3->gl2} and \ref{P:o4->gl2}.
\end{proof}

\begin{prop} \label{P:z(u)}
The following relations hold in $X^{\tw}(\mfg_N)$:
\equ{
\ms{z}(u) = \theta\,d_i(-u)\,d_{N-i+1}(u) \qu\text{for}\qu 1 \le i \le N-n .
\label{z(u)}
}
\end{prop}

\begin{proof}
Conjugation by $J$ reverses the order of rows and columns. Hence, by \eqref{S=FDE} and \eqref{Si=FDEi}, relation \eqref{S=zS} reads
\[
J F(-u) J^{-1} \cdot J D(-u) J \cdot J^{-1} E(-u) J = E^{-1}(u) \cdot \ms{z}(u)\,D^{-1}(u) \cdot F^{-1}(u) ,
\]
where on both sides the first factor is upper unitriangular, the second is diagonal and the third is lower unitriangular. Such a decomposition is unique: if $U_1 D_1 L_1 = U_2 D_2 L_2$ with $U_k$ upper unitriangular, $D_k$ diagonal and invertible, and $L_k$ lower unitriangular, then $U_2^{-1} U_1 D_1 = D_2 L_2 L_1^{-1}$ is both upper and lower triangular, and comparing diagonals gives $D_1=D_2$. Therefore
\equ{
J D(-u) J = \ms{z}(u)\,D^{-1}(u) . \label{d=zd:2}
}
Comparing the diagonal entries on both sides, we find
\equ{
\theta\,d_{N-i+1}(-u) = \ms{z}(u)\,d^{-1}_{i}(u) \qu\text{for}\qu 1\le i\le N ,
}
which, after replacing $i$ with $N-i+1$ and multiplying by $d_{N-i+1}(u)$ from the right, gives \eqref{z(u)}.
\end{proof}

We close this subsection with a consequence of \eqref{f-e:6} that will be needed later. In type DI the two nodes $n-1$ and $n$ of the Dynkin diagram are exchanged by a diagram automorphism. Since $f_{n+1,n}(u)=e_{n,n+1}(u)=0$, this automorphism is realised on $X^\tw(\mfo_{2n})$ by a permutation of the Gaussian series. It allows us to obtain the relations at the node $n$ from those at the node $n-1$, which are of type AI; see Lemmas \ref{L:ff} and \ref{L:heven}, Proposition \ref{P:Serre} and the proof of Theorem \ref{T:iso}.

\begin{lemma} \label{L:fork}
Let $\mfg_N=\mfo_{2n}$ with $n\ge2$, and let $A$ be the permutation matrix of the transposition $(n\;\,n+1)$. The mapping $\si : S(u)\mapsto A\,S(u)\,A^t$ defines an involutive automorphism of $X^\tw(\mfo_{2n})$. It fixes the series $d_k(u)$ with $k\ne n,n+1$ and $f_{k+1,k}(u)$ with $1\le k\le n-2$, and
\equ{
\si(d_n(u))=d_{n+1}(u) , \qq \si(f_{n,n-1}(u)) = f_{n+1,n-1}(u) . \label{fork}
}
\end{lemma}

\begin{proof}
Since $\bar n=n+1$, the matrix $A$ is orthogonal and $AJA^t=J$. As $G=I$, relation \eqref{AARAA} shows that $A\,S(u)\,A^t$ satisfies \eqref{RE} and has constant term $I$; hence $\si$ is a homomorphism, and $\si^2=\id$. Write $\pi$ for the transposition $(n\;\,n+1)$. By \eqref{f-e:6}, $f_{n+1,n}(u)=e_{n,n+1}(u)=0$, so the matrices $A\,F(u)\,A^t$ and $A\,E(u)\,A^t$, with entries $f_{\pi(i)\pi(j)}(u)$ and $e_{\pi(i)\pi(j)}(u)$, are again lower and upper unitriangular, respectively, while $A\,D(u)\,A^t$ is diagonal with entries $d_{\pi(i)}(u)$. By the uniqueness of the Gaussian decomposition of 
\[
A\,S(u)\,A^t = A\,F(u)A^t\,A\,D(u)A^t\,A\,E(u)A^t,
\]
we have $\si(d_i(u))=d_{\pi(i)}(u)$ and $\si(f_{ji}(u))=f_{\pi(j)\pi(i)}(u)$, which gives the claim.
\end{proof}


\subsection{Relations of types BI, CI and DI} \label{sec:BCD}

We will now derive relations involving series $f_{n+1,n}(u)$ and $d_{n+1}(u)$ in the case of $X^\tw(\mfo_{2n+1})$, series $f_{n+1,n}(u)$ in the case of $X^\tw(\mfsp_{2n})$, and series $f_{n+1,n-1}(u)$ in the case~of~$X^\tw(\mfo_{2n})$.
By the quasi-determinant formulas for the Gaussian series, in the notation \eqref{si[m]}, we have $d_{n+1}(u)=\si^{[n]}_{n+1,n+1}(u)$ in type BI and $f_{ji}(u)=\si^{[i-1]}_{ji}(u)\,d^{-1}_i(u)$ for $i\ge2$, while $d_i(u)$ and $f_{i+1,i}(u)$ are expressed through the series $s_{ab}(u)$ with $1\le a,b\le i$ and $1\le a,b\le i+1$, respectively. Hence Lemma \ref{L:emb-comm} implies the following trivial relations: take $m=n$ for \eqref{B:11}--\eqref{B:12}, $m=n-1$ for \eqref{B:13}--\eqref{B:14} and in type CI, and $m=n-2$ in type DI, whenever the corresponding range of indices is nonempty.

\begin{lemma} \label{L:BCD-triv}
In the algebra $X^{\tw}(\mfo_{2n+1})$ we have:
\ali{
[d_{n+1}(u), d_{i}(v)] &= 0 \qu \text{for}\qu 1 \le i \le n , \label{B:11}\\
[d_{n+1}(u), f_{i+1,i}(v)] &= 0 \qu \text{for}\qu 1 \le i < n , \label{B:12}\\
[f_{n+1,n}(u), d_{i}(v)] &= 0 \qu \text{for}\qu 1 \le i < n , \label{B:13}\\
[f_{n+1,n}(u), f_{i+1,i}(v)] &= 0 \qu \text{for}\qu 1 \le i < n - 1 . \label{B:14}
\intertext{In the algebra $X^{\tw}(\mfsp_{2n})$ we have:}
[f_{n+1,n}(u), d_{i}(v)] &= 0 \qu \text{for}\qu 1 \le i < n ,\\
[f_{n+1,n}(u), f_{i+1,i}(v)] &= 0 \qu \text{for}\qu 1 \le i < n - 1 .
\intertext{In the algebra $X^{\tw}(\mfo_{2n})$ we have:}
[f_{n+1,n-1}(u), d_{i}(v)] &= 0 \qu \text{for}\qu 1 \le i < n-1 ,\\
[f_{n+1,n-1}(u), f_{i+1,i}(v)] &= 0 \qu \text{for}\qu 1 \le i < n - 2 . 
}
\end{lemma}

Next, we state relations involving diagonal type Gaussian series that follow by combining Corollary \ref{C:RE[m]} with Propositions \ref{P:sp2->gl2}, \ref{P:o3->gl2}~and~\ref{P:o4->gl2}.

\begin{lemma} \label{L:BCD-df} 
\allowdisplaybreaks
In the algebra $X^{\tw}(\mfo_{2n+1})$ we have:
\ali{
[d_{i}(u),d_{n+1}(v)] &= 0 \qu\text{for}\qu i=n,n+1, \label{dd1} \\[0.2em] 
d_{n}(u)\, f_{n+1,n}(v) &= \bigg[\frac{(u - v - 1) (u + v + \frac12)}{(u - v) (u + v - \frac12)} \, f_{n+1,n}(v) \nn \\ & \qq + \frac{u + \frac14}{(u - \frac14) (u - v)}\, f_{n+1,n}(u) - \frac{u - \frac34}{(u - \frac14) (u + v - \frac12)}\, f_{n+1,n}(\tfrac12-u) \bigg]\, d_n(u) , \label{hf1} \\[0.4em]
d_{n+1}(u)\, f_{n+1,n}(v) &= \bigg[\frac{(u + v) (u - v - 1) (u - v + \frac12) (u + v + \frac12) }{(u - v) (u + v + 1) (u - v - \frac12) (u + v - \frac12)} \, f_{n+1,n}(v) \nn\\ 
&\qq - \frac{u}{u - \frac14}\bigg( \frac{u + \frac14}{(u + \frac12) (u - v)}\, f_{n+1,n}(u) + \frac{u - \frac34}{3 (u - \frac12) (u + v - \frac12)}\, f_{n+1,n}(\tfrac12-u) \bigg) \nn\\ 
&\qq + \frac{u}{u + \frac14}\bigg(\frac{u - \frac14}{(u - \frac12) (u - v - \frac12)}\, f_{n+1,n}(u - \tfrac12) \nn \\ & \hspace{4.6cm} + \frac{u + \frac34}{3 (u + \frac12) (u + v + 1)}\, f_{n+1,n}(-u - 1) \bigg) \bigg]\, d_{n+1}(u) . \label{hf2}
\intertext{In the algebra $X^{\tw}(\mfsp_{2n})$ we have:}
d_{n}(u)\, f_{n+1,n}(v) &= \bigg[ \frac{(u + v) (u - v - 2)}{(u - v) (u + v - 2)} \, f_{n+1,n}(v) \nn \\ & \qq + \frac{2 u}{(u - 1) (u - v)}\, f_{n+1,n}(u) - \frac{2 (u - 2)}{(u - 1) (u + v - 2)} f_{n+1,n}(2-u) \bigg]\, d_{n}(u) . \label{hf3}
\intertext{In the algebra $X^{\tw}(\mfo_{2n})$ we have, for $i=n,n-1$:}
d_{i}(u)\, f_{n+1,n-1}(v) &= \bigg[ \frac{(u + v) (u - v - 1)}{(u - v) (u + v - 1)} \, f_{n+1,n-1}(v) \nn \\ & \qq + \frac{u}{(u - \frac12) (u - v)} \, f_{n+1,n-1}(u) - \frac{u - 1}{(u - \frac12)(u + v - 1)} \, f_{n+1,n-1}(1-u) \bigg]\, d_{i}(u) . \label{hf4}
}
\end{lemma}

\begin{proof}
Consider the algebra $X^{\tw}(\mfo_{2n+1})$. By Corollary \ref{C:RE[m]} and \eqref{D[m]}--\eqref{S[m]}, the subalgebra $X^{\tw[n-1]}(\mfo_{2n+1})$ is isomorphic to $X^{\tw}(\mfo_3)$, and its Gaussian series are those of $X^{\tw}(\mfo_{2n+1})$ with indices $n,n+1,n+2$ in place of $1,2,3$. Relations \eqref{dd1}, \eqref{hf1} and \eqref{hf2} are therefore the images of \eqref{so3:4}, of \eqref{so3:5} with $u\mapsto u-\tfrac14$, and of \eqref{so3:6}, respectively. In the same way, using $X^{\tw[n-1]}(\mfsp_{2n})\cong X^{\tw}(\mfsp_2)$ and $X^{\tw[n-2]}(\mfo_{2n})\cong X^{\tw}(\mfo_4)$, relation \eqref{hf3} is the image of \eqref{sp2:2} with $u\mapsto u-1$, and relation \eqref{hf4} for $i=n-1$ and $i=n$ is the image of \eqref{so4:5} with $k=3$ and of \eqref{so4:7}, respectively, both with $u\mapsto u-\tfrac12$.
\end{proof}

We state relations involving Gaussian series associated with the $n$-th simple root of $\mfg_N$. Here and in Propositions \ref{P:ff0} and \ref{P:Serre} below we assume $n\ge2$ in types BI and CI, and $n\ge3$ in type DI, whenever the series $f_{n,n-1}(u)$, resp.~$f_{n-1,n-2}(u)$, occurs; the cases $\mfg_N=\mfo_3$, $\mfsp_2$ and $\mfo_4$ were treated in Section~\ref{sec:low}.

\begin{lemma} \label{L:ff} \allowdisplaybreaks
In the algebra $X^{\tw}(\mfo_{2n+1})$ we have:
\ali{
[f_{n+1,n}(u), f_{n+1,n}(v)] &= \frac{\frac12}{v-u} \, (f_{n+1,n}(u) - f_{n+1,n}(v))^2 \nn\\ & \qu + \frac{1}{v+u+\frac12}\,\bigg( \frac{v}{v-\frac14}\,d^{-1}_n(v)\,d_{n+1}(v) - \frac{u}{u-\frac14}\,d^{-1}_n(u)\,d_{n+1}(u) \bigg) , \label{B:ff1} \\[0.3em]
[f_{n,n-1}(u), f_{n+1,n}(v)] &= \frac{1}{u-v}\,\big( f_{n+1,n}(v)\,(f_{n,n-1}(v) - f_{n,n-1}(u) ) + f_{n+1,n-1}(u) - f_{n+1,n-1}(v) \big) . \label{B:ff2} 
\intertext{In the algebra $X^{\tw}(\mfsp_{2n})$ we have:}
[f_{n+1,n}(u), f_{n+1,n}(v)] &= \frac{2}{v-u} \, (f_{n+1,n}(v) - f_{n+1,n}(u))^2 \nn\\ & \qu + \frac{2}{u+v}\,\big( d^{-1}_n(v)\,d_{n+1}(v) - d^{-1}_n(u)\,d_{n+1}(u) \big) , \label{C:ff1} \\[0.3em]
[f_{n,n-1}(u), f_{n+1,n}(v)] &= \frac{2}{u-v}\,\big( f_{n+1,n}(v)\,(f_{n,n-1}(v) - f_{n,n-1}(u) ) + f_{n+1,n-1}(u) - f_{n+1,n-1}(v) \big) . \label{C:ff2}
\intertext{In the algebra $X^{\tw}(\mfo_{2n})$ we have:} 
[f_{n,n-1}(u), f_{n+1,n-1}(v)] &= 0 , \label{D:ff1} \\[0.3em]
[f_{n+1,n-1}(u), f_{n+1,n-1}(v)] &= \frac{1}{v-u} \, (f_{n+1,n-1}(v) - f_{n+1,n-1}(u))^2 \nn\\ & \qu + \frac{1}{u+v}\,\big( d^{-1}_{n-1}(v)\,d_{n+1}(v) - d^{-1}_{n-1}(u)\,d_{n+1}(u) \big) , \label{D:ff2} \\[0.3em]
[f_{n-1,n-2}(u), f_{n+1,n-1}(v)] &= \frac{1}{u-v}\,\big( f_{n+1,n-1}(v)\,(f_{n-1,n-2}(v) - f_{n-1,n-2}(u) ) \nn\\[0.2em] & \hspace{6.45cm} + f_{n+1,n-2}(u) - f_{n+1,n-2}(v) \big) . \label{D:ff3}
}
\end{lemma}

\begin{proof}
{\it Type BI.} By Corollary \ref{C:RE[m]}, the subalgebra $X^{\tw[n-1]}(\mfo_{2n+1})$ is isomorphic to $X^\tw(\mfo_3)$. Hence, the Gaussian series $f_{n+1,n}(u)$, $d_{n}(u)$, $d_{n+1}(u)$ and $e_{n,n+1}(u)$ of $X^\tw(\mfo_{2n+1})$ obey relations described in Corollary \ref{C:so3}. In particular, the wanted relation \eqref{B:ff1} follows from \eqref{so3:7}. To obtain relation \eqref{B:ff2} it is sufficient to consider the subalgebra $X^{\tw[n-2]}(\mfo_{2n+1})$ which, by Corollary~\ref{C:RE[m]}, is isomorphic to $X^\tw(\mfo_5)$. 
Defining relations \eqref{[s,s]} of $X^\tw(\mfo_5)$ imply
\ali{
& [s_{21}(u), s_{32}(v) ] - [s_{21}(u), s_{31}(v) ]\,s^{-1}_{11}(v)\,s_{12}(v) \nn \\ & \qq = \frac{1}{u+v-\frac32}\,s_{32}(v)\,s_{21}(u) + \frac{s_{31}(u)\,s_{22}(v) - s_{31}(v)\,s_{22}(u)}{u-v} - \frac{s_{32}(v)\,s_{12}(u)}{(u-v)(u+v-\frac32)} \nn\\
& \qq\qu -\bigg( \frac{s_{31}(u)\,s_{21}(v) - s_{31}(v)\,s_{21}(u)}{u-v} + \frac{(u-v-1)\,s_{32}(v)\,s_{11}(u)}{(u-v)(u+v-\frac32)} \bigg) s_{11}^{-1}(v)\,s_{12}(v) . \label{B:ff2-1}
}
The series $s_{ij}(u)$ with $1\le i,j\le 2$ satisfy defining relations of $X^+_{\frac32}(\mfgl_2)$. Applying Gaussian decomposition \eqref{S=FDE} and the commutativity property \eqref{AI:dd} of diagonal Gaussian series to \eqref{B:ff2-1} we obtain
\ali{
& f_{21}(u)\,d_1(u)\,f_{32}(v)\,d_2(v) - \frac{f_{31}(u)\,d_2(v)\,d_1(u)-f_{31}(v)\,d_2(u)\,d_1(v)-f_{31}(v)\,d_1(v)\,f_{21}(u)\,d_1(u)\,e_{12}(u)}{u-v} \nn
\\
& + \frac{1}{(u-v)(u+v-\frac32)} \Big( \big(f_{32}(v)\,d_2(v)+f_{31}(v)\,d_1(v)\,e_{12}(v) \big)\,d_1(u)\,e_{12}(u) \nn \\ 
&\qq + (u-v-1) \big( f_{32}(v)\,d_2(v) + f_{31}(v)\,d_1(v)\,e_{12}(v) + (u+v-\tfrac{3}{2})\,f_{31}(v)\,d_1(v)\,f_{21}(u) \big)\,d_1(u)\,e_{12}(v) \Big) \nn\\
& - \frac{u+v-\frac12}{u+v-\frac32}\,\big( f_{32}(v)\,d_2(v) + f_{31}(v)\,d_1(v)\,e_{12}(v) \big)\,f_{21}(u)\,d_1(u) = 0 \label{B:ff2-2}
}
Next, we use \eqref{AI:f=e} to replace $e_{12}(u)$ with $f_{21}(\tfrac12-u)$ and $e_{12}(v)$ with $f_{21}(\tfrac12-v)$, and \eqref{AI:df1}--\eqref{AI:df2}, together with the commutativity of $d_1(u)$ and $f_{32}(v)$, see \eqref{B:13}, to move all diagonal Gaussian series rightwards. By doing so we obtain
\ali{
\bigg( f_{21}(u)\,f_{32}(v) & -\frac{f_{32}(v)\,f_{21}(v) + (u-v-1)\,f_{32}(v)\,f_{21}(u)}{u-v} \bigg)\, d_1(u)\,d_2(v) \nn \\[0.2em]
& = \frac{f_{31}(u)\,d_1(u)\,d_2(v) - f_{31}(v)\,d_1(v)\,d_2(u) - f_{31}(v)\, A\, d_1(u)\,d_1(v)}{u-v} 
\label{B:ff2-3}
}
where 
\ali{
A & = \frac{(u-v+1)^2(u+v-\frac12)}{(u-v)(u+v-\frac32)} \bigg( \frac{f_{21}(u)\,f_{21}(u) }{u+v-\frac32} - \frac{f_{21}(v)\,f_{21}(u) }{2\,(v-\frac34)}\bigg) \nn
\\
& \qu + \frac{(u+v-\frac12)}{(v-\frac34)(u-v)}\bigg( \frac{(u-v-1)(u-v+1)\,f_{21}(u)\,f_{21}(v)}{2\,(u+v-\frac32)} + \frac{f_{21}(v)\,f_{21}(v)}{4\,(v-\frac34)} \bigg) \nn
\\
& \qu + \frac{(v-\frac54)(u-v+1)(u+v-\frac12)}{(v-\frac34)(u+v-\frac32)^2} \bigg( (u+v-\tfrac52)\,f_{21}(u)\,f_{21}(\tfrac32-v) - (u+v-\tfrac12)\,f_{21}(\tfrac32-v)\,f_{21}(u) \bigg) \nn
\\
& \qu + \frac{(v-\frac54)(u+v-\frac12)}{(v-\frac34)^2(u+v-\frac32)} \bigg( (v-\tfrac14)\,f_{21}(\tfrac32-v)\,f_{21}(v) - (v-\tfrac54)\,f_{21}(v)\,f_{21}(\tfrac32-v) \bigg) \nn
\\
& \qu + \frac{(v-\frac54)^2(u-v)(u+v-\frac12)}{(v-\frac34)^2(u+v-\frac32)^2}\,f_{21}(\tfrac32-v)\,f_{21}(\tfrac32-v) . \label{B:ff2-4}
}
Applying \eqref{AI:ff2} to \eqref{B:ff2-4} gives $A = d_2(v)\,d_1^{-1}(v) - d_2(u)\,d_1^{-1}(u)$.
Substituting this last expression into \eqref{B:ff2-3} yields
\ali{
\bigg(f_{21}(u)\,f_{32}(v) - \frac{(u-v-1)\,f_{32}(v)\,f_{21}(u) + f_{32}(v)\,f_{21}(v) + f_{31}(u) - f_{31}(v)}{(u-v)} \bigg)\, d_1(u)\,d_2(v) = 0
}
which, by Corollary \ref{C:RE[m]}, implies the wanted relation \eqref{B:ff2}.


{\it Type CI.} By the same arguments as above, $X^{\tw[n-1]}(\mfsp_{2n})\cong X^\tw(\mfsp_2)$ and the wanted relation \eqref{C:ff1} follows from its counterpart \eqref{sp2:4}. 
The wanted relation \eqref{C:ff2} is obtained by computing the expression
\[
[s_{21}(u), s_{32}(v) ] - [s_{21}(u), s_{31}(v) ]\,s^{-1}_{11}(v)\,s_{12}(v)
\]
in the algebra $X^\tw(\mfsp_4)$ by means of \eqref{[s,s]}, applying Gaussian decomposition \eqref{S=FDE}, and using the commutativity of $d_1(u)$ and $d_2(v)$, see \eqref{AI:dd}. This gives
\ali{
& \frac{f_{31}(v)\,d_1(v)\,f_{21}(u)\,d_1(u)\,\big(e_{12}(u) + (u-v-1)\,e_{12}(v)\big) -f_{31}(u)\,d_1(u)\,d_2(v)+f_{31}(v)\,d_1(v)\,d_2(u)}{u-v} \nn
\\
& + \frac{1}{(u-v)(u+v-3)} \bigg( \big(f_{32}(v)\,d_2(v)+f_{31}(v)\,d_1(v)\,e_{12}(v)\big)\,d_1(u)\,e_{12}(u) \nn \\
& \hspace{3.5cm} + (u-v-1)\big(f_{32}(v)\,d_2(v)+f_{31}(v)\,d_1(v)\,e_{12}(v)\big)\,d_1(u)\,e_{12}(v) \bigg) \nn
\\
& - \frac{1}{u-v-3}\bigg( f_{31}(u)\,d_2(v)\,d_1(u)-(u-v-2)\,f_{21}(u)\,d_1(u)\,f_{32}(v)\,d_2(v) \nn \\
& \hspace{3.5cm} - \frac{d_1(u)\,\big(e_{13}(u)+(u+v-1)\,f_{42}(v)-e_{12}(u)\,f_{32}(v)\big)\,d_2(v)}{(u+v)} \bigg) \nn 
\\
& - \frac{u+v-2}{u+v-3}\,\big(f_{32}(v)\,d_2(v)+f_{31}(v)\,d_1(v)\,e_{12}(v)\big)\,f_{21}(u)\,d_1(u) = 0 . \label{C:ff2-2}
}
We next record three further consequences of \eqref{RE} in $X^\tw(\mfsp_4)$. Taking the coefficients of $u^{-1}$ in \eqref{[s,s]} with $(i,j,k,l)=(3,2,2,1)$ and $(3,2,1,1)$ gives $[s^{(1)}_{32},s_{21}(v)]=2\,s_{31}(v)$ and $[s^{(1)}_{32},s_{11}(v)]=0$. Since $f^{(0)}_{32}=s^{(1)}_{32}$, it follows that
\equ{
f_{31}(v) = \tfrac{1}{2}\,[f^{(0)}_{32},f_{21}(v)] . \label{C:f31}
}
Likewise, taking $(i,j,k,l)=(2,3,1,2)$ and $(2,3,1,1)$ gives $[s^{(1)}_{23},s_{12}(v)]=-2\,s_{13}(v)$ and $[s^{(1)}_{23},s_{11}(v)]=0$, whence $e_{13}(v)=-\frac12\,[s^{(1)}_{23},e_{12}(v)]$. By Lemma \ref{L:deg0} below, whose proof uses only \eqref{X-graded-iso}, we have $s^{(1)}_{23}=-s^{(1)}_{32}$. Hence \eqref{AI:f=e} and \eqref{C:f31} give $e_{13}(u)=\frac12\,[f^{(0)}_{32},f_{21}(2-u)]=f_{31}(2-u)$. Finally, $f_{32}(v)$ and hence $f^{(0)}_{32}$ commute with $d_1(u)$ by Lemma \ref{L:BCD-triv}; by \eqref{C:f31}, conjugation by $d_1(u)$, resp.~$d^{-1}_1(u)$, therefore acts on $f_{31}(v)$ by the same formula \eqref{AI:df1}, resp.~\eqref{AI:dif1}, as on $f_{21}(v)$.

We now replace $e_{13}(u)$ with $f_{31}(2-u)$, and, using \eqref{AI:f=e}, $e_{12}(u)$ with $f_{21}(2-u)$ and $e_{12}(v)$ with $f_{21}(2-v)$. We then use \eqref{AI:df1}--\eqref{AI:df2}, the above action of $d_1(u)$ on $f_{31}(v)$ and the commutativity of $d_1(u)$ and $f_{32}(v)$ to move diagonal Gaussian series rightwards, apply \eqref{AI:ff2}, and multiply with $d^{-1}_1(u) \,d^{-1}_2(v)$ from the right. This gives
\ali{
& [f_{21}(u), f_{32}(v)] + \frac{f_{32}(v)\,(f_{21}(u) - f_{21}(v) ) - f_{31}(u) + f_{31}(v) }{u-v} + \frac{(u+v-1)\, d_1(u)\,f_{42}(v)\,d^{-1}_1(u)}{(u-v-3)(u+v)} \nn\\
& \qu - \frac{(u-2)\,(f_{21}(3-u)\,f_{32}(v) - f_{31}(3-u)) - ( u\,(u+v-\frac32) - \frac32 (v+\frac13))\,(f_{21}(u)\, f_{32}(v) - f_{31}(u)) }{(u-\frac32)(u-v-3)(u+v)} = 0 . \!\!\!\!
\label{C:ff2-3}
}
We multiply this expression by $(u-v-3)$, substitute $u \to v + 3$ (the remaining denominators $u-v$, $u+v$ and $u-\frac32$ do not vanish there), and solve the resulting identity for $f_{42}(v)$, giving
\ali{
f_{42}(v) & = \frac{d^{-1}_1(v+3) }{v+\frac32} \Big( (v+2)\,f_{31}(v+3) - \tfrac12\,f_{31}(-v) \nn \\[-0.3em] & \hspace{2.5cm} - \big( (v+2)\,f_{21}(v+3) - \tfrac12\,f_{21}(-v)\big)\, f_{32}(v) \Big) d_1(v+3)
\nn \\[0.2em] 
&= f_{31}(2+v) - f_{21}(2+v)\,f_{32}(v) \label{C:f42}
}
where we used \eqref{AI:dif1}, together with \eqref{C:f31} and the commutativity of $d_1(u)$ and $f_{32}(v)$, to move $d_1^{-1}(v+3)$ rightwards. Substituting \eqref{C:f42} back into \eqref{C:ff2-3} and applying \eqref{AI:df1}, together with the above action of $d_1(u)$ on $f_{31}(v)$, to $d_1(u)\,f_{42}(v)\,d^{-1}_1(u)$ we obtain
\ali{
[f_{21}(u), f_{32}(v)] & + \frac{f_{32}(v)\,(f_{21}(u) - f_{21}(v)) - f_{31}(u) + f_{31}(v)}{u-v} \nn \\ & + \frac{(f_{21}(u) - f_{21}(v+2))\,f_{32}(v) - f_{31}(u) + f_{31}(v+2)}{u-v-2} = 0 .
} 
Denoting this identity by $B(u,v)=0$, we consider the quantity
$
\frac{1}{u-v-1}\big((u-v-2)\,B(u,v) + B(v+1,v)\big)
$ 
yielding
\equ{
[f_{21}(u),f_{32}(v) ] = \frac{2}{u-v}\big( f_{32}(v)\,(f_{21}(v) - f_{21}(u)) + f_{31}(u) - f_{31}(v) \big) 
}
which, by Corollary \ref{C:RE[m]}, implies the wanted relation \eqref{C:ff2}.

{\it Type DI.} By Corollary \ref{C:RE[m]}, the subalgebra $X^{\tw[n-2]}(\mfo_{2n})$ is isomorphic to the algebra $X^\tw(\mfo_4)$. The wanted relations \eqref{D:ff1}--\eqref{D:ff2} of $X^\tw(\mfo_{2n})$ then follow from their counterpart \eqref{so4:8} in the algebra $X^\tw(\mfo_4)$. This argument does not give \eqref{D:ff3}, since the series $f_{n-1,n-2}(u)$ is not a Gaussian series of $X^{\tw[n-2]}(\mfo_{2n})$. Instead, we apply the automorphism $\si$ of Lemma \ref{L:fork} to relation \eqref{AI:ff3} with $i=n-2$. By the proof of that lemma, $\si$ fixes $f_{n-1,n-2}(u)$ and sends $f_{n,n-1}(u)$ and $f_{n,n-2}(u)$ to $f_{n+1,n-1}(u)$ and $f_{n+1,n-2}(u)$, respectively. This gives \eqref{D:ff3}.
\end{proof}

\begin{prop} \label{P:ff0}
In the algebra $X^{\tw}(\mfg_N)$ we have
\ali{
f_{n,n-1}(u)\,f_{n+1,n}(v) &= \frac{u-v-\eta}{u-v}\,f_{n+1,n}(v)\,f_{n,n-1}(u) \nn\\ & \qu + \frac{1}{u-v}\,([f^{(0)}_{n,n-1}, f_{n+1,n}(v)] - [f_{n,n-1}(u),f^{(0)}_{n+1,n}]) \label{ff0-1}
}
where $\eta=1$ when $\mfg_N=\mfo_{2n+1}$ and $\eta=2$ when $\mfg_N=\mfsp_{2n}$, and
\ali{
f_{n-1,n-2}(u)\,f_{n+1,n-1}(v) &= \frac{u-v-1}{u-v}\,f_{n+1,n-1}(v)\,f_{n-1,n-2}(u) \nn\\ & \qu + \frac{1}{u-v}\,([f^{(0)}_{n-1,n-2}, f_{n+1,n-1}(v)] - [f_{n-1,n-2}(u),f^{(0)}_{n+1,n-1}]) \label{ff0-2}
}
when $\mfg_N=\mfo_{2n}$.
\end{prop}

\begin{proof}
Write $(f,g,h)$ for the triple $(f_{n,n-1},f_{n+1,n},f_{n+1,n-1})$ when $\mfg_N=\mfo_{2n+1}$ or $\mfsp_{2n}$, and for the triple $(f_{n-1,n-2},f_{n+1,n-1},f_{n+1,n-2})$ when $\mfg_N=\mfo_{2n}$, in which case we set $\eta=1$. Then each of the relations \eqref{B:ff2}, \eqref{C:ff2} and \eqref{D:ff3} reads
\[
(u-v)\,[f(u),g(v)] = \eta\,\big(g(v)\,(f(v)-f(u)) + h(u) - h(v)\big) .
\]
Taking the coefficients at $v^0$ and at $u^0$ on both sides gives
\[
\eta\,h(u) = [g^{(0)},f(u)] , \qq \eta\,h(v) = \eta\,g(v)\,f(v) - [f^{(0)},g(v)] ,
\]
respectively. Substituting these two expressions back yields
\[
(u-v)\,[f(u),g(v)] = -\eta\,g(v)\,f(u) + [f^{(0)},g(v)] - [f(u),g^{(0)}] ,
\]
which is \eqref{ff0-1}, resp.~\eqref{ff0-2}.
\end{proof}

Introduce the notation
\equ{
D^\pm_v f_{i,i-1}(w) := 3 f_{i,i-1}(w) + d^{\mp1}_{i-1}(v)\,d^{\pm1}_{i}(v)\,f_{i,i-1}(w)\,d^{\pm1}_{i-1}(v)\,d^{\mp1}_{i}(v) .
\label{Df}
}

\begin{prop} \label{P:Serre}
In the algebra $X^{\tw}(\mfo_{2n+1})$ we have
\ali{
& \Sym_{w,v}\, [f_{n,n-1}(w),[f_{n,n-1}(v),f_{n+1,n}(u)]] = \Sym_{w,v} \, \frac{1}{v+w-\frac12}\, x(u,v) , \label{B:S1}
\\[0.2em]
& \Sym_{q,w,v}\, [f_{n+1,n}(q),[f_{n+1,n}(w),[f_{n+1,n}(v),f_{n,n-1}(u)]]] \nn\\
& \qq = \Sym_{q,w,v} \,\frac{v}{2\,(v-\frac14)(v+q+\frac12)}\, ( D^+_v f_{n+1,n}(w)\, y(u,v) - y(u,v)\, D^-_v f_{n+1,n}(w)) , \label{B:S2}
}
where
\ali{
x(u,v) & := \bigg( \frac{2\,(v-\frac14)\,f_{n+1,n}(u)}{(u-v)(u+v-\frac12)} - \frac{f_{n+1,n}(v)}{u-v} + \frac{f_{n+1,n}(\tfrac12-v)}{u+v-\frac12} \bigg) \,d^{-1}_{n-1}(v)\,d_{n}(v) , 
\\[0.3em]
y(u,v) & := \bigg( \frac{2\,(v+\frac14)\,f_{n,n-1}(u)}{(u-v-1)(u+v-\frac12)} - \frac{f_{n,n-1}(v+1)}{u-v-1} + \frac{f_{n,n-1}(\tfrac12-v)}{u+v-\frac12} \bigg) \,d^{-1}_{n}(v)\,d_{n+1}(v) . 
}
In the algebra $X^{\tw}(\mfsp_{2n})$ we have
\ali{
& \Sym_{q,w,v}\, [f_{n,n-1}(q),[f_{n,n-1}(w),[f_{n,n-1}(v),f_{n+1,n}(u)]]] \nn\\ 
& \qq = \Sym_{q,w,v} \,\frac{1}{v+q-2}\, ( D^+_v f_{n,n-1}(w)\, x(u,v) - x(u,v)\, D^-_v f_{n,n-1}(w)) , \label{C:S1}
\\
& \Sym_{w,v}\, [f_{n+1,n}(w),[f_{n+1,n}(v),f_{n,n-1}(u)]] = \Sym_{w,v}\, \frac{4}{v+w}\, y(u,v) , \label{C:S2}
}
where
\ali{
x(u,v) & := \bigg( \frac{2\,(v-1)\,f_{n+1,n}(u)}{(u-v)(u+v-2)} - \frac{f_{n+1,n}(v)}{u-v} + \frac{f_{n+1,n}(2-v)}{u+v-2} \bigg) \,d^{-1}_{n-1}(v)\,d_{n}(v) ,
\\[0.3em]
y(u,v) & := \bigg( \frac{2\,v\,f_{n,n-1}(u)}{(u-v-2)(u+v-2)} - \frac{f_{n,n-1}(v+2)}{u-v-2} + \frac{f_{n,n-1}(2-v)}{u+v-2} \bigg) \,d^{-1}_{n}(v)\,d_{n+1}(v) .
}
In the algebra $X^{\tw}(\mfo_{2n})$ we have
\ali{
\Sym_{w,v}\, [f_{n-1,n-2}(w),[f_{n-1,n-2}(v),f_{k,n-1}(u)]] &= \Sym_{w,v} \, \frac{1}{w+v-1}\, x(u,v) , \label{D:S1} 
\\[0.3em]
\Sym_{w,v}\, [f_{k,n-1}(w),[f_{k,n-1}(v),f_{n-1,n-2}(u)]] &= \Sym_{w,v} \, \frac{1}{w+v}\, y(u,v) , \label{D:S2}
}
where
\ali{
x(u,v) & := \bigg( \frac{(2v-1)\,f_{k,n-1}(u)}{(u-v)(u+v-1)} - \frac{f_{k,n-1}(v)}{u-v} + \frac{f_{k,n-1}(1-v)}{u+v-1} \bigg) \,d^{-1}_{n-2}(v)\,d_{n-1}(v) ,
\\[0.3em]
y(u,v) & := \bigg( \frac{2v\,f_{n-1,n-2}(u)}{(u-v-1)(u+v-1)} - \frac{f_{n-1,n-2}(v+1)}{u-v-1} + \frac{f_{n-1,n-2}(1-v)}{u+v-1} \bigg) \, d^{-1}_{n-1}(v)\,d_{k}(v) ,
}
for $k=n,n+1$.
\end{prop}

\begin{proof}
For $\mfg_N = \mfo_{2n+1}, \mfsp_{2n}$ the wanted relations follow from Lemmas \ref{L:AI-rels}, \ref{L:BCD-triv}, \ref{L:BCD-df} and \ref{L:ff}, Remark~\ref{R:dif}, Proposition~\ref{P:z(u)} and the auxiliary Lemmas \ref{L:B-aux}~and~\ref{L:C-aux}. We give the computations for \eqref{B:S1} and \eqref{C:S2} in full; of the auxiliary relations they use only \eqref{B-aux-ff1} and \eqref{C-aux-ff2}, respectively. Relations \eqref{B:S2} and \eqref{C:S1} are obtained by the same method, but the computations are considerably longer, and we do not reproduce them in full. Appendix~\ref{app:Comms} gives the evaluated form of their left hand sides and sketches the computation.

Consider \eqref{B:S1}. Relation \eqref{B:ff2} implies 
\ali{
[f_{n,n-1}(w),[f_{n,n-1}(v),f_{n+1,n}(u)]] = \frac{1}{v-u}\,( A - f_{n+1,n}(u)\,B ) \label{B:S10}
}
where
\ali{
A & = [f_{n,n-1}(w),\,f_{n+1,n-1}(v) - f_{n+1,n-1}(u) ] \nn\\ & \qu - \frac{1}{w-u} \, ( f_{n+1,n-1}(w) - f_{n+1,n-1}(u) ) \, (f_{n,n-1}(v) - f_{n,n-1}(u) ) , \label{B:S1A} \\[0.2em]
B & = [f_{n,n-1}(w),\,f_{n,n-1}(v) - f_{n,n-1}(u) ] \nn\\ & \qu - \frac{1}{w-u} (f_{n,n-1}(w) - f_{n,n-1}(u))\, (f_{n,n-1}(v) - f_{n,n-1}(u)) . \label{B:S1B} 
}
Relation \eqref{B-aux-ff1} implies
\ali{
\Sym_{w,v} \frac{A}{v-u} = \Sym_{w,v} \bigg( \!\bigg( \frac{f_{n+1,n}(v)}{v-u} + \frac{f_{n+1,n}(\frac12-v)}{v+u-\frac12} \bigg) \frac{d_{n-1}^{-1}(v)\,d_{n}(v)}{w+v-\frac12} -\frac{f_{n+1,n}(u)\,d_{n-1}^{-1}(u)\,d_{n}(u) }{(v-u)(w+u-\frac12)} \bigg) . \label{B:S1A-sym}
}
Relation \eqref{AI:ff2} implies
\ali{
\Sym_{w,v} \frac{B}{v-u} = \Sym_{w,v} \frac{1}{v-u} \bigg( \frac{(2v-\frac12)\,d_{n-1}^{-1}(v)\,d_{n}(v)}{(v+u-\frac12)(w+v-\frac12)} - \frac{d_{n-1}^{-1}(u)\,d_{n}(u)}{(w+u-\frac12)} \bigg) . \label{B:S1B-sym}
}
Then
\ali{
\Sym_{w,v} [f_{n,n-1}(w),[f_{n,n-1}(v),f_{n+1,n}(u)]] = \eqref{B:S1A-sym} - f_{n+1,n}(u) \cdot \eqref{B:S1B-sym} = \Sym_{w,v}\, \frac{1}{w+v-\frac12}\,x(u,v)
}
which is exactly the wanted \eqref{B:S1}.

Consider \eqref{C:S2}. Relation \eqref{C:ff2} implies 
\ali{
& [f_{n+1,n}(w),[f_{n+1,n}(v),f_{n,n-1}(u)]] = \frac{2}{v-u} ( A - B ) \label{C:S10}
}
where
\ali{
A &= [f_{n+1,n}(w), f_{n+1,n}(v)] \, (f_{n,n-1}(v)- f_{n,n-1}(u)) \nn\\ 
& \qu - 2\,f_{n+1,n}(v) \, f_{n+1,n}(w) \bigg(\frac{f_{n,n-1}(w) - f_{n,n-1}(u)}{w-u} - \frac{f_{n,n-1}(w) - f_{n,n-1}(v)}{w-v} \bigg) , \label{C:S1A} \\[0.2em]
B &= [f_{n+1,n}(w), f_{n+1,n-1}(v)- f_{n+1,n-1}(u) ] \nn\\
& \qu - 2\,f_{n+1,n}(v) \bigg(\frac{f_{n+1,n-1}(w) - f_{n+1,n-1}(u)}{w-u} - \frac{f_{n+1,n-1}(w) - f_{n+1,n-1}(v)}{w-v} \bigg). \label{C:S1B}
}
Relations \eqref{C:ff1} and \eqref{z(u)}, together with \eqref{AI:dif2} applied at $v$ and at $-v$, imply
\ali{
\Sym_{w,v} \frac{A}{v-u} &= 2\,\Sym_{w,v} \bigg(\!\bigg(
\frac{w-v}{(w+v)(v-u+2)} \bigg( \frac{(v+1)\,f_{n,n-1}(v+2)}{v(w-v-2)}-\frac{(v+u)\,f_{n,n-1}(u)}{(w-u)(v+u-2)} \bigg) \nn \\
& \hspace{2cm} + \frac{w-v}{w+v-2} \bigg( \frac{f_{n,n-1}(2-v)}{v(w+v)(v+u-2)} - \frac{f_{n,n-1}(w)}{(w-u)(w-v-2)}\bigg)\!\bigg)\, d^{-1}_n(v) \, d_{n+1}(v) \nn\\
& \hspace{2cm} + f_{n+1,n}(v)\,f_{n+1,n}(v)\,\bigg(\frac{f_{n,n-1}(u)}{(v-u)(w-u)}-\frac{f_{n,n-1}(v)}{(v-u)(w-v)} + \frac{f_{n,n-1}(w)}{(w-u)(w-v)}\bigg) \bigg). \label{C:S1A-sym}
}
Relation \eqref{C-aux-ff2} implies
\ali{
\Sym_{w,v} \frac{B}{v-u} &= 2\,\Sym_{w,v} \bigg(\!\bigg(
\frac{1}{v-u+2} \bigg( \frac{f_{n,n-1}(v+2)}{v(w-v-2)}+\frac{(v-u)\,f_{n,n-1}(u)}{(w-u)(v+u-2)} \bigg) \nn \\
& \hspace{2.2cm} - \frac{1}{w+v-2} \bigg( \frac{(v-1)\,f_{n,n-1}(2-v)}{v(v+u-2)} + \frac{(w-v)\,f_{n,n-1}(w)}{(w-u)(w-v-2)}\bigg)\!\bigg)\, d^{-1}_n(v)\, d_{n+1}(v) \nn\\
& \hspace{2.2cm} + f_{n+1,n}(v)\,f_{n+1,n}(v)\,\bigg(\frac{f_{n,n-1}(u)}{(v-u)(w-u)}-\frac{f_{n,n-1}(v)}{(v-u)(w-v)} +\frac{f_{n,n-1}(w)}{(w-u)(w-v)}\bigg) \bigg). \label{C:S1B-sym}
}
Then
\ali{
\Sym_{w,v} [f_{n+1,n}(w),[f_{n+1,n}(v),f_{n,n-1}(u)]] = 2 \cdot \eqref{C:S1A-sym} - 2\cdot \eqref{C:S1B-sym} = \Sym_{w,v} \,\frac{4}{v+w}\,y(u,v)
}
which is exactly the wanted \eqref{C:S2}.

For $\mfg_N=\mfo_{2n}$ and $k=n$ the wanted relations \eqref{D:S1} and \eqref{D:S2} follow from \eqref{AI:fff}, where $\ka=n-1$. Taking $i=n-1$ and $j=n-2$ in \eqref{AI:fff}, we have $v+w-\ka+j = v+w-1$, $2v-\ka+j=2v-1$, $\frac{i-j-1}{2}=0$ and $\ka-v-\frac{i+j-1}{2} = 1-v$, which gives \eqref{D:S1}. Taking $i=n-2$ and $j=n-1$ instead, we have $v+w-\ka+j = v+w$, $2v-\ka+j=2v$, $\frac{i-j-1}{2}=-1$ and $\ka-v-\frac{i+j-1}{2} = 1-v$, which gives \eqref{D:S2}. In both cases $u+v-\ka+\frac{i+j-1}{2} = u+v-1$. Finally, the automorphism $\si$ of Lemma \ref{L:fork} fixes $f_{n-1,n-2}(u)$, $d_{n-2}(u)$ and $d_{n-1}(u)$, and sends $f_{n,n-1}(u)$ and $d_n(u)$ to $f_{n+1,n-1}(u)$ and $d_{n+1}(u)$, respectively. Applying it to the case $k=n$ gives the case $k=n+1$.
\end{proof}


\subsection{Drinfeld type relations} \label{sec:Dr}

Define the series
\ali{
\label{Dr1}
b_i(u) &:= f_{i+1,i}(u+\tfrac{\ka-i}{2}) , \qq &
h_i(u) &:= d_i^{-1}(u+\tfrac{\ka-i}{2}) \,d_{i+1}(u+\tfrac{\ka-i}{2}) \hspace{0.86cm} \text{for } 1 \le i < n,
\\[0.4em]
\label{Dr2}
b_n(u) &:= \begin{cases} 
f_{n+1,n}(u-\tfrac14), \\[0.5em]
\tfrac{1}{\sqrt{-2}}\,f_{n+1,n}(u), \\[0.5em]
f_{n+1,n-1}(u),
\end{cases} &
h_n(u) &:= \begin{cases}
\frac{u-\frac14}{u-\frac12} \, d^{-1}_n(u-\tfrac14) \, d_{n+1}(u-\tfrac14)\; &\text{for } X^\tw(\mfo_{2n+1}),
\\[0.5em]
-d^{-1}_n(u) \, d_{n+1}(u) &\text{for } X^\tw(\mfsp_{2n}),
\\[0.5em]
d^{-1}_{n-1}(u) \, d_{n+1}(u) &\text{for } X^\tw(\mfo_{2n}) ,
\end{cases}
}
and elements $b_{i,r}$, $h_{i,r}$ with $1\le i \le n$ and $r\ge0$ by 
\equ{
b_i(u) = \sum_{r\ge 0} b_{i,r} u^{-r-1} , \qu h_i (u) = 1 + \sum_{r\ge 0} h_{i,r} u^{-r-1} . 
\label{b-h-series}
}
We call the series in \eqref{Dr1} and \eqref{Dr2} the {\it Drinfeld series} and the elements $b_{i,r}$, $h_{i,r}$ the {\it Drinfeld generators}.
In type DI, $\ka=n-1$, so that $b_{n-1}(u)=f_{n,n-1}(u)$ and $h_{n-1}(u)=d^{-1}_{n-1}(u)\,d_n(u)$. Hence, by Lemma \ref{L:fork}, the automorphism $\si$ exchanges $b_{n-1}(u)$ with $b_n(u)$ and $h_{n-1}(u)$ with $h_n(u)$, and fixes $b_i(u)$ and $h_i(u)$ for $1\le i\le n-2$.

\begin{lemma} \label{L:heven}
We have $h_i(u) = h_i(-u)$ for $1\le i \le n$. 
\end{lemma}

\begin{proof}
The cases $1\le i <n$ follow from \eqref{AI:dd=dd}, since the coefficients of the series $d_j(u)$ with $1\le j\le n$ pairwise commute by \eqref{AI:dd}. Let $i=n$. When $\mfg_N=\mfsp_{2n}$ or $\mfo_{2n+1}$, Corollary \ref{C:RE[m]} identifies $d_n(u)$ and $d_{n+1}(u)$ with the series $d_1(u)$ and $d_2(u)$ of $X^\tw(\mfsp_2)$ or $X^\tw(\mfo_3)$, respectively, and the claim follows from \eqref{sp2:1} or from \eqref{so3:3} and \eqref{so3:4}, respectively. When $\mfg_N=\mfo_{2n}$, the claim follows from the case $i=n-1$ by applying the automorphism $\si$ of Lemma \ref{L:fork}.
\end{proof}

\begin{lemma} \label{L:gen}
$ $
\begin{enumerate}
\item The elements $b_{i,r}$, $h_{i,2r+1}$ with $1\le i \le n$ and $r\ge 0$ generate the algebra $SY^\tw(\mfg_N)$. 

\item The elements $b_{i,r}$, $d_{1,r}$, $d_{i+1,2r+1}$ with $1\le i \le n$ and $r\ge0$ generate the algebra $X^\tw(\mfg_N)$.
\end{enumerate}
\end{lemma}

\begin{proof}
(1) We pass to the associated graded algebra. Equip $SY^{\tw}(\mfg_N)$ with the filtration of Section \ref{sec:tY}, defined by $\deg s^{(r)}_{ij}=r-1$, so that $\gr SY^{\tw}(\mfg_N)\cong U(\mfg_N[x]^\om)$. Under this filtration the generators $b_{i,r}$ and $h_{i,2r+1}$ have respective degrees $r$ and $2r+1$, and their images $\bar b_{i,r}$ and $\bar h_{i,2r+1}$ in $\gr SY^{\tw}(\mfg_N)$ are mapped by the graded isomorphism \eqref{SY-graded-iso} to
\ali{
\label{SY-graded-map-1}
\bar{b}_{i,r} &\mapsto F_{i+1,i}^{(\om,r)} , \qq & 
\bar{h}_{i,2r+1} &\mapsto F_{i+1,i+1}^{(\om,2r+1)} - F_{ii}^{(\om,2r+1)} \hspace{0.96cm} \text{for } 1\le i < n,
\\[0.3em]
\label{SY-graded-map-2}
\bar{b}_{n,r} &\mapsto 
\begin{cases} 
F_{n+1,n}^{(\om,r)} , \\[0.4em]
\frac{1}{\sqrt{-2}} F_{n+1,n}^{(\om,r)} , \\[0.4em]
F_{n+1,n-1}^{(\om,r)} ,
\end{cases} 
&
\bar{h}_{n,2r+1} &\mapsto 
\begin{cases} 
-F_{nn}^{(\om,2r+1)} &\text{for } X^\tw(\mfo_{2n+1}), \\[0.4em]
-2F_{nn}^{(\om,2r+1)} &\text{for } X^\tw(\mfsp_{2n}), \\[0.4em]
-F_{nn}^{(\om,2r+1)} - F_{n-1,n-1}^{(\om,2r+1)} &\text{for } X^\tw(\mfo_{2n}) .
\end{cases} 
}
Here we have used that $F^{(\om,2r+1)}_{n+1,n+1}=0$ when $\mfg_N = \mfo_{2n+1}$ and $F^{(\om,2r+1)}_{n+1,n+1}=-\theta F^{(\om,2r+1)}_{nn}$ otherwise. The algebra $U(\mfg_N[x]^\om)$ is spanned by ordered monomials in the following elements:
\equ{
F^{(\om,r)}_{ij}\;\text{ with }\; 1\le j< i < \bar\jmath + \tfrac{1-\theta}{2}\; \text{ and }\; F^{(\om,2r+1)}_{ii} \;\text{ with }\; 1 \le i \le n \;\text{ and }\; r\ge 0. \label{F-basis}
}
It suffices to obtain all these elements from the right hand sides of \eqref{SY-graded-map-1}--\eqref{SY-graded-map-2}. This is clear for the elements $F_{ii}^{(\om,2r+1)}$. The remaining elements $F^{(\om,r)}_{ij}$ are obtained by the following relations in $U(\mfg_N[x]^\om)$:
\ali{
[F^{(\om,0)}_{i+1,i}, F^{(\om,r)}_{ij}] &= F^{(\om,r)}_{i+1,j} \; &&\text{ with }\; 1\le j < i < N-n,
\intertext{and}
[F^{(\om,0)}_{n+i,n-i+1}, F^{(\om,r)}_{n-i+2,n-i}] + \del_{i1} F^{(\om,r)}_{n,n-1} &= - F^{(\om,r)}_{n+i+1,n-i} \; &&\text{ with }\; 1\le i \le n-1 , \\
[F^{(\om,0)}_{n+i+1,n-i}, F^{(\om,r)}_{n-i,j}] &= F^{(\om,r)}_{n+i+1,j} \; &&\text{ with }\; 1\le i \le n-2 , \;\; 1\le j \le n-i-1 ,
\intertext{when $\mfg_N = \mfo_{2n+1}$ and}
[F^{(\om,0)}_{n-i+1,n-i}, F^{(\om,r)}_{n+i,n-i}] - F^{(\om,r)}_{n+i,n-i+1} &= - F^{(\om,r)}_{n+i+1,n-i} \; &&\text{ with }\; 1\le i \le n-1 , \\
[F^{(\om,0)}_{n+i,n-i+1}, F^{(\om,r)}_{n-i+1,j}] &= 2F^{(\om,r)}_{n+i,j} \; &&\text{ with }\; 1\le i \le n-1 , \; 1\le j \le n-i ,
\intertext{when $\mfg_N = \mfsp_{2n}$ and}
[F^{(\om,0)}_{n+i,n-i}, F^{(\om,r)}_{n-i+1,n-i-1}] &= - F^{(\om,r)}_{n+i+1,n-i-1} \; &&\text{ with }\; 1\le i \le n-2 , \\
[F^{(\om,0)}_{n+i,n-i}, F^{(\om,r)}_{n-i,j}] &= F^{(\om,r)}_{n+i,j} \; &&\text{ with }\; 1\le i \le n-2 ,\;\; 1\le j \le n-i-1 ,
}
when $\mfg_N = \mfo_{2n}$.

\smallskip

\noindent (2) We indicate only the differences from (1). Now $\gr X^\tw(\mfg_N) \cong U(\mfg_N[x]^\om) \ot \C[\zeta_0,\zeta_1,\ldots]$, and the graded isomorphism \eqref{X-graded-iso} maps the images of the generators $d_{1,r}$ and $d_{i+1,2r+1}$ to
\gat{
\bar{d}_{1,2r} \mapsto \tfrac12 \zeta_{2r} + \ga_{11}^{(2r)}, \qu
\bar{d}_{i,2r+1} \mapsto F^{(\om,2r+1)}_{ii} + \tfrac12 \zeta_{2r+1} \;\text{ for } 1\le i \le n, \\ 
\bar{d}_{n+1,2r+1} \mapsto \tfrac12 \theta \zeta_{2r+1} +
\begin{cases}
0 &\text{for } X^\tw(\mfo_{2n+1}), \\[0.4em]
F_{nn}^{(\om,2r+1)} &\text{for } X^\tw(\mfsp_{2n}), \\[0.4em]
-F_{nn}^{(\om,2r+1)} &\text{for } X^\tw(\mfo_{2n}) .
\end{cases}
}
Hence the elements $F_{ii}^{(\om,2r+1)}$ and $\zeta_r$ with $1\le i \le n$ and $r\ge 0$ are in the image.
\end{proof}

Let $A=(a_{ij})_{1\le i,j\le n}$ be the Cartan matrix of $\mfg_N$ and let ${\mr d}_i=1$ for $1\le i\le n$ except ${\mr d}_n=2$ when $\mfg_N=\mfsp_{2n}$ and ${\mr d}_n=\frac12$ when $\mfg_N=\mfo_{2n+1}$. Set $c_{ij}:={\mr d}_i a_{ij}$. The Serre type relations between the lowest Drinfeld generators are already determined by the degree zero component of the filtration.

\begin{lemma} \label{L:deg0}
The unital subalgebra $\mr{F}_0$ of $X^\tw(\mfg_N)$ generated by the elements $s^{(1)}_{ij}$ with $1\le i,j\le N$ is isomorphic to $U(\mfg_N^\om)\ot\C[\zeta_0]$ via the mapping $s^{(1)}_{ij}\mapsto F^{(\om,0)}_{ij}+\ga^{(0)}_{ij}+\frac12 g_{ij}\,\zeta_0$. Consequently, for $1\le i\ne j\le n$, the elements $b_{i,0}$ of $X^\tw(\mfg_N)$, and hence those of $SY^\tw(\mfg_N)$, satisfy
\gat{
[b_{i,0},b_{j,0}]=0 \qu\text{if}\qu a_{ij}=0, \qq
[b_{i,0},[b_{i,0},b_{j,0}]]=-{\mr d}_i\,b_{j,0} \qu\text{if}\qu a_{ij}=-1, \nn\\
[b_{i,0},[b_{i,0},[b_{i,0},b_{j,0}]]]=-4{\mr d}_i\,[b_{i,0},b_{j,0}] \qu\text{if}\qu a_{ij}=-2 . \label{deg0-Serre}
}
\end{lemma}

\begin{proof}
Since $\deg s^{(1)}_{ij}=0$, the filtered component of degree $0$ of $X^\tw(\mfg_N)$ is $\mr{F}_0$, and that of degree $-1$ is zero. Hence $\mr F_0$ coincides, as an algebra, with the degree zero component of $\gr X^\tw(\mfg_N)$, which \eqref{X-graded-iso} maps isomorphically onto the degree zero component $U(\mfg_N^\om)\ot\C[\zeta_0]$ of $U(\mfg_N[x]^\om)\ot\C[\zeta_0,\zeta_1,\ldots]$. 

The series $f_{ji}(u)\,d_i(u)$ is $s_{ji}(u)$ plus products each containing at least two off-diagonal series, hence of order $u^{-2}$, and $d_i(u)=g_{ii}+\mc{O}(u^{-1})$, so $f^{(0)}_{ji}=g^{-1}_{ii}\,s^{(1)}_{ji}$. By \eqref{Dr1}--\eqref{Dr2}, we have $b_{i,0}=s^{(1)}_{i+1,i}$ for $1\le i<n$, while $b_{n,0}$ equals $s^{(1)}_{n+1,n}$, $\frac{1}{\sqrt{-2}}\,s^{(1)}_{n+1,n}$ and $s^{(1)}_{n+1,n-1}$ when $\mfg_N=\mfo_{2n+1}$, $\mfsp_{2n}$ and $\mfo_{2n}$, respectively. For these entries $g_{ij}=\ga^{(0)}_{ij}=0$. Therefore \eqref{deg0-Serre} is equivalent to the same relations for the elements $X_i:=F^{(\om,0)}_{i+1,i}$ with $1\le i<n$ and $X_n:=F^{(\om,0)}_{n+1,n}$, $\frac{1}{\sqrt{-2}}\,F^{(\om,0)}_{n+1,n}$, $F^{(\om,0)}_{n+1,n-1}$, respectively, of the Lie algebra $\mfg_N^\om$. We verify these in three cases.

If $1\le i,j<n$, the mapping $E_{ab}\mapsto F_{ab}$ with $1\le a,b\le n$ is an embedding of Lie algebras $\mfgl_n\into\mfg_N$ by \eqref{[F,F]}, and $X_i$ is the image of $E_{i+1,i}-E_{i,i+1}$. The relations then reduce to the familiar relations between these elements of $\mfo_n\subset\mfgl_n$.

If one of $i,j$ equals $n$ and $a_{ij}=0$, then, except for the pair $\{n-1,n\}$ in type DI, the index sets $\{a,b,\bar a,\bar b\}$ of the elements $F_{ab}$ occurring in $X_i$ and in $X_j$ are disjoint. Hence $X_i$ and $X_j$ commute.

In the remaining cases all elements involved are linear combinations of the elements $F_{ab}$ with $m<a,b\le N-m$, where $m=n-2$, except that $m=n-3$ for the pair $\{n-2,n\}$ in type DI. These elements span a subalgebra isomorphic to $\mfg_{N-2m}$, that is, to $\mfo_5$, $\mfsp_4$, $\mfo_4$ or $\mfo_6$, and the relations are verified by a direct computation in its vector representation.
\end{proof}

We are now in a position to state the main result of this paper. Recall notation \eqref{u-symm} and \eqref{Sym}.

\begin{thm} \label{T:iso}
The special twisted Yangian $SY^\tw(\mfg_N)$ is isomorphic to the associative unital $\C$-algebra generated by the elements $b_{i,r}$ and $h_{i,2r+1}$ with $1\le i \le n$ and $r\ge 0$ and satisfying relations given by the following identities, where $b_i(u) = \sum_{r\ge0} b_{i,r}u^{-r-1}$ and $h_i(u) = 1+ \sum_{r\ge0} h_{i,2r+1}u^{-2r-2}$
and each identity is understood after multiplication by the product of its scalar denominators or, equivalently, with rational functions expanded in the domain $|u|>|v|>|w|>|q|$:
\ali{
h_i(u)\,b_j(v) &= \bigg( \frac{(u-v+\frac{c_{ij}}{2})(u+v-\frac{c_{ij}}{2})}{(u-v-\frac{c_{ij}}{2})(u+v+\frac{c_{ij}}{2})}\,b_j(v) - \bigg\{ \frac{c_{ij} (u-\tfrac{c_{ij}}2)}{ u\,(u-v-\tfrac{c_{ij}}2)} \, b_j(u-\tfrac{c_{ij}}2) \bigg\}^{\!u}\, \bigg)\,h_i(u) , \label{hb}
\\
b_i(u)\,b_j(v) &= \frac{u-v+\tfrac{c_{ij}}{2}}{u-v-\tfrac{c_{ij}}{2}}\,b_j(v)\,b_i(u) + \frac{1}{u-v-\tfrac{c_{ij}}{2}}\,([b_{i,0},b_j(v)] - [b_i(u),b_{j,0}]) \qu\text{if}\qu a_{ij}<0, \label{bb1}
\\[0.2em]
[b_i(u),b_i(v)] &= \frac{c_{ii}}{2} \frac{(b_i(v) - b_i(u))^2}{v-u} + \frac{h_i(v) - h_i(u)}{v+u} , \qq [h_i(u), h_j(v)]=0, \label{bb2}
}
and Serre type relations
\ali{
& [b_i(v),b_j(u)] = 0 &&\qu\text{if}\qu a_{ij} = 0 \label{S0} , \\
\Sym_{w,v} [b_i(w),\,&[b_i(v),b_j(u)]] = \Sym_{w,v} \, \frac{x_{ij}(v,u)}{v+w} &&\qu\text{if}\qu a_{ij} = -1 \label{S1} , \\[0.2em]
\Sym_{q,w,v} [b_i(q),\,&[b_i(w),[b_i(v),b_j(u)]]] \nn\\ &= \Sym_{q,w,v} \frac{D^+_v b_i(w) \, x_{ij}(v,u) - x_{ij}(v,u)\, D^-_v b_i(w)}{v+q} &&\qu\text{if}\qu a_{ij} = -2 , \label{S2}
}
where
\gat{
D^\pm_v b_j(w) = 3\,b_j(w) + h^{\pm1}_j(v)\,b_j(w)\,h_j^{\mp1}(v) , \label{Db} \\
x_{ij}(v,u) = {\mr d}_i \bigg(\frac{2\,v\, b_{j}(u)}{(u-v+\frac{c_{ij}}2)(u+v+\frac{c_{ij}}2)} - \frac1v \bigg\{ \frac{v\, b_{j}(v-\frac{c_{ij}}2)}{u-v+\frac{c_{ij}}2} \bigg\}^v\bigg) h_i(v) . \label{xij} 
}
\end{thm}

Before proving the theorem, we record the relations between the coefficients of the series that follow from \eqref{hb}--\eqref{S2}. They are the defining relations of the Drinfeld presentations of \cite{Lu26, LWZ25b}, and Step~1 of the proof below uses them.

\begin{lemma} \label{L:LWZ}
Let $b_{i,r}$ and $h_{i,2r+1}$, with $1\le i\le n$ and $r\ge0$, be the coefficients of the Drinfeld series \eqref{Dr1}--\eqref{Dr2} of $SY^{\tw}(\mfg_N)$, and set $h_{i,-1}=1$, $h_{i,r}=0$ if $r$ is even or $r<-1$, and $\{ a, b \} = ab + ba$. Then these elements satisfy the relations
\ali{
[h_{i,r}, h_{j,s}] &= 0 \label{modes-hh} \\
[h_{i,r+1},b_{j,s}] - [h_{i,r-1},b_{j,s+2}] &= c_{ij}\,\{ h_{i,r-1}, b_{j,s+1} \} + \tfrac{1}{4}\,c_{ij}^2\, [ h_{i,r-1}, b_{j,s}] , \label{modes-hb} \\[0.2em]
[b_{i,r+1},b_{j,s}] - [b_{i,r},b_{j,s+1}] &= \tfrac{1}{2}\,c_{ij}\,\{ b_{i,r}, b_{j,s} \} - 2\,\del_{ij} (-1)^s h_{i,r+s+1} \label{modes-bb}
}
together with the finite Serre type relations
\ali{
[b_{i,0}, b_{j,0}] &= 0 &&\text{if}\qu a_{ij} = 0 , \label{modes-S0} \\
[b_{i,0},[b_{i,0},b_{j,0}]] &= - {\mr d}_i b_{j,0} &&\text{if}\qu a_{ij} = -1, \label{modes-S1} \\
[b_{i,0},[b_{i,0},[b_{i,0},b_{j,0}]]] &= -4 {\mr d}_{i}\,[b_{i,0}, b_{j,0}] \qq &&\text{if}\qu a_{ij} = -2 . \label{modes-S2}
}
These are the defining relations of the twisted Yangian $\mc{Y}^{\imath}(\mfg_N)$ in the Drinfeld presentation of \cite[Def.~2.7]{Lu26}. Up to the rescaling
\equ{
b_{i,r} \mapsto \sqrt{\mr{d}_i}\,b_{i,r},\qu
h_{i,2r+1} \mapsto \mr{d}_i\,h_{i,2r+1} , \label{LWZ-scaling}
}
they coincide with the $\C$-algebra form of the reduced presentation \cite[Thm.~4.14]{LWZ25b} of the algebra \cite[Def.~4.1]{LWZ25b}, with $\hbar$ specialised by $\hbar\,d_i=\mr{d}_i$. Here $d_i$ are the symmetrising integers of loc.~cit., so that $\hbar=\frac12$ in type BI, $\hbar=2$ for $\mfsp_2$, and $\hbar=1$ otherwise; their $\hbar\,d_i\,c_{ij}$ then correspond to our $c_{ij}$.

The minimalistic presentation of \cite[Thm.~3.1]{Lu26} consists of the finite type Serre relations \eqref{modes-S0}--\eqref{modes-S2} together with
\ali{
[h_{i,1}, h_{j,1}] = 0, &\qu
[h_{i,1},b_{j,0}] = 2 c_{ij} b_{j,1} , \\
[b_{i,1},b_{j,0}] - [b_{i,0},b_{j,1}] &= \tfrac{1}{2} c_{ij} \{ b_{i,0}, b_{j,0}\} - 2\del_{ij} h_{i,1} , \hspace{-2cm} \\
[h_{i,1},[b_{i,1},[h_{i,1},b_{i,1}]]] &= c_{ii}^2\,[b_{i,1}^2,h_{i,1}] \qu\text{if}\qu \mfg_N = \mfsp_2, \mfsp_4, \mfo_3, \mfo_4, \mfo_5. \label{hbhb}
}
The first two lines are \eqref{modes-hh} with $r=s=1$ and \eqref{modes-hb}--\eqref{modes-bb} with $r=s=0$. Relation \eqref{hbhb} holds at every node; for simple $\mfg_N$ in the cases listed, \cite[Thm.~3.1]{Lu26} requires it at a single node.
\end{lemma}
\begin{proof}
The argument uses only the current relations \eqref{hb}--\eqref{S2} and the evenness of the series $h_i(u)$, so it applies to any algebra whose generating series satisfy them. For the Drinfeld series \eqref{Dr1}--\eqref{Dr2}, the current relations hold by the results of the preceding sections (see the proof of Theorem \ref{T:iso} below), and the evenness holds by Lemma \ref{L:heven}.

Clear the denominators in \eqref{bb2}, \eqref{hb} and \eqref{bb1} and compare coefficients. This gives \eqref{modes-hh}--\eqref{modes-bb}, where the conventions on $h_{i,r}$ record the constant term of $h_i(u)$. For $a_{ij}=0$, relations \eqref{modes-bb} and \eqref{modes-S0} are immediate from \eqref{S0}.

For the Serre type relations, put $b[x,y]:=(b(x)-b(y))/(x-y)$. Then
\[
x_{ij}(v,u)=\mr{d}_i\big(b_j[u,v-\tfrac{c_{ij}}{2}]-b_j[u,-v-\tfrac{c_{ij}}{2}]\big)\,h_i(v) .
\]
Hence $x_{ij}(v,u)$ contains only negative powers of $u$ and $v$ and is odd in $v$. Its coefficient of $v^{-1}$ is $-2\mr{d}_i\,b_j(u)$. Moreover, $D^\pm_v b_i(w)$ is even in $v$. Therefore, in each symmetrised right-hand side of \eqref{S1}--\eqref{S2}, the two terms that differ by exchanging $v$ with $w$, resp.~$q$, combine into a divided difference. It contains only negative powers of all variables, and its lowest coefficient is determined by the coefficient of $v^{-1}$ of the numerator. Now take the coefficients of $w^{-1}v^{-1}u^{-1}$ in \eqref{S1} and of $q^{-1}w^{-1}v^{-1}u^{-1}$ in \eqref{S2}. This gives \eqref{modes-S1}--\eqref{modes-S2}. In $SY^\tw(\mfg_N)$ these relations also follow from Lemma \ref{L:deg0}.

The algebra $\mfo_4$ is not simple, so the case $\mfg_N=\mfo_4$ is not covered by \cite[Def.~2.7]{Lu26}; it is covered componentwise. Here the relations decouple the two nodes into two commuting copies of the rank one split type AI case, in accordance with the isomorphism $SY^{\tw}(\mfo_4) \cong Y^+(\mfsl_2) \ot Y^+(\mfsl_2)$ of \cite[Cor.~4.18]{GRW16}. Relation \eqref{hbhb} follows from \eqref{modes-hh}--\eqref{modes-bb} as in the proof of \cite[Lem.~2.9]{Lu26}, by \eqref{modes-bb} with $i=j$ and $r=s=1$ and by $[h_{i,1},h_{i,3}]=0$.

It remains to compare with \cite{LWZ25b}. The specialisation $\hbar\,d_i=\mr{d}_i$ turns the coefficients $\hbar\,d_i\,c_{ij}$ of {\it loc.~cit.} into our $c_{ij}$. The invertible rescaling \eqref{LWZ-scaling} leaves the relations with these coefficients unchanged and supplies the factors $\mr{d}_i$ in the relation $[h_{i,1},b_{j,r}]=2c_{ij}\,b_{j,r+1}$ and in the finite Serre relations. Hence \eqref{modes-hh}--\eqref{modes-S2} are the relations of \cite[Thm.~4.14]{LWZ25b}, specialised and rescaled as above. These relations are homogeneous for the grading with $\deg\hbar=1$, so the specialisations at all nonzero values of $\hbar$ are isomorphic.
\end{proof}

\begin{proof}[Proof of Theorem \ref{T:iso}]
Let $ \mr{Y}(\mfg_N)$ denote the associative unital $\C$-algebra with generators $b_{i,r}$ and $h_{i,2r+1}$ with $1\le i\le n$ and $r\ge 0$ and satisfying the relations in the theorem. The proof is in three steps. In Steps~1 and~2 we construct an isomorphism $\Xi_N$ from the twisted Yangian $\mc{Y}^{\imath}(\mfg_N)$ in the Drinfeld presentation onto $SY^{\tw}(\mfg_N)$. These two steps use the Serre type relations only in the finite form of Lemma \ref{L:deg0}. In Step~3 we show that $\Xi_N$ factors through $\mr{Y}(\mfg_N)$.

\smallskip

{\it Step 1.} Let $\mc{Y}^{\imath}(\mfg_N)$ be the algebra of Lemma \ref{L:LWZ}, with generators $b^{\imath}_{i,r}$ and $h^{\imath}_{i,2r+1}$, where $1\le i\le n$ and $r\ge0$, subject to the relations \eqref{modes-hh}--\eqref{modes-S2}, the case $\mfg_N=\mfo_4$ being understood componentwise.
We show first that the Drinfeld series \eqref{Dr1}--\eqref{Dr2} of $SY^{\tw}(\mfg_N)$ satisfy relations \eqref{hb}--\eqref{S0}.

{\it Nodes $1\le i,j<n$.} These are type AI relations of Lemma \ref{L:AI-rels}: relation \eqref{hb} is \eqref{AI:df1}--\eqref{AI:df2} and \eqref{AI:dd}, relation \eqref{bb1} is \eqref{AI:ff1}, relation \eqref{bb2} is \eqref{AI:dd}, \eqref{AI:ff2} and \eqref{AI:dd=dd}, and relation \eqref{S0} is \eqref{AI:ff0}.

{\it Types BI and CI, pairs involving the node $n$.} Relation \eqref{hb} and the commutativity of the series $h_i(u)$ follow from Lemmas \ref{L:AI-rels}, \ref{L:BCD-triv} and \ref{L:BCD-df}, Remark \ref{R:dif} and Proposition \ref{P:z(u)}. The latter expresses $d_{n+1}(u)$ through $d_n(-u)$ and the central series $\ms{z}(u)$ in type CI. Relation \eqref{bb2} with $i=n$ is \eqref{B:ff1} and \eqref{C:ff1}, relation \eqref{bb1} is Proposition \ref{P:ff0}, and relation \eqref{S0} is \eqref{B:14} and its type CI counterpart in Lemma \ref{L:BCD-triv}.

{\it Type DI, pairs involving the node $n$.} The automorphism $\si$ of Lemma \ref{L:fork} exchanges the nodes $n-1$ and $n$ and fixes all other nodes. Hence relations \eqref{hb} and \eqref{bb2} with $i=j=n$, and every relation involving the node $n$ and a node $j\le n-2$, are images under $\si$ of type AI relations. The nodes $n-1$ and $n$ are not adjacent: relation \eqref{S0} for them is \eqref{D:ff1}, relation \eqref{hb} with $(i,j)=(n-1,n)$ follows from \eqref{hf4}, and relation \eqref{hb} with $(i,j)=(n,n-1)$ is its image under $\si$. The series $h_{n-1}(u)$ and $h_n(v)$ commute by \eqref{so4:4} and Corollary \ref{C:RE[m]}.

Moreover, $h_{i,2r}=0$ by Lemma \ref{L:heven}. Hence the coefficient extractions of Lemma \ref{L:LWZ} give \eqref{modes-hh}--\eqref{modes-bb} and \eqref{modes-S0}, while \eqref{modes-S1}--\eqref{modes-S2} hold by Lemma \ref{L:deg0}. Therefore the assignment
\equ{
\Xi_N : \mc{Y}^{\imath}(\mfg_N) \to SY^{\tw}(\mfg_N) , \qu
b^{\imath}_{i,r} \mapsto b_{i,r} , \qu h^{\imath}_{i,2r+1} \mapsto h_{i,2r+1}
\label{Xi-hom}
}
extends to a homomorphism of algebras, which is surjective by Lemma \ref{L:gen}\,(1).

\smallskip

{\it Step 2.} To show that $\Xi_N$ is injective, we compare the associated graded algebras. By \eqref{SY-graded-iso}, $\gr SY^{\tw}(\mfg_N) \cong U(\mfg_N[x]^{\om})$. By \eqref{Frr} and $F^{(\om,r)}_{ij} = g_{jj}\,(F_{ij} + (-1)^r\,\om(F_{ij}))\,x^r$, the degree-$r$ component of $\mfg_N[x]^{\om}$ is $\{\, Z x^r : Z \in \mfg_N ,\; \om(Z) = (-1)^r Z \,\}$. Its dimension is $\dim \mfg_N^{\om}$ for $r$ even and $\dim\mfg_N - \dim\mfg_N^{\om}$ for $r$ odd. Filter $\mc{Y}^{\imath}(\mfg_N)$ by $\deg b^{\imath}_{i,r} = r$ and $\deg h^{\imath}_{i,2r+1} = 2r+1$. The rescaling \eqref{LWZ-scaling} preserves this filtration, so \cite[Cor.~4.13]{LWZ25b} applies in the present normalisation (see also \cite[\S2.3]{Lu26}): $\gr\mc{Y}^{\imath}(\mfg_N)\cong U(\mfg^{\imath})$. Here the twisted current algebra $\mfg^{\imath}$ of the pair $(\mfg_N,\mfg_N^{\om})$ is a graded Lie algebra generated by the images $\bar b^{\imath}_{i,r}$ of the $b^{\imath}_{i,r}$, and its components have the same dimensions as those of $\mfg_N[x]^{\om}$.

By the proof of Lemma~\ref{L:gen},
\equ{
b_{i,r} \in \mr{F}_r
\qu\text{and}\qu
h_{i,2r+1}\in \mr F_{2r+1} ,
\label{Xi-symb}
}
where $\mr{F}_k$ denotes the $k$-th filtered component of $SY^{\tw}(\mfg_N)$. Hence $\Xi_N$ is filtered. By \eqref{SY-graded-map-1}--\eqref{SY-graded-map-2}, the map $\gr\Xi_N$ sends the generators $\bar b^{\imath}_{i,r}$ into $\mfg_N[x]^{\om}$. Therefore $\gr\Xi_N=U(\tau)$ for a homomorphism of graded Lie algebras $\tau : \mfg^{\imath} \to \mfg_N[x]^{\om}$.

Next, we show that $\tau$ is surjective. By \eqref{Xi-hom}, the right-hand sides of \eqref{SY-graded-map-1}--\eqref{SY-graded-map-2} are the images $\tau(\bar b^{\imath}_{i,r})$ and $\tau(\bar h^{\imath}_{i,2r+1})$; here $\bar h^{\imath}_{i,2r+1}\in\mfg^{\imath}$ by \cite[Cor.~4.13]{LWZ25b}. By the bracket relations displayed in the proof of Lemma \ref{L:gen}\,(1), every element of \eqref{F-basis} is an iterated Lie bracket of these images. The elements \eqref{F-basis} span $\mfg_N[x]^{\om}$, so $\tau$ is surjective.

The map $\tau$ preserves the degree, and the components of $\mfg^{\imath}$ and $\mfg_N[x]^{\om}$ are finite-dimensional, of equal dimensions. Hence $\tau$ is an isomorphism, and so is $\gr\Xi_N = U(\tau)$. Both filtrations are exhaustive and bounded below, so the injectivity of $\gr\Xi_N$ implies that of $\Xi_N$. Hence $\Xi_N$ is an isomorphism.

\smallskip

{\it Step 3.} By Proposition \ref{P:Serre} and relation \eqref{AI:fff}, the Drinfeld series of $SY^{\tw}(\mfg_N)$ also satisfy the Serre relations \eqref{S1}--\eqref{S2}. Indeed, let $x(u,v)$ and $y(u,v)$ be the series defined in Proposition \ref{P:Serre} for the respective type. By \eqref{Dr1}--\eqref{Dr2} we have
\eqn{
& x(u-\tfrac14,v+\tfrac14) = x_{n-1,n}(v,u) , \qq && y(u+\tfrac14,v-\tfrac14) = 2\,\frac{v-\frac12}{v-\frac14}\, x_{n,n-1}(v,u) \qq && \text{for } \mfo_{2n+1}, \\[-0.2em]
& x(u,v+1) = \sqrt{-2}\, x_{n-1,n}(v,u) , && y(u+1,v) = -\tfrac12\, x_{n,n-1}(v,u) && \text{for } \mfsp_{2n}, \\
& x(u,v+\tfrac12) = x_{n-2,j}(v,u) , && y(u+\tfrac12,v) = x_{j,n-2}(v,u) && \text{for } \mfo_{2n},
}
where in type DI we put $j=n-1$ for $k=n$ and $j=n$ for $k=n+1$. Hence the assignment
\eqn{
\Phi_N : \mr{Y}(\mfg_N) \to SY^{\tw}(\mfg_N) , \qu b_{i,r} \mapsto b_{i,r} , \qu h_{i,2r+1} \mapsto h_{i,2r+1}
}
respects the defining relations of $ \mr{Y}(\mfg_N)$ and extends to a homomorphism of algebras. Next, by the proof of Lemma \ref{L:LWZ}, the defining relations of $\mc{Y}^{\imath}(\mfg_N)$ hold in $\mr{Y}(\mfg_N)$ for the elements $b_{i,r}$ and $h_{i,2r+1}$. Hence the assignment
\eqn{
\Psi_N : \mc{Y}^{\imath}(\mfg_N) \to \mr{Y}(\mfg_N) , \qu
b^{\imath}_{i,r} \mapsto b_{i,r} , \qu h^{\imath}_{i,2r+1} \mapsto h_{i,2r+1}
}
extends to a homomorphism of algebras, which is surjective since its image contains the generators of $\mr{Y}(\mfg_N)$. Finally, $\Phi_N\circ\Psi_N=\Xi_N$ is injective by Step~2. Hence $\Psi_N$ is an isomorphism, and so is $\Phi_N=\Xi_N\circ\Psi_N^{-1}$.
\end{proof}

\begin{crl} \label{C:LWZ}
The maps $\Psi_N$ and $\Xi_N = \Phi_N \circ \Psi_N$ are isomorphisms of algebras,
\equ{
\mc{Y}^{\imath}(\mfg_N) \;\cong\; \mr{Y}(\mfg_N) \;\cong\; SY^{\tw}(\mfg_N) .
}
In particular, this establishes, for the split types BI, CI and DI, the conjecture stated in \cite[\S1.4]{LWZ25b}, namely that the twisted Yangians of split type in the Drinfeld presentation are isomorphic to the corresponding twisted Yangians in the $R$-matrix presentation of \cite{GR16}. Here the twisted Yangian $Y(\mfg_N,G)^{\tw}$ of \cite{GR16} is isomorphic to $SY^{\tw}(\mfg_N)$, the quotient of $X^{\tw}(\mfg_N)$ by the symmetry and unitarity relations of Section \ref{sec:tY}, by \cite[Thm.~4.1, Cor.~5.2]{GR16}.
\end{crl}

\begin{proof}
Both maps were shown to be isomorphisms in the proof of Theorem \ref{T:iso}. Note that Steps 1 and 2 of that proof, which show that $\Xi_N$ is an isomorphism, use neither the Serre relations \eqref{S1}--\eqref{S2} nor Proposition \ref{P:Serre}: they rely only on relations \eqref{hb}, \eqref{bb1}, \eqref{bb2}, \eqref{S0} and on Lemmas \ref{L:heven} and \ref{L:deg0}. Proposition \ref{P:Serre} enters only through Step 3, which shows that $\Psi_N$ is an isomorphism.
\end{proof}


\subsection{The extended twisted Yangian and linear bases} \label{sec:ext}

Theorem \ref{T:iso} has the following counterpart for the extended twisted Yangian $X^\tw(\mfg_N)$, in which the diagonal Gaussian series $d_j(u)$ are among the generators in place of the series $h_i(u)$.

\begin{crl} \label{C:iso-ext} \allowdisplaybreaks
The extended twisted Yangian $X^\tw(\mfg_N)$ is isomorphic to the associative unital $\C$-algebra $\mr{X}(\mfg_N)$ generated by elements $b_{i,r}$ with $1\le i \le n$ and $d_{j,r}$ with $1\le j \le n+1$ and $r\ge 0$ satisfying relations given by the following identities, where $b_i(u) = \sum_{r\ge0} b_{i,r}u^{-r-1}$ and $d_j(u) = g_{jj} + \sum_{r\ge0} d_{j,r}u^{-r-1}$
and each identity is understood after multiplication by the product of its scalar denominators or, equivalently, with rational functions expanded in the domain $|u|>|v|>|w|>|q|$:
\gat{
d_{i}(u)\,d_{i+1}(\ka-i-u) = d_{i+1}(u)\,d_{i}(\ka-i-u) \qu\text{for}\qu 1\le i < n, \label{dd-sym1}\\
\frac{u-\frac12}{u-\frac14}\, d_{n}(u-\tfrac14)\,d_{n+1}(-u-\tfrac14) = \frac{u+\frac12}{u+\frac14} \, d_{n+1}(u-\tfrac14)\,d_{n}(-u-\tfrac14) \qu\text{if}\qu \mfg_N = \mfo_{2n+1}, \label{dd-sym2}\\
d_{n}(u)\,d_{n+1}(-u) = d_{n+1}(u)\,d_{n}(-u) \qu\text{if}\qu \mfg_N = \mfo_{2n}\text{\;or } \mfsp_{2n} , \label{dd-sym3}
\\[0.5em]
[d_j(u),d_k(v) ] = 0 \qu\text{for all}\; j,k, \qq [d_j(u), b_i(v)] = 0 \qu \text{if}\qu j \le n \;\text{ and }\; j\ne i,i+1, \label{db0}
\\
d_i(u+\tfrac{\ka-i+1}{2})\, b_i(v) = \bigg(\frac{(u-v-\frac{1}{2})(u+v+\frac12)}{(u-v+\frac12)(u+v-\frac12)}\, b_i(v) + \bigg\{ \frac{u+\frac12}{u\,(u-v+\frac12)} \, b_i(u+\tfrac12) \bigg\}^{\!u}\,\bigg)\, d_i(u+\tfrac{\ka-i+1}2) , \label{db1}
\\[0.2em]
d_{i+1}(u+\tfrac{\ka-i-1}{2})\, b_i(v) = \bigg(\frac{(u-v+\frac{1}{2})(u+v-\frac12)}{(u-v-\frac12)(u+v+\frac12)}\, b_i(v) - \bigg\{ \frac{u-\frac12}{u\,(u-v-\frac12)} \, b_i(u-\tfrac12) \bigg\}^{\!u}\,\bigg)\, d_{i+1}(u+\tfrac{\ka-i-1}{2}) ,\label{db2}
\\[0.2em]
d_{j}(u+\tfrac{{\mr d}_n}{2})\, b_n(v) = \bigg(\frac{(u-v-\frac{{\mr d}_n}{2})(u+v+\frac{{\mr d}_n}2)}{(u-v+\frac{{\mr d}_n}2)(u+v-\frac{{\mr d}_n}2)}\, b_n(v) + \bigg\{ \frac{{\mr d}_n\,(u+\frac{{\mr d}_n}2)}{u\,(u-v+\frac{{\mr d}_n}2)} \, b_n(u+\tfrac{{\mr d}_n}2) \bigg\}^{\!u}\,\bigg)\, d_j(u+\tfrac{{\mr d}_n}2) , \label{db3}
}
together with relations \eqref{hb}--\eqref{xij}, where $h_i(u)$ is defined by \eqref{Dr1} and \eqref{Dr2}. Here relations \eqref{db1} and \eqref{db2} are imposed for $1\le i<n$, and \eqref{db1} also for $i=n$ when $\mfg_N = \mfo_{2n+1}$; relation \eqref{db3} is imposed with $j=n$ when $\mfg_N = \mfsp_{2n}$ and with $j=n-1,n$ when $\mfg_N=\mfo_{2n}$; and for $\mfg_N=\mfo_{2n}$ relation \eqref{db0} excludes the pair $(j,i)=(n-1,n)$.
\end{crl}

\begin{proof}
To distinguish the two algebras, we decorate the generators and series of $\mr{X}(\mfg_N)$ with a hat: $\wh{b}_{i,r}$, $\wh{d}_{j,r}$, $\wh{b}_i(u)$, $\wh{d}_j(u)$, and we write $\wh{h}_i(u)$ for the series defined by \eqref{Dr1}--\eqref{Dr2} with $\wh{d}_j(u)$ in place of ${d}_j(u)$.

We show first that the assignment
\eqn{
\Theta_N : \mr{X}(\mfg_N) \to X^\tw(\mfg_N), \qu \wh{b}_{i,r} \mapsto b_{i,r} , \qu \wh{d}_{j,r} \mapsto d_{j,r}
}
defines a homomorphism of algebras. Here $b_{i,r}$ and $d_{j,r}$ are the coefficients of the Drinfeld series \eqref{Dr1}--\eqref{Dr2} and of the diagonal Gaussian series of $X^\tw(\mfg_N)$, respectively. We need to check that the defining relations of $\mr{X}(\mfg_N)$ hold in $X^\tw(\mfg_N)$.

{\it Relations \eqref{dd-sym1}--\eqref{dd-sym3}.} Relation \eqref{dd-sym1} is \eqref{AI:dd=dd}. The series $d_j(u)$ commute, so relation \eqref{dd-sym2} and, when $\mfg_N = \mfsp_{2n}$, relation \eqref{dd-sym3} are equivalent to the evenness relation $h_n(u) = h_n(-u)$ of Lemma \ref{L:heven}. For $\mfg_N=\mfo_{2n}$, relation \eqref{dd-sym3} is the relation $\ms{z}(u)=\ms{z}(-u)$ rewritten by Proposition \ref{P:z(u)}.

{\it Relations \eqref{db0}--\eqref{db3}.} The commutativity relations \eqref{db0} are \eqref{AI:dd} together with Lemma \ref{L:BCD-triv}, \eqref{dd1}; the remaining cases follow from \eqref{z(u)}. Substitute $(u,v)\mapsto \big(u+\tfrac{\ka-i+1}{2}, v+\tfrac{\ka-i}{2}\big)$ in \eqref{AI:df1} and $(u,v) \mapsto \big(u+\tfrac{\ka-i-1}{2}, v+\tfrac{\ka-i}{2}\big)$ in \eqref{AI:df2}. By \eqref{Dr1}, this gives \eqref{db1} and \eqref{db2} for $1\le i <n$. In the same way, by \eqref{Dr2}, relation \eqref{db1} with $i=n$ for $\mfg_N = \mfo_{2n+1}$ is \eqref{hf1} with $(u,v)\mapsto(u+\tfrac14, v-\tfrac14)$. Relation \eqref{db3} is \eqref{hf3} with $(u,v)\mapsto(u+1,v)$ when $\mfg_N = \mfsp_{2n}$, and \eqref{hf4} with $(u,v)\mapsto(u+\tfrac12,v)$ when $\mfg_N=\mfo_{2n}$. The normalisation of $b_n(u)$ in \eqref{Dr2} plays no role here, since these relations are linear in the coefficients of $b_n(u)$.

{\it Relations \eqref{hb}--\eqref{xij}.} These hold by the proof of Theorem \ref{T:iso}, since all relations used there were established in the algebra $X^\tw(\mfg_N)$.

Hence $\Theta_N$ is a homomorphism of algebras, and it is surjective by Lemma \ref{L:gen}\,(2). We prove its injectivity in three steps. In Step~1 we show that the series $\wh h_i(u)$ are even. In Step~2 we use the resulting homomorphism $\wt{\Psi}_N \colon \mc{Y}^{\imath}(\mfg_N) \to \mr{X}(\mfg_N)$ to exhibit a spanning set of $\mr{X}(\mfg_N)$. In Step~3 we compare it with the image of an injective homomorphism $\Omega_N$ on $\mc{Y}^{\imath}(\mfg_N) \ot \C[\ms{t}_0, \ms{t}_1, \ldots]$.

\enlargethispage{0.5em}

\smallskip 

{\it Step 1.} The coefficients of the series $\wh{d}_j(u)$ pairwise commute by \eqref{db0}. Hence relations \eqref{dd-sym1}--\eqref{dd-sym3} give $\wh{h}_i(u) = \wh{h}_i(-u)$ for all $i$, cf.~the proof of Lemma \ref{L:heven}. The only case requiring comment is $\wh{h}_n(u)=\wh{d}^{\,-1}_{n-1}(u)\,\wh{d}_{n+1}(u)$ for $\mfg_N = \mfo_{2n}$. Here $\ka=n-1$, and \eqref{dd-sym1} with $i=n-1$ and \eqref{dd-sym3} give
\[
\wh{d}_{n-1}(-u)\,\wh{d}^{\,-1}_{n-1}(u) = \wh{d}_{n}(-u)\,\wh{d}^{\,-1}_{n}(u) = \wh{d}_{n+1}(-u)\,\wh{d}^{\,-1}_{n+1}(u).
\]
Thus each $\wh h_i(u)$ is an even series of the form $1 + \sum_{r\ge0} \wh{h}_{i,2r+1} u^{-2r-2}$. Its constant term is $1$, since $g_{nn}\, g_{n+1,n+1} = -1$ when $\mfg_N = \mfsp_{2n}$, and $g_{jj}=1$ otherwise.

\smallskip

{\it Step 2.} By Step~1 and the proof of Lemma \ref{L:LWZ}, the defining relations \eqref{modes-hh}--\eqref{modes-S2} of $\mc{Y}^{\imath}(\mfg_N)$ hold for the elements $\wh{b}_{i,r}$ and $\wh{h}_{i,2r+1}$. Hence the assignment
\eqn{
\wt{\Psi}_N : \mc{Y}^{\imath}(\mfg_N) \to \mr{X}(\mfg_N) , \qu
b^{\imath}_{i,r} \mapsto \wh{b}_{i,r} , \qu h^{\imath}_{i,2r+1} \mapsto \wh{h}_{i,2r+1}
}
extends to a homomorphism of algebras, whose image is the unital subalgebra of $\mr{X}(\mfg_N)$ generated by the elements $\wh{b}_{i,r}$ and $\wh{h}_{i,2r+1}$.

Define the series
\equ{
\wh{\ms{c}}(u) := p(u)\,\wh{d}_1(u)\,\wh{d}^{\,-1}_1(\ka-u) , \qq
\wh{\ms{z}}(u) := \begin{cases}
\wh{d}_{n+1}(-u)\, \wh{d}_{n+1}(u) & \text{for } \mfg_N = \mfo_{2n+1}, \\
\theta\, \wh{d}_{n}(-u)\, \wh{d}_{n+1}(u) & \text{for } \mfg_N = \mfsp_{2n},\, \mfo_{2n} ,
\end{cases}
\label{cz-hat}
}
cf.~\eqref{c(u)} and Proposition \ref{P:z(u)}. Then $\Theta_N(\wh{\ms c}(u)) = \ms{c}(u)$ and $\Theta_N(\wh{\ms z}(u)) = \ms{z}(u)$. A direct computation gives $p(u)\,p(\ka-u) = 1$, whence $\wh{\ms c}(u)\,\wh{\ms c}(\ka-u) = 1$. Moreover, $\wh{\ms z}(-u) = \wh{\ms z}(u)$ by \eqref{dd-sym3}; for $\mfg_N = \mfo_{2n+1}$ this holds trivially. Compare the coefficients in these two identities. This shows that the elements $\wh{\ms c}_r$ with $r$ odd (resp.~$\wh{\ms z}_r$ with $r$ even) are polynomials in the elements $\wh{\ms c}_s$ with $s<r$ even (resp.~$\wh{\ms z}_s$ with $s<r$ odd).

We claim that the coefficients of all the series $\wh{d}_j(u)$ are polynomials in the coefficients of the series $\wh h_i(u)$, $\wh{\ms c}(u)$ and $\wh{\ms z}(u)$. Indeed, by \eqref{Dr1}--\eqref{Dr2},
\equ{
\wh{d}_{i+1}(u) = \wh{d}_i(u)\, \wh{h}_i\big(u - \tfrac{\ka-i}{2}\big) \;\text{ for } 1\le i<n, \qq
\wh{d}_{n+1}(u) = \begin{cases}
\tfrac{u-\frac14}{u}\; \wh{d}_n(u)\, \wh{h}_n(u+\tfrac14) & \text{for } \mfo_{2n+1} , \\
-\,\wh{d}_n(u)\, \wh{h}_n(u) & \text{for } \mfsp_{2n} , \\
\wh{d}_{n-1}(u)\, \wh{h}_n(u) & \text{for } \mfo_{2n} ,
\end{cases}
\label{d-recur}
}
so every series $\wh{d}_j(u)$ is expressed through $\wh{d}_1(u)$ and the series $\wh{h}_i(u)$. Substitute these expressions into \eqref{cz-hat} and take the coefficients of $u^{-r-1}$. Since the elements $\wh d_{j,r}$ pairwise commute, this gives
\eqn{
\wh{\ms{c}}_r = \big(1 + (-1)^r\big)\, \wh{d}_{1,r} + \Pi_r , \qq
\wh{\ms{z}}_r = \big(1 - (-1)^{r}\big)\, \wh{d}_{1,r} + \Pi'_r ,
}
where $\Pi_r$ and $\Pi'_r$ are polynomials in the elements $\wh{d}_{1,s}$ with $s < r$ and the coefficients of the series $\wh{h}_i(u)$ (and scalars). By induction on $r$, every $\wh{d}_{1,r}$, and hence every $\wh{d}_{j,r}$, is a polynomial in the coefficients of $\wh{\ms c}(u)$, $\wh{\ms z}(u)$ and the $\wh{h}_i(u)$. This proves the claim.

Next, let $\mc{A}$ be the image of $\wt\Psi_N$, and let $\mc{D}$ be the subalgebra of $\mr X(\mfg_N)$ generated by the coefficients of $\wh d_1(u)$. These coefficients commute with the elements $\wh h_{i,2r+1}$ by \eqref{db0}. Relations \eqref{db0}, \eqref{db1} and \eqref{db3} with $j=1$ (the latter for $\mfg_N=\mfsp_2$, and for $i=2$ when $\mfg_N=\mfo_4$) write each product $\wh d_1(x)\,\wh b_i(v)$, up to a shift of $x$, as a linear combination of series $\wh b_i$ with scalar rational coefficients, followed by $\wh d_1(x)$. Take the coefficients in these relations. Each product $\wh d_{1,r}\,\wh b_{i,s}$ is then a linear combination of products $\wh b_{i,s'}\,\wh d_{1,r'}$. The combination is finite, since the rational coefficients, expanded in the prescribed domain, have total degree bounded from above. Hence $\mc{D}\mc{A}\subseteq\mc{A}\mc{D}$, and $\mc{A}\mc{D}$ is a subalgebra. By \eqref{d-recur}, it contains the coefficients of all the series $\wh d_j(u)$, and hence equals $\mr X(\mfg_N)$. By the claim above, $\mc{D}$ lies in the commutative algebra generated by the coefficients of the $\wh h_i(u)$, $\wh{\ms c}(u)$ and $\wh{\ms z}(u)$. The $\wh h$-coefficients lie in $\mc{A}$, and the coefficients $\wh{\ms c}_{2r+1}$ and $\wh{\ms z}_{2r}$ are polynomials in the others. Therefore
\equ{
\mr{X}(\mfg_N) = \mathrm{span}_\C \big\{\, \wt{\Psi}_N(y)\; m \;:\; y \in \mc{Y}^{\imath}(\mfg_N) ,\; m \text{ a monomial in the } \wh{\ms{c}}_{2r},\, \wh{\ms z}_{2r+1} \big\} .
\label{X-span}
}

{\it Step 3.} Let $\mc{P} := \C[\ms{t}_0, \ms{t}_1, \ms{t}_2, \ldots]$ be the algebra of polynomials in variables $\ms{t}_r$. Let $\ms y(u)$ and $\wt S(u)$ be as in \S\ref{sec:central}, and let $\iota : SY^{\tw}(\mfg_N) \to X^{\tw}(\mfg_N)$ be the embedding $S(u)\mapsto\wt S(u)$. Recall that $S(u)$ and $\wt S(u)$ have the same matrices $F(u)$ and $E(u)$, while $d_j(u)=\ms y(u)\,\wt d_j(u)$ for all $j$. The series \eqref{Dr1}--\eqref{Dr2} involve the $d_j(u)$ only through ratios of two of them at the same argument, cf.~\eqref{d-recur}. Hence the Drinfeld series of $S(u)$ and $\wt S(u)$ coincide, and the map $\wt\Xi_N:=\iota\circ\Xi_N : \mc{Y}^{\imath}(\mfg_N) \to X^{\tw}(\mfg_N)$ sends the generators $b^{\imath}_{i,r}$ and $h^{\imath}_{i,2r+1}$ to the coefficients $b_{i,r}$ and $h_{i,2r+1}$ of the Drinfeld series of $X^{\tw}(\mfg_N)$. By Steps 1 and 2 of the proof of Theorem \ref{T:iso}, it is an isomorphism onto $\iota(SY^{\tw}(\mfg_N))$. By \S\ref{sec:central}, the elements $\ms c_{2r}$ and $\ms z_{2r+1}$ with $r\ge0$ are algebraically independent and generate the same algebra as the elements $\ms y_r$. Since $X^{\tw}(\mfg_N)$ is the tensor product of $\iota(SY^{\tw}(\mfg_N))$ and the polynomial algebra in the central elements $\ms y_r$, the assignment
\eqn{
\Omega_N : \mc{Y}^{\imath}(\mfg_N) \ot \mc{P} \to X^{\tw}(\mfg_N) , \qu
y \ot 1 \mapsto \wt{\Xi}_N(y) , \qu 1 \ot \ms{t}_{2r} \mapsto \ms{c}_{2r} , \qu 1 \ot \ms{t}_{2r+1} \mapsto \ms{z}_{2r+1}
}
defines an isomorphism of algebras. Since $\Theta_N \circ \wt\Psi_N = \wt\Xi_N$ and $\Theta_N$ maps $\wh{\ms c}(u)$ and $\wh{\ms z}(u)$ to $\ms c(u)$ and $\ms z(u)$, we have
\equ{
\Theta_N\big( \wt{\Psi}_N(y)\, m(\wh{\ms{c}}_{2r}, \wh{\ms z}_{2r+1}) \big) = \Omega_N\big( y \ot m(\ms{t}_{2r}, \ms{t}_{2r+1}) \big) .
\label{Th-Om}
}

Finally, the elements $\wh{\ms c}_{r}$ and $\wh{\ms z}_{r}$ pairwise commute by \eqref{db0}. Hence there is a well-defined linear map $L : \mc{Y}^{\imath}(\mfg_N) \ot \mc{P} \to \mr{X}(\mfg_N)$ given by $y \ot m(\ms{t}_{2r}, \ms{t}_{2r+1}) \mapsto \wt{\Psi}_N(y)\, m(\wh{\ms{c}}_{2r}, \wh{\ms z}_{2r+1})$ for monomials $m$. It is surjective by \eqref{X-span}, and $\Theta_N \circ L = \Omega_N$ by \eqref{Th-Om}. Since $\Omega_N$ is injective, $L$ is bijective, and $\Theta_N = \Omega_N \circ L^{-1}$ is injective. Therefore $\Theta_N$ is an isomorphism.
\end{proof}

\begin{remark} \label{R:Omega}
The proof shows that $\Omega_N$ is an isomorphism, $X^{\tw}(\mfg_N) \cong \mc{Y}^{\imath}(\mfg_N) \ot \C[\ms{t}_0, \ms{t}_1, \ldots]$, refining the decomposition \eqref{X=SYcz}.
\end{remark}

The isomorphisms of Corollary \ref{C:LWZ} and Remark \ref{R:Omega} transport the Poincar\'e--Birkhoff--Witt theorem for $\mc{Y}^{\imath}(\mfg_N)$ \cite[Thm.~4.12]{LWZ25b} to the twisted Yangians in the $R$-matrix presentation. This gives linear bases in the Drinfeld generators. For each positive root $\al$ of $\mfg_N$ fix, following \cite[\S4.2]{LWZ25b}, a real root vector $b^{\imath}_{\al,r}\in\mc{Y}^{\imath}(\mfg_N)$ $(r\ge0)$, obtained as an iterated commutator of the generators $b^{\imath}_{i,s}$; for a simple root $\al=\al_i$ we have $b^{\imath}_{\al_i,r}=b^{\imath}_{i,r}$.

\begin{crl} \label{C:PBW}
Let $b_{\al,r}$ and $h_{i,2r+1}$ denote the images in $SY^{\tw}(\mfg_N)$ of the root vectors $b^{\imath}_{\al,r}$ and $h^{\imath}_{i,2r+1}$ under the isomorphism $\Xi_N$ of Corollary \ref{C:LWZ}; for a simple root these are the Drinfeld generators \eqref{Dr1}--\eqref{b-h-series}, and for a general positive root $b_{\al,r}$ is the corresponding iterated commutator of the $b_{i,s}$.
\begin{enumerate}
\item Given any total ordering, the ordered monomials in the elements $b_{\al,r}$ and $h_{i,2r+1}$, where $\al$ ranges over the positive roots of $\mfg_N$, $1\le i\le n$ and $r\ge0$, form a linear basis of $SY^{\tw}(\mfg_N)$.
\item Writing $b_{\al,r}$ and $h_{i,2r+1}$ now for the corresponding elements of $X^{\tw}(\mfg_N)$, obtained through the homomorphism $\wt\Xi_N$ of the proof of Corollary \ref{C:iso-ext}, the ordered monomials in the elements $b_{\al,r}$, $h_{i,2r+1}$ and the central coefficients $\ms{c}_{2r}$, $\ms{z}_{2r+1}$ $(r\ge0)$ form a linear basis of $X^{\tw}(\mfg_N)$.
\end{enumerate}
\end{crl}

\begin{proof}
By \cite[Thm.~4.12]{LWZ25b}, transported by the invertible rescaling \eqref{LWZ-scaling} (and componentwise for $\mfg_N = \mfo_4$, cf.~Step~1 of the proof of Theorem \ref{T:iso}), the ordered monomials in the $b^{\imath}_{\al,r}$ and $h^{\imath}_{i,2r+1}$ form a basis of $\mc{Y}^{\imath}(\mfg_N)$. Statement (1) follows by applying the isomorphism $\Xi_N$ of Corollary \ref{C:LWZ}, and (2) by applying the isomorphism $\Omega_N$ of Remark \ref{R:Omega} to the products of this basis with the monomial basis of $\C[\ms{t}_0, \ms{t}_1, \ldots]$, since $\Omega_N$ restricts to $\wt\Xi_N$ on the first tensor factor and sends $\ms{t}_{2r} \mapsto \ms{c}_{2r}$, $\ms{t}_{2r+1} \mapsto \ms{z}_{2r+1}$.
\end{proof}

The isomorphism $\Omega_N$ of Remark \ref{R:Omega} and the bases of Corollary \ref{C:PBW} do not depend on Proposition \ref{P:Serre} either. They rest on $\Xi_N$ and on the construction of $\Omega_N$ in Step~3 of the proof of Corollary \ref{C:iso-ext}. That construction uses only $\Xi_N$, the decomposition of $X^{\tw}(\mfg_N)$ recalled before \eqref{X-graded-iso} and the central series of \S\ref{sec:central}. Proposition \ref{P:Serre} enters the proof of Corollary \ref{C:iso-ext} only through the verification that $\Theta_N$ is a homomorphism, since relations \eqref{S1}--\eqref{S2} are among the defining relations of $\mr{X}(\mfg_N)$.


\subsection{Coideal structure} \label{sec:coideal}

By the embedding \eqref{S->TGT}, the special twisted Yangian $SY^{\tw}(\mfg_N,G)$ is a subalgebra of the Yangian $Y(\mfg_N)$. By \cite[Prop.~3.3]{AMR06}, $Y(\mfg_N)$ is a Hopf subalgebra of the extended Yangian $X(\mfg_N)$, with $\Delta(t_{ij}(u)) = \sum_{a} t_{ia}(u) \ot t_{aj}(u)$. Applying $\Delta$ to the entries of \eqref{S->TGT}, we obtain
\equ{
\Delta\big(s_{ij}(u)\big) = \sum_{1\le a,b\le N} t_{ia}(u-\tfrac\ka2)\, t_{jb}(\tfrac\ka2-u) \ot s_{ab}(u) \;\in\; Y(\mfg_N) \ot SY^{\tw}(\mfg_N,G) [[u^{-1}]], \label{coid-Ds}
}
so that $SY^{\tw}(\mfg_N,G)$ is a left coideal subalgebra of $Y(\mfg_N)$. This is the transposed-presentation form of the coaction of \cite[Prop.~3.2 and Cor.~3.2]{GR16}.

We expand \eqref{coid-Ds} in the low Drinfeld modes. Let $Y(\mfg_N)_{Q_+}$ be the span of the elements of strictly positive root degree, where $\deg t_{ij}(u) = \eps_i-\eps_j$ and $\eps_1,\ldots,\eps_N$ are the weights of the vector representation, $\eps_{\bar\imath} = -\eps_i$. We write $t_{ij}(u) = \del_{ij} + \sum_{r\ge 1} t^{(r)}_{ij}\, u^{-r}$ and $\msf{s}_{ab} := s^{(1)}_{ab}-s^{(1)}_{ba}$, and we denote by $t^{(1)}_{[l,m]}$ the raising generator of the unordered pair $\{l,m\}$, that is, $t^{(1)}_{[l,m]} = t^{(1)}_{l,m}$ for $l<m$ and $t^{(1)}_{[l,m]} = t^{(1)}_{m,l}$ for $l>m$.

\begin{prop} \label{P:coideal}
For $1\le i<n$ we have
\ali{
\Delta(b_{i,0}) &\equiv \msf{f}_{i,0} \ot 1 + 1 \ot b_{i,0} \pmod{Y(\mfg_N)_{Q_+} \ot SY^{\tw}(\mfg_N)} , \label{coid-low} \\[2pt]
\Delta(h_{i,1}) &= \msf{h}_{i,1} \ot 1 + 1 \ot h_{i,1} + \Omega_i , \label{coid-h1}
}
where $\msf{f}_{i,0} = t^{(1)}_{i+1,i} = F_{i+1,i}$, $\msf{h}_{i,1} \in Y(\mfg_N)$, and
\equ{
\Omega_i = \sum_{l\ne i+1} t^{(1)}_{[l,\,i+1]} \ot \msf{s}_{l,i+1} \;-\; \sum_{l\ne i} t^{(1)}_{[l,\,i]} \ot \msf{s}_{l,i} \;\in\; Y(\mfg_N)_{Q_+} \ot SY^{\tw}(\mfg_N) . \label{coid-omega}
}
\end{prop}

\begin{proof}
The argument is based on the two lowest coefficients of \eqref{coid-Ds}. For split pairs $G$ is diagonal, with $g_{kk}=1$ for $k\le n$. Taking the coefficients of $u^{-1}$ and $u^{-2}$ in \eqref{coid-Ds}, we obtain, for arbitrary $1\le i,j\le N$,
\ali{
\Delta\big(s^{(1)}_{ij}\big) &= \big(g_{jj}\,t^{(1)}_{ij}-g_{ii}\,t^{(1)}_{ji}\big) \ot 1 + 1 \ot s^{(1)}_{ij} , \label{coid-Ds1} \\[2pt]
\Delta\big(s^{(2)}_{ij}\big) &= \msf{q}_{ij} \ot 1 + 1 \ot s^{(2)}_{ij} + \sum_{1\le a \le N} \big( t^{(1)}_{ia} \ot s^{(1)}_{aj} - t^{(1)}_{ja} \ot s^{(1)}_{ia} \big) , \label{coid-Ds2}
}
with $\msf{q}_{ij} \in Y(\mfg_N)$. Here the constant term $g_{ab}$ of $s_{ab}(u)$ contributes only to the summands with second leg~$1$. The shifts $\pm\ka/2$ of the argument enter only the second modes of the $t$-legs, and are absorbed into $\msf{q}_{ij}$.

{\it The element $b_{i,0}$.} By the proof of Lemma \ref{L:deg0}, $b_{i,0} = s^{(1)}_{i+1,i}$. Since $g_{ii}=g_{i+1,i+1}=1$, relation \eqref{coid-Ds1} gives $\Delta(b_{i,0}) = \big(t^{(1)}_{i+1,i}-t^{(1)}_{i,i+1}\big) \ot 1 + 1 \ot b_{i,0}$. The term $t^{(1)}_{i,i+1} \ot 1$ has positive root degree, so \eqref{coid-low} follows.

{\it The element $h_{i,1}$.} By Lemma \ref{L:heven}, $h_{i,0}=0$, so the shift of the argument in \eqref{Dr1} does not affect $h_{i,1}$. Hence $h_{i,1}$ is the coefficient of $u^{-2}$ in the product $d^{-1}_i(u)\, d_{i+1}(u)$. For $k\le n$, the series $d_k(u)$ has constant term $g_{kk}=1$, and its next two coefficients are $d_{k,0} = s^{(1)}_{kk}$ and $d_{k,1} = s^{(2)}_{kk} - \sum_{l<k} s^{(1)}_{kl} s^{(1)}_{lk}$. Hence
\equ{
h_{i,1} = d_{i+1,1} - d_{i,1} + d_{i,0}^2 - d_{i,0}\, d_{i+1,0} . \label{coid-h1elt}
}
We apply $\Delta$ by \eqref{coid-Ds1}--\eqref{coid-Ds2} and keep the cross-terms, that is, the summands with both tensor legs nontrivial. By \eqref{coid-Ds1}, $\Delta(d_{k,0}) = 1\ot d_{k,0}$, so $d_{i,0}^2$ and $d_{i,0}\,d_{i+1,0}$ contribute no cross-terms. In the quadratic terms $s^{(1)}_{kl}\,s^{(1)}_{lk}$ we have $l<k\le i+1\le n$, so the diagonal entries of $G$ entering \eqref{coid-Ds1} there are $1$. Hence $\Delta(s^{(1)}_{kl}) = A_{kl}\ot1 + 1\ot s^{(1)}_{kl}$ with $A_{kl} = t^{(1)}_{kl}-t^{(1)}_{lk} = -A_{lk}$, and the cross-terms of $s^{(1)}_{kl}\,s^{(1)}_{lk}$ are $A_{kl}\ot\msf{s}_{lk}$. By \eqref{coid-Ds2}, the cross-terms of $s^{(2)}_{kk}$ are $\sum_{a} t^{(1)}_{ka}\ot\msf{s}_{ak}$. Therefore the cross-terms of $d_{k,1}$ are $\sum_{l\ne k} t^{(1)}_{[l,k]}\ot\msf{s}_{lk}$, and \eqref{coid-h1elt} yields \eqref{coid-omega}. Each $t^{(1)}_{[l,m]}$ has row index smaller than column index, hence positive root degree, so $\Omega_i \in Y(\mfg_N)_{Q_+} \ot SY^{\tw}(\mfg_N)$, as required by the coideal property.
\end{proof}

Thus $b_{i,0}$ is primitive modulo $Y(\mfg_N)_{Q_+}$, so $SY^{\tw}(\mfg_N)$ quantises the left Lie coideal structure of \cite{GR16}. By contrast, $\Delta(h_{i,1})$ carries the cross-terms $\Omega_i$, the coideal counterpart of the non-cocommutativity of the Drinfeld--Cartan generators of $Y(\mfg_N)$. The $2n$ elements $b_{i,0}, h_{i,1}$ $(1\le i\le n)$ generate $SY^{\tw}(\mfg_N)$. Indeed, the relations $[h_{i,1},b_{i,s}]=2c_{ii}\,b_{i,s+1}$ and \eqref{modes-bb} with $i=j$ recover inductively all $b_{i,r}$ and $h_{i,2r+1}$, cf.~\cite[Lem.~2.8]{Lu26}, and these generate by Lemma \ref{L:gen}. Hence $\Delta$ is determined on $SY^{\tw}(\mfg_N)$ by its values on these elements, which \eqref{coid-Ds} provides.

These estimates agree with \cite[Thm.~C]{Lu26}. There, in the right-coideal convention, $\Delta(h_i(u)) \equiv h_i(u) \ot \xi_i(u)\,\xi_i(-u)$ modulo terms whose second tensor factor has strictly positive root degree, where $\xi_i(u)$ is the Drinfeld--Cartan current of $Y(\mfg_N)$. The comparison presumes the compatibility of the Drinfeld and $R$-matrix coproducts under the isomorphism of \cite{JLM18}.

The sharper form conjectured by Wang and Zhang, see \cite[Conj.~5.2]{Lu26}, confines the correction to the upper Borel subalgebra and is proved there only for $\mfsl_2$, via the $R$-matrix presentation \cite{LWZ25a}. Through \eqref{coid-Ds} and Corollary \ref{C:LWZ} it reduces to an identity in the Gaussian generators in types B, C and D.

\enlargethispage{1em}

\begin{remark} \label{R:quantum}
The Drinfeld-type relations of Theorem \ref{T:iso} are the rational counterparts of the current relations of the affine $\imath$quantum groups (the coideal subalgebras of quantum affine algebras attached to quantum symmetric pairs) whose Drinfeld presentations are due to Lu and Wang \cite{LW21} in split ADE type and to Zhang \cite{Zh22} in split BCFG type; the relations of \cite{LWZ25b} are degenerations of the latter. 

In the quantum setting, however, the two features exploited here are largely absent. First, those presentations were obtained from the Drinfeld--Jimbo presentation through relative braid group symmetries rather than from a reflection equation, and a $q$-analogue of the present derivation exists so far only in split type AI, where Lu \cite{Lu24} identified the twisted $q$-Yangian of Molev, Ragoucy and Sorba \cite{MRS03} with the affine $\imath$quantum group. In the untwisted case, such an isomorphism is known for the quantum affine algebras of types $B$, $C$ and~$D$ \cite{JLM20a, JLM20b}, the $q$-analogue of \cite{JLM18}. 
Second, the current Serre relations of \cite{LW21, Zh22} carry nontrivial lower-order correction terms, and one may ask whether they too admit a closed current form governed by a single building block, in the spirit of \eqref{S1}--\eqref{xij}. A reflection-equation presentation of the affine $\imath$quantum groups of types B, C and D, together with a Gaussian-decomposition derivation of their Drinfeld presentation in closed current form, would supply such a $q$-analogue.
\end{remark}


{\noindent \bf Acknowledgements.} This research has been carried out in the framework of the “Universities’ Excellence Initiative” programme by the Ministry of Education, Science and Sports of the Republic of Lithuania under the agreement with the Research Council of Lithuania (project No. S-A-UEI-23-6). AI tools were used for proofreading the manuscript. The author takes full responsibility for the content of the manuscript.


\appendix

\allowdisplaybreaks

\section{Fundamental representation} \label{app:Rep}

Replacing the normalised matrix $T(u)$ in \eqref{S->TGT} by the generating matrix of $X(\mfg_N)$ gives a homomorphism $X^\tw(\mfg_N,G) \to X(\mfg_N)$, whose image is isomorphic to the quotient of $X^\tw(\mfg_N,G)$ by the symmetry relation \eqref{sym}, cf.~\cite[\S3.2 and Thm.~4.2]{GR16}. We compose it with the vector representation $T(u) \mapsto \mc{T}(u)$ of $X(\mfg_N)$, where
\equ{
\mc{T}(u) := \sum_{1\le i,j\le N} E_{ij} \ot \big(\del_{ij} I + u^{-1}E_{ij} - (u+\ka)^{-1} \theta_{ij} E_{\bar\jmath\bar\imath} \big) .
}
This gives the representation $S(u) \mapsto \mc{T}(u-\tfrac\ka2)\,G(u)\,\mc{T}^t(\tfrac\ka2-u)$ of $X^\tw(\mfg_N,G)$, which we call its fundamental representation. Since $\mc{T}(u)$ is not normalised, the central series $\ms{c}(u)$ is represented by the identity matrix but $\ms{z}(u)$ by a nontrivial scalar series, so the representation does not factor through $SY^\tw(\mfg_N,G)$. For the split pairs, the Gaussian decomposition \eqref{S=FDE} and \eqref{Dr1}--\eqref{Dr2} give the following images.

\begin{lemma} \label{L:rep} 
Set $\vka := N-2n$. The mapping 
\ali{
d_i(u) &\mapsto \bigg[\frac{u}{u-\frac14}\bigg]^{\vka} \bigg[ I - \sum_{k=1}^{i-1} \bigg[ \frac{E_{kk}}{(u-\frac{\ka}{2})^2} + \frac{E_{\bar{k}\bar{k}}}{(u+\frac{\ka}{2})^2} - \frac{\vka\,u\,E_{k\bar k}}{(u-\frac{\ka}{2})^2(u+\frac{\ka}{2})^2} \bigg] \nn\\*
& \hspace{2.4cm} - \frac{i\,E_{ii}}{(u-\frac{\ka}{2})(u-\frac{\ka}{2}+i-1)} + \frac{(2\ka-i)\,E_{\bar\imath\bar\imath}}{(u+\frac{\ka}{2})(u-\frac{3\ka}{2}+i-1)} \nn\\* 
& \hspace{2.4cm} + \vka\,\frac{((u-\ka)(u+i-1)+\frac{\ka}{4}(\ka-2)+\frac{i}{2}(i-1))\,E_{i\bar\imath} }{(u-\frac{\ka}{2})(u+\frac{\ka}{2})(u-\frac{\ka}{2}+i-1)(u-\frac{3\ka}{2}+i-1)}\bigg] \hspace{1.75cm} \text{for } 1\le i \le n , \label{rep-di} 
\\[0.5em]
d_{n+1}(u) &\mapsto \theta\,I - \theta\sum_{k=1}^{n+\vka-1} \bigg[ \frac{E_{kk}}{(u-\frac{\ka}{2})^2} + \frac{E_{\bar{k}\bar{k}}}{(u+\frac{\ka}{2})^2} - \frac{\vka\,u\,E_{k\bar{k}}}{(u-\frac{\ka}{2})^2(u+\frac{\ka}{2})^2} \bigg] \nn\\*
& \hspace{1cm} - \frac{1}{(u-\frac{\ka}{2})(u+\frac{\ka}{2})}
\begin{cases}
\dfrac{u^2 + \frac{u}{2} + \frac{\ka^2-1}{4}}{(u-\frac{\ka-1}{2})(u+\frac{\ka+1}{2})}\,E_{n+1,n+1} & \text{for } X^\tw(\mfo_{2n+1}), \\[1em]
\dfrac{(\ka-2)(u-\frac{\ka}{2}-1)-(u+\frac{\ka}{2})}{u-\frac{\ka}{2}}\,E_{nn} \\ \qu - \dfrac{(\ka+2)(u+\frac{\ka}{2}-1)+(u-\frac{\ka}{2})}{u+\frac{\ka}{2}}\,E_{n+1,n+1} \hspace{1cm} & \text{for } X^\tw(\mfsp_{2n}), \\[1em]
(1-\ka)\,E_{nn} + (1+\ka)\,E_{n+1,n+1} & \text{for } X^\tw(\mfo_{2n}), 
\end{cases}
\label{rep-dn+1} \\[0.5em]
h_i(u) &\mapsto I + \frac{(i-1)\,E_{ii} - (i+1)\,E_{i+1,i+1} }{(u-\frac{i}2)(u+\frac{i}2)} + \frac{(2\ka-i-1)\,E_{\bar\imath-1,\bar\imath-1} - (2\ka-i+1)\,E_{\bar\imath\bar\imath} }{(u-\ka+\frac{i}2)(u+\ka-\frac{i}2)} \nn\\*[0.25em] & \qq - \vka\, \frac{(u^2 - \frac{\ka(i-1)}{2} + \frac{i(i-2)}{4}) E_{i,\bar\imath} - (u^2 - \frac{\ka(i+1)}{2} + \frac{i(i+2)}{4}) E_{i+1,\bar\imath-1}}{(u-\frac{i}{2})(u+\frac{i}{2})(u-\ka+\frac{i}2)(u+\ka-\frac{i}2)} \hspace{1.65cm} \text{for } 1\le i < n, \label{rep-hi} \\[0.5em]
h_n(u) &\mapsto I + \begin{cases}
\dfrac{(n-1)\,E_{nn}}{(u-\frac{n}{2})(u+\frac{n}{2})} - \dfrac{n\,E_{n+2,n+2}}{(u+\frac{n-1}{2})(u-\frac{n-1}{2})} \\[0.8em] \qu - \dfrac{(u^2 + \frac{n^2-n-1}{4})\,E_{n+1,n+1} + (u^2 - \frac{n^2-n+1}{4})\,E_{n,n+2} }{(u-\frac{n}{2})(u+\frac{n}{2})(u+\frac{n-1}{2})(u-\frac{n-1}{2})} 
& \text{for } X^\tw(\mfo_{2n+1}), \\[1em]
\dfrac{2\,(\ka-2)\,E_{nn} - 2\,(\ka+2)\,E_{n+1,n+1} }{(u-\frac{\ka}2)(u+\frac{\ka}2)} & \text{for } X^\tw(\mfsp_{2n}), \\[1em]
\dfrac{(\ka-1)\,(E_{n-1,n-1}+E_{nn}) - (\ka+1)\,(E_{n+1,n+1}+E_{n+2,n+2}) }{(u-\frac{\ka}2)(u+\frac{\ka}2)} \hspace{1.1cm} & \text{for } X^\tw(\mfo_{2n}), 
\end{cases}
\label{rep-hn}
\\[0.5em]
b_i(u) &\mapsto \frac{E_{i+1,i}}{u+\frac{i}{2}} - \frac{E_{i,i+1}}{u-\frac{i}{2}} + \frac{E_{\bar\imath -1,\bar\imath}}{u-\ka+\frac{i}{2}} - \frac{E_{\bar\imath,\bar\imath-1}}{u+\ka-\frac{i}{2}} \nn\\*[0.25em] & \qq + \frac{\vka}{2} \bigg[ \frac{E_{i,\bar\imath-1}}{(u-\frac{i}2)(u+\ka-\frac{i}2)} + \frac{E_{i+1,\bar\imath}}{(u+\frac{i}2)(u-\ka+\frac{i}2)} \bigg] \hspace{3.9cm} \text{for } 1\le i < n,
\label{rep-bi} 
\\[0.7em]
b_n(u) &\mapsto \begin{cases}
\dfrac{E_{n+1,n}}{u+\frac{n}{2}} - \dfrac{(u+\frac{n-2}{2}) E_{n,n+1}}{(u-\frac{n}{2})(u+\frac{n-1}{2})} + \dfrac{(u+\frac{n+1}{2}) E_{n+1,n+2}}{(u+\frac{n}{2})(u-\frac{n-1}{2})} - \dfrac{E_{n+2,n+1}}{u+\frac{n-1}{2}} \hspace{1.02cm} & \text{for } X^\tw(\mfo_{2n+1}),
\\[1em]
-\sqrt{-2} \,\bigg[ \dfrac{E_{n+1,n}}{u+\frac{\ka}{2}} + \dfrac{E_{n,n+1}}{u-\frac{\ka}{2}} \bigg] & \text{for } X^\tw(\mfsp_{2n}) ,
\\[1em]
\dfrac{E_{n+1,n-1}-E_{n+2,n}}{u+\frac{\ka}{2}} - \dfrac{E_{n-1,n+1}-E_{n,n+2}}{u-\frac{\ka}{2}} & \text{for } X^\tw(\mfo_{2n}),
\end{cases}
\label{rep-bn}
}
defines the fundamental representation of $X^\tw(\mfg_N)$.
\end{lemma}


\section{Auxiliary relations} \label{app:Aux}

\setlength{\medmuskip}{0.5mu plus 0.5mu minus 0.5mu} 

This appendix collects relations used in the proof of Proposition \ref{P:Serre}. All series in them are Gaussian series of the subalgebra $X^{\tw[n-2]}(\mfg_N)$, which is isomorphic to $X^\tw(\mfo_5)$, resp.~$X^\tw(\mfsp_4)$, by Corollary \ref{C:RE[m]}, so all computations take place in rank two. Relation \eqref{B-aux-ff1} follows from the components of \eqref{RE} for the entries of the first column of $S^{[n-2]}(u)$, with the first row replaced as in \eqref{si1-sim}; the argument $\frac12-u$ then arises from \eqref{f-e:3} and \eqref{hf1}. The other relations are obtained by applying the derivations $[f^{(0)}_{n+1,n},\,\cdot\,]$ and $[f^{(0)}_{n,n-1},\,\cdot\,]$, which by \eqref{RE} act on $S^{[n-2]}(u)$ through constant matrices, to relations established earlier. Where the commutator being computed reappears at a second spectral parameter, a suitable specialisation of $v$ determines it and produces the shifted arguments.

\begin{lemma} \label{L:B-aux}
In the algebra $X^{\tw}(\mfo_{2n+1})$ we have:
\ali{
& [f_{n,n-1}(u),f_{n+1,n-1}(v)] \nn \\*
& \qq = \frac{(f_{n+1,n-1}(v)-f_{n+1,n-1}(u))(f_{n,n-1}(v)-f_{n,n-1}(u))}{v-u} \nn\\
& \qq\qu + \frac{f_{n+1,n}(v)\,d_{n-1}^{-1}(v)\,d_n(v)-f_{n+1,n}(\frac12-u)\,d_{n-1}^{-1}(u)\,d_n(u)}{v+u-\frac12} , \label{B-aux-ff1} \\[0.5em]
& [f_{n+1,n}(u),f_{n+1,n-1}(v)] \nn \\*
& \qq = f_{n+1,n}(u) \bigg( \frac{f_{n+1,n}(u)\,f_{n,n-1}(u) - f_{n+1,n-1}(u)}{u-v} \nn\\ 
& \hspace{3cm} - \frac{\frac12\,f_{n+1,n}(u)\,f_{n,n-1}(\frac12+u) - f_{n+1,n-1}(\frac12+u)}{u-v+\frac12} \nn\\
& \hspace{3cm} + \frac{u-v+1}{2\,(u-v)(u-v+\frac12)}\bigg( \frac{f_{n+1,n-1}(v)}{u-v+1}-f_{n+1,n}(u)\,f_{n,n-1}(v)\bigg)\bigg) \nn\\
& \qq\qu +\frac{f_{n+2,n-1}(\frac12+u)-f_{n+2,n-1}(v)}{u-v+\frac12} + \frac{f_{n,n-1}(\frac12-u)\,d_n^{-1}(u)\,d_{n+1}(u)}{u+v-\frac12} \nn\\
& \qq\qu - \frac{1}{(u-\frac14)(u-v+\frac12)} \bigg( \frac{u\,(u-v)}{u+v-\frac12}\,f_{n,n-1}(v) + \tfrac14\,f_{n,n-1}(\tfrac12+u) \bigg) d_{n}^{-1}(u)\,d_{n+1}(u) , \label{B-aux-ff3} 
\\[0.5em]
& [f_{n+1,n}(u),f_{n+2,n-1}(v)] \nn \\*
& \qq = -f_{n+2,n-1}(u)\,f_{n+1,n}(u)-f_{n+1,n}(u) \,\frac{f_{n+2,n-1}(v)-(v-u+\tfrac{1}{2})\,f_{n+2,n-1}(u+\tfrac12)}{v-u-\tfrac12} \nn\\
& \qq\qu -f_{n+1,n}(u) \bigg[ \tfrac12f_{n+1,n-1}(u)\,f_{n+1,n}(u) - \tfrac12 (f_{n+1,n}(u))^2 \times \nn\\ & \qq\qq\qu \times \bigg( \frac{v-u+1}{v-u}\,f_{n,n-1}(u) + \frac{f_{n,n-1}(v)}{2(v-u)(v-u-\tfrac12)} - \frac{v-u+\tfrac12}{v-u-\tfrac12}\,f_{n,n-1}(u+\tfrac12) \bigg) \nn\\
& \qq\qq + f_{n+1,n}(u)\bigg(\frac{(v-u+1)f_{n+1,n-1}(u)}{2(v-u)} - \frac{v-u+\tfrac12}{v-u-\tfrac12}\bigg( f_{n+1,n-1}(u+\tfrac12) - \frac{f_{n+1,n-1}(v)}{2(v-u)} \bigg)\!\bigg) \bigg] \nn \\
& \qq\qu + \frac{u}{u-\tfrac14}\bigg[\frac{(u+\frac34)(u-v)}{2(u+\tfrac12)}\, f_{n+1,n}(-u-1) \bigg(\frac{f_{n,n-1}(v)}{(u-v+1)(u+v+1)(u+v-\tfrac12)} \nn\\ & \qq\qq\qu -\frac{f_{n,n-1}(\tfrac12-u)}{3(u+\tfrac14)(u+v-\tfrac12)} - \frac{f_{n,n-1}(u+1)}{4(u+1)(u+\tfrac14)(u-v+1)}+\frac{f_{n,n-1}(-u-1)}{3(u+1)(u+v+1)} \bigg) \nn \\
& \qq\qq\qu + f_{n+1,n}(u) \bigg( \frac{(u+\tfrac14)f_{n,n-1}(u)}{4(u+\frac12)(u-\tfrac14)} - \frac{(u+v+\tfrac12)f_{n,n-1}(v)}{8\,(u+\frac12)(u-v+\tfrac12)(u+v-\tfrac12)(u-v+1)} \nn\\
& \qq\qq\qu + \frac{(u+\tfrac34)(u-v)f_{n,n-1}(u+1)}{4(u+\frac12)(u+\tfrac14)(u-v+1)} - \frac{(u-v-\tfrac12)f_{n,n-1}(u+\tfrac12)}{4u(u-v+\tfrac12)} \nn\\
& \qq\qq\qu -\frac{((64u^4-20u^2-4u+1)+2v(32u^3+16u^2-2u+1))f_{n,n-1}(\tfrac12-u)}{256u(u+\tfrac12)(u+\tfrac14)(u-\tfrac14)(u+v-\tfrac12)} \bigg) \nn
\\
& \qq\qq + \frac{(u+v+\tfrac12)(u-v+\tfrac12)\, f_{n+1,n-1}(v)}{(u+v+1)(u-v+1)(u+v-\tfrac12)} - \frac{(u+\tfrac14)\,f_{n+1,n-1}(u)}{4(u+\tfrac12)(u-\frac14)} -\frac{u\,(u-v)\,f_{n+1,n-1}(\tfrac12-u)}{3\,(u+\tfrac14)(u-\tfrac14)(u+v-\tfrac12)} \!\!\! \nn\\
& \qq\qq - \frac{(u+\frac34)(u-v)}{u+1} \bigg(\frac{f_{n+1,n-1}(u+1)}{4\,(u+\tfrac14)(u-v+1)} + \frac{f_{n+1,n-1}(-u-1)}{6\,(u+\tfrac12)(u+v+1)}\bigg) \bigg] d^{-1}_n(u) d_{n+1}(u) ,
\label{B-aux-ff4} \\[0.5em]
& d_{n}(u)\,f_{n+1,n-1}(v)\,d_{n}^{-1}(u) \nn\\*
& \qq = \frac{(u-v+1)(u-v-1)f_{n+1,n-1}(v)}{(u-v)^2} +\frac{f_{n+1,n-1}(u)}{(u-v)^2} - \frac{f_{n,n-1}(u)\,f_{n+1,n}(u)}{u-v} \nn\\ 
& \qq\qu + \bigg(\frac{(u+\frac{1}{4})(u+v-\frac{1}{2}) f_{n+1,n}(u)}{u-v} - (u-\tfrac{3}{4})f_{n+1,n}(\tfrac{1}{2}-u)\bigg)\,\frac{(u-v+1)f_{n,n-1}(v)}{(u-\frac{1}{4})(u+v+\frac{1}{2})(u-v)} 
 \nn\\ 
& \qq\qu + \bigg(\frac{(-16u^2+4u-4v+1)f_{n+1,n}(u)}{8\,(u-v)}
+(u-\tfrac34)f_{n+1,n}(\tfrac12-u)\bigg) \frac{f_{n,n-1}(u)}{2(u+\frac{1}{4})(u-\frac{1}{4})(u-v)} \nn\\ 
& \qq\qu +\frac{(u+\frac{3}{4})\big(\frac{1}{2}f_{n+1,n}(u) + (u-\frac{3}{4})f_{n+1,n}(\frac{1}{2}-u)\big) f_{n,n-1}(-u-\frac{1}{2})}{(u+\frac{1}{4})(u-\frac{1}{4})(u+v+\frac{1}{2})} , \label{B-aux-df2} 
\\[0.5em]
& d_{n+1}(u)\,f_{n+1,n-1}(v)\,d_{n+1}^{-1}(u) \nn\\*
& \qq = \frac{(u+v)(u+v+\frac{1}{2})(u-v+\frac{1}{2})(u-v-1)}{(u-v)(u+v-\frac{1}{2})(u-v-\frac{1}{2})(u+v+1)}\,f_{n+1,n-1}(v) \nn\\ 
& \qq\qu -\frac{u(u+\frac{1}{4})\big(f_{n+1,n}(u)\,(f_{n,n-1}(v)-f_{n,n-1}(u))+f_{n+1,n-1}(u)\big)}{(u+\frac{1}{2})(u-\frac{1}{4})(u-v)} \nn\\
& \qq\qu +\frac{u(u-\frac{1}{4})\big(f_{n+1,n}(u-\frac{1}{2})\,(f_{n,n-1}(v)-f_{n,n-1}(u-\frac{1}{2}))+f_{n+1,n-1}(u-\frac{1}{2})\big)}{(u-\frac{1}{2})(u+\frac{1}{4})(u-v-\frac{1}{2})} \nn\\ 
& \qq\qu +\frac{u(u+\frac{3}{4})\big(f_{n+1,n}(-u-1)\,(f_{n,n-1}(v)-f_{n,n-1}(-u-1))+f_{n+1,n-1}(-u-1)\big)}{3(u+\frac{1}{2})(u+\frac{1}{4})(u+v+1)} \nn\\ 
& \qq\qu -\frac{u(u-\frac{3}{4})\big(f_{n+1,n}(\frac{1}{2}-u)\,(f_{n,n-1}(v)-f_{n,n-1}(\frac{1}{2}-u))+f_{n+1,n-1}(\frac{1}{2}-u)\big)}{3(u-\frac{1}{2})(u-\frac{1}{4})(u+v-\frac{1}{2})} . \label{B-aux-df3} 
}
\end{lemma}

\begin{lemma} \label{L:C-aux}
In the algebra $X^{\tw}(\mfsp_{2n})$ we have:
\ali{
& [ f_{n,n-1}(u), f_{n+1,n-1}(v)] \nn\\*
& \qq = \frac{2f_{n+1,n-1}(u) ( f_{n,n-1}(v)-f_{n,n-1}(u))}{u-v} \nn\\
& \qq\qu - \frac{2f_{n+1,n-1}(v)f_{n,n-1}(u)+(u-v-2)f_{n+1,n-1}(v)f_{n,n-1}(v)}{(u-v)(u-v-1)} \nn\\
& \qq\qu - \frac{f_{n+1,n-1}(u-1)(f_{n,n-1}(u-1)-2f_{n,n-1}(u))+f_{n+2,n-1}(v)-f_{n+2,n-1}(u-1)}{u-v-1} \nn\\
& \qq\qu + \frac{(u-v-2)f_{n+1,n}(v)\,d^{-1}_{n-1}(v)\,d_n(v)}{(u-v-1)(u+v-2)} + \frac{f_{n+1,n}(u-1)\,d^{-1}_{n-1}(u-1)\,d_n(u-1)}{2(u-\frac{3}{2})(u-v-1)} \nn\\
& \qq\qu - \frac{(u-2)f_{n+1,n}(2-u)\,d^{-1}_{n-1}(u)d_n(u)}{(u-\frac{3}{2})(u+v-2)} , \label{C-aux-ff1}
\\[0.5em]
& [ f_{n+1,n}(u), f_{n+1,n-1}(v)] \nn\\*
& \qq = \frac{2f_{n+1,n}(u)\big(f_{n+1,n-1}(u)-f_{n+1,n-1}(v)-f_{n+1,n}(u)(f_{n,n-1}(u)-f_{n,n-1}(v))\big)}{v-u} \nn\\& \qq\qu - 2\bigg(\frac{(u-v)f_{n,n-1}(v)}{(u-v+2)(u+v-2)} + \frac{f_{n,n-1}(u+2)}{u\,(u-v+2)} - \frac{(u-1)f_{n,n-1}(2-u)}{u\,(u+v-2)} \bigg) d^{-1}_n(u)\, d_{n+1}(u) , \label{C-aux-ff2}
\\[0.5em]
& [ f_{n,n-1}(u), f_{n+2,n-1}(v)] \nn\\*
& \qq = \frac{2f_{n+1,n-1}(u)f_{n,n-1}(u)\big(f_{n,n-1}(u)-f_{n,n-1}(v)\big)}{u-v} \nn\\
& \qq\qu - \frac{\big(f_{n+1,n-1}(u-1)\big(f_{n,n-1}(u-1)-2f_{n,n-1}(u)\big)+f_{n+2,n-1}(v)-f_{n+2,n-1}(u-1)\big)}{u-v-1} \nn\\*
& \hspace{9.5cm} \times \big(f_{n,n-1}(v)-2f_{n,n-1}(u)\big) \nn\\
& \qq\qu -\frac{f_{n+1,n-1}(v)f_{n,n-1}(u)(2f_{n,n-1}(v)-(u-v+2)f_{n,n-1}(u))}{(u-v)(u-v-1)} \nn\\
& \qq\qu + \bigg( \frac{(u-v)f_{n+1,n-1}(v)}{(u-v-1)(u+v-2)}+\frac{2f_{n+1,n}(v)f_{n,n-1}(u)}{(u-v)(u+v-1)} + \frac{f_{n+1,n}(v)}{(v-\frac{1}{2})(u-v-1)(u+v-2)} \times \nn\\
& \qq\qq \times \bigg(\frac{2vf_{n,n-1}(1-v)}{u+v-1}-\frac{((v-\frac12)(u-v)^2+2-2v)f_{n,n-1}(v)}{(u-v)}\bigg)\!\bigg) d^{-1}_{n-1}(v)\,d_n(v) \nn\\
& \qq\qu + \bigg(\frac{((u - v)^2 (u + v - 3) - 2 (u + v - 2))f_{n+1,n-1}(v)}{(u-v)(u-v-1)(u+v-1)(u+v-2)} - \frac{f_{n+1,n-1}(u)}{(u-\frac{1}{2})(u-v)} \nn\\ 
& \qq\qq - \frac{2uf_{n+1,n-1}(1-u)}{(u-\frac{1}{2})(u+v-1)} - \frac{(u-2)f_{n+1,n}(2-u)}{u-\frac{3}{2}} \bigg( \frac{(u+v-3)(u-v+1)f_{n,n-1}(v)}{(u-v-1)(u+v-1)(u+v-2)} \nn\\
& \hspace{4.5cm} - \frac{2uf_{n,n-1}(1-u)}{(u-1)(u+v-1)} - \frac{2(v-1)f_{n,n-1}(u-1)}{(u-1)(u-v-1)(u+v-2)} \bigg)\! \bigg)d^{-1}_{n-1}(u)\,d_n(u) \nn\\
& \qq\qu + \frac{f_{n+1,n}(u-1)}{u-\frac{3}{2}}\bigg( \frac{(u+v-4)(u-v)f_{n,n-1}(v)}{2(u-v-1)(u-v-2)(u+v-2)} \nn\\
& \qq\qq -\frac{(u-3)f_{n,n-1}(u-2)}{(u-2)(u-v-2)} - \frac{(v-1)f_{n,n-1}(2-u)}{(u-2)(u-v-1)(u+v-2)} \bigg) d^{-1}_{n-1}(u-1)\,d_n(u-1) , \label{C-aux-ff3}
\\[0.5em]
& d_{n-1}(u)\,f_{n+1,n-1}(v)\,d_{n-1}^{-1}(u) \nn\\*
& \qq = \frac{(u+v-2)(u-v-1)f_{n+1,n-1}(v)}{(u-v)(u+v-3)} +\frac{(u-1)f_{n+1,n-1}(u)}{(u-\frac32)(u-v)} -\frac{(u-2)f_{n+1,n-1}(3-u)}{(u-\frac32)(u+v-3)} , \label{C-aux-df1}
\\[0.5em]
& d_{n}(u)\,f_{n+1,n-1}(v)\,d_{n}^{-1}(u) \nn\\*
& \qq = \frac{(u-v-2)(u-v+1)(u+v)f_{n+1,n-1}(v)}{(u-v)^2(u+v-1)} + \frac{(u (u - \frac12) (u - v + 1) + u^2 - \frac12 v)f_{n+1,n-1}(u)}{(u-\frac12)^2(u-v)^2} \nn\\ 
& \qq\qu -\frac{uf_{n,n-1}(u)f_{n+1,n}(u)}{(u-\frac12)(u-v)} - \frac{u(u-\frac32)f_{n+1,n-1}(1-u)}{(u-\frac12)^2(u+v-1)} + \frac{2(u-2)uf_{n+1,n}(2-u)f_{n,n-1}(1-u)}{(u-\frac12)(u-1)(u+v-1)} \nn\\ 
& \qq\qu + \frac{2u(u-v+1)(u+v-2)f_{n+1,n}(u)f_{n,n-1}(v)}{(u-1)(u-v)^2(u+v-1)} + \frac{u^2f_{n+1,n}(u)f_{n,n-1}(1-u)}{(u-\frac12)^2(u-1)(u+v-1)} \nn\\ 
& \qq\qu - \frac{(u + u^2(u-\frac32)(u-v+2))\,f_{n+1,n}(u)\,f_{n,n-1}(u)}{(u-\frac12)^2(u-1)(u-v)^2} \nn\\ 
& \qq\qu + \frac{(u-2)\,f_{n+1,n}(2-u)\,f_{n,n-1}(u)}{(u-\frac12)(u-1)(u-v)}-\frac{2(u-2)(u-v+1)\,f_{n+1,n}(2-u)\,f_{n,n-1}(v)}{(u-1)(u-v)(u+v-1)} . \label{C-aux-df3}
}
\end{lemma}


\section{Evaluated commutators} \label{app:Comms}

In this appendix we state the evaluated left hand sides of \eqref{B:S2} and \eqref{C:S1}. They are obtained by evaluating the nested commutators from the inside out and bringing each result to a normal form; the computation is lengthy but mechanical. In type BI, set $D(x):=d^{-1}_n(x)\,d_{n+1}(x)$. Besides \eqref{B:ff1}, \eqref{B:ff2}, \eqref{hf1}, \eqref{hf2}, \eqref{AI:df2}, \eqref{B:12} and \eqref{z(u)}, the computation uses \eqref{B-aux-ff3}, \eqref{B-aux-ff4}, \eqref{B-aux-df2} and \eqref{B-aux-df3}. The conjugation relations among these concern $d_n(x)$ and $d_{n+1}(x)$, and $D^{-1}(x)=\ms{z}^{-1}(x)\,d_{n+1}(-x)\,d_n(x)$ by \eqref{z(u)}, so the normal form has $D$ on the left. In type CI, set $D(x):=d^{-1}_{n-1}(x)\,d_n(x)$; the computation uses \eqref{AI:ff2}, \eqref{C:ff2}, \eqref{AI:df2}, \eqref{AI:dif1}, \eqref{C-aux-ff1} and \eqref{C-aux-ff3}, and $D$ is kept on the right. In each type the first bracket produces $f_{n+1,n-1}$, the second introduces $D$ and $f_{n+2,n-1}$, and the third uses \eqref{B-aux-ff4}, resp.~\eqref{C-aux-ff3}. In both types, products with repeated factors are reordered by solving small linear systems, and a conjugation by $D(x)$ that meets a series at the same spectral parameter $x$ is used in its limiting form, which is regular although the individual conjugation relations have poles there. After symmetrisation in $q,w,v$ all words without a diagonal factor cancel, and the remaining words agree with \eqref{BSR1}--\eqref{BSR2}, once the displayed factors $D$ are moved to the left, and with \eqref{CSR1}--\eqref{CSR2}. Expanding $D^\pm_v$ in the right hand sides of \eqref{B:S2} and \eqref{C:S1} and normal-ordering in the same way gives the same result.

Consider the algebra $X^\tw(\mfo_{2n+1})$. Then
\ali{
&\Sym_{q,w,v} [f_{n+1,n}(q),\, [f_{n+1,n}(w),\,[f_{n+1,n}(v),\,f_{n,n-1}(u)]]] \nn\\ 
& \qu = A(q,w,v)\,d_{n}^{-1}(v)\,d_{n+1}(v) + A(q,v,w)\, d_{n}^{-1}(w)\, d_{n+1}(w) + A(v,w,q)\, d_{n}^{-1}(q)\, d_{n+1}(q) \label{BSR1}
}
where $A(q,w,v)$ is
\ali{
& 4v\frac{v+\frac14}{v-\frac14} \, \bigg( \frac{(q(2-8u)-2u+4v+8v^2-1)\,f_{n+1,n}(q) f_{n,n-1}(u)}{(q-u)(2q-2v+1)(2u+2v-1)(2v+2w+1)(u-v-1)(q+v+1)} \nn\\
& \hspace{1.5cm} + \frac{\left(-1+4v+8v^2+2w-2u(1+4w)\right)f_{n+1,n}(w) f_{n,n-1}(u)}{(u-v-1)(1+2q+2v)(2u+2v-1)(2v-2w-1)(u-w)(1+v+w)} \nn\\
& \hspace{1.5cm} + \frac{(8q^2-4v(1+2v)+1)\left(f_{n+1,n}(q) f_{n,n-1}(q)-f_{n+1,n-1}(q)\right)}{(q-u)(1+2q-2v)(q-v-1)(1+q+v)(2q+2v-1)
(1+2v+2w)} \nn\\
& \hspace{1.5cm} - \frac{(4v+8v^2-8w^2-1)(f_{n+1,n}(w) f_{n,n-1}(w)-f_{n+1,n-1}(w))}{(1+2q+2v)(2v-2w-1)(1+v-w)(-u+w)(1+v+w)(2v+2w-1)} \nn\\
& \hspace{1.5cm} + \frac{(1+8u^2-4v(1+2v))(q(2u-2v-4w-1)-(2v+1)w+2u(2v+w+1))f_{n+1,n-1}(u)}{(q-u)(u-w)(2u-2v+1)(2q+2v+1)(2u+2v-1)(2v+2w+1)(u-v-1)(u+v+1)}\bigg) \nn\\
& + \frac{2v(3+4v)}{(v-u+1)(4v-1)}\bigg(\frac{f_{n+1,n}(w) f_{n,n-1}(1+v)}{\left(v+q+\frac{1}{2}\right)(v-w+1)(v+w+1)}+\frac{f_{n+1,n}(q) f_{n,n-1}(1+v)}{(v+w+\frac12)(v-q+1)(v+q+1)} \bigg) \nn\\
& + \frac{2v(3+4v)}{(1+v)(1+2q+2v)(-1+4v)(1+2v+2w)} \bigg(\frac{(2+q+6v+4v^2+w-4qw)f_{n+1,n-1}(1+v)}{(q-v-1)(1-u+v)(1+v-w)} \nn\\ & \hspace{6cm} +\frac{(2+6v+3w+4v(v+w)+q(3+4v+4w)) f_{n+1,n-1}(-1-v)}{(1+q+v)(1+u+v)(1+v+w)} \bigg) \nn\\
& - \frac{2v(3+4v)(2+4v^2+3w+v(6+4w)+q(3+4v+4w))\, f_{n+1,n}(-v-1) }{(u-v-1)(1+v)(1+q+v)(1+u+v)(1+2q+2v)(4v-1)(1+v+w)(1+2v+2w)} \nn\\
& \hspace{3cm} \times \Big( 2\,(v+1) f_{n,n-1}(u) + (u-v-1) f_{n,n-1}(-v-1) - (u+v+1) f_{n,n-1}(v+1) \Big) \nn\\
& - \frac{64v(1-4v^2+2w+q(2+4w)) f_{n+1,n}(v-\tfrac12) }{((1+2q-2v)(-1+2v)(1+2q+2v)(-1-2u+2v)(2u+2v-1)(2v-2w-1)(1+2v+2w))} \nn\\
& \hspace{3cm} \times \Big( (2v-1) f_{n,n-1}(u) + (u-v+\tfrac12) f_{n,n-1}(\tfrac12-v) - (u+v-\tfrac12) f_{n,n-1}(v-\tfrac12) \Big) \nn\\
& + \frac{2v}{(q+v-\frac12)(u+v-\frac12)(v+w+\frac12)} \bigg(\frac{f_{n+1,n}(q) f_{n,n-1}(\frac12-v)}{q-v+\frac12} + \frac{(v^2+vw+qv+qw-\frac14) f_{n+1,n-1}(\frac12-v)}{(v-\frac12)(v+q+\frac12)(v+w-\frac12)}\bigg) \nn\\
& - \frac{2v}{(q+v+\frac12)(v-w-\frac12)}\bigg(\frac{f_{n+1,n}(w) f_{n,n-1}(\frac12-v)}{(u+v-\frac12)(v+w-\frac12)}+\frac{(\frac12-2v^2+w+q(2w+1))\,f_{n+1,n-1}(v-\frac12)}{(2v-1)(q-v+\frac12)(v-u-\frac12)(v+w+\frac12)}\bigg) .
\label{BSR2}
}

\smallskip

\noindent Consider the algebra $X^\tw(\mfsp_{2n})$. Then
\ali{
&\Sym_{q,w,v} [f_{n,n-1}(q),\,[f_{n,n-1}(w),\,[f_{n,n-1}(v),\,f_{n+1,n}(u)]]] \nn\\ 
& \qu = 8\,(A(q,w,v)\,d_{n-1}^{-1}(v)\,d_{n}(v) + A(q,v,w)\, d_{n-1}^{-1}(w)\, d_{n}(w) + A(v,w,q)\, d_{n-1}^{-1}(q)\, d_{n}(q) ) \label{CSR1}
}
where $A(q,w,v)$ is
\ali{
& \frac{v\,(2 v (v+w-3)+q (2 v+2 w-3)-3 w+4) }{(2 v-1) (u-v) (q+v-1) (u+v-1)(v+w-1)(v+w-2)(q+v-2)} \nn\\
& \hspace{3.95cm} \times \big((u+v-1) f_{n+1,n}(v) f_{n,n-1}(1-v)-(2 v-1) f_{n+1,n}(u) f_{n,n-1}(1-v) \big) \nn\\
& + \frac{v\,(q+v(v-2)-w(q-1))\,(f_{n+1,n}(v) f_{n,n-1}(v)-f_{n+1,n-1}(v))}{(u-v) (v-q) (q+v-2) \left(v-\frac{1}{2}\right) (v-w) (v+w-2)} \nn\\
& + \frac{v}{(v+q-1)(v+w-2)}\left(\frac{f_{n+1,n}(v) f_{n,n-1}(q)}{(q-v) (u-v)}-\frac{f_{n+1,n-1}(1-v)}{(2v-1)(u+v-1)}\right) \nn\\
& + \frac{v}{(v+w-1)(v+q-2)} \bigg(\frac{f_{n+1,n}(v) f_{n,n-1}(w)}{(v-u) (v-w)} - \frac{f_{n+1,n-1}(1-v)}{(2 v-1)(u+v-1)}\bigg) \nn\\
& + \frac{2(v-1)(q(q-1)-v(v-2)) f_{n+1,n-1}(q)}{(q-u) (q-v) (q-v+1)(q+v-1)(q+v-2)(v+w-2)} + \frac{2(v-1)(w(w-1)-v(v-2)) f_{n+1,n-1}(w)}{(w-u)(w-v)(w-v+1)(w+v-1)(w+v-2)(v+q-2)} \nn\\
& + \frac{2(v-1)(v(v-2)-u(q-1)) f_{n+1,n}(u) f_{n,n-1}(q)}{(u-v)(q-u) (q-v+1) (q+v-1) (u+v-2) (v+w-2)} + \frac{2(v-1)(v(v-2)-u(w-1)) f_{n+1,n}(u) f_{n,n-1}(w)}{(u-v)(w-u) (w-v+1) (w+v-1) (u+v-2) (q+v-2)} \nn\\
& + \frac{2(v-1)(u(u-1)-v(v-2)) (q (u-v-2 w+2)-w(v-2)+u (2 v+w-4))}{(q-u) (u-v)(u-w)(u-v+1)(u+v-1)(q+v-2) (u+v-2) (v+w-2)} \, \big( f_{n+1,n}(u) f_{n,n-1}(u)-f_{n+1,n-1}(u) \big) \nn\\
& + \frac{(v-2)}{(q+v-2) (u+v-2) (v+w-2)} \times\nn\\
& \qq \times \bigg( f_{n+1,n}(2-v) \bigg( \frac{ f_{n,n-1}(q)}{q-v+1} + \frac{f_{n,n-1}(w)}{w-v+1} \bigg) + \frac{f_{n+1,n}(2-v) f_{n,n-1}(2-v)-f_{n+1,n-1}(2-v)}{v-\frac{3}{2}} \nn\\
& \qq\qq + \frac{q\,(2 w-1)-2v\,(v-3)-w-4}{(2v-3)(q-v+1)(v-u-1)(v-w-1)} \, \big((v+u-2) f_{n+1,n-1}(v-1) \nn\\ & \hspace{3.5cm} + (v-u-1) f_{n+1,n}(2-v) f_{n,n-1}(v-1)-(2 v-3) f_{n+1,n}(u) f_{n,n-1}(v-1)\big) \bigg) . \nn\\[-2em] \label{CSR2}
}


\setlength{\medmuskip}{2mu plus 1mu minus 1mu} 


\nc{\arxiv}[1]{\href{http://arxiv.org/abs/#1}{\tt arXiv:\nolinkurl{#1}}}
\nc\doi[1]{\href{https://dx.doi.org/#1}{doi:#1}}

\end{document}